\documentclass[11pt,reqno]{amsart}
\usepackage[margin=1in]{geometry} 
\usepackage{amsmath, amsthm, mathtools, amssymb, enumitem}
\usepackage{stmaryrd}

\usepackage[
    backend=bibtex,
    style=alphabetic,
    sorting=anyt,
    maxalphanames=99,
    maxbibnames=99,
    backref=true
]{biblatex}
\DeclareFieldFormat[article]{pages}{#1}
\renewbibmacro{in:}{}
\AtEveryBibitem{\clearfield{month}}
\AtEveryBibitem{\clearfield{day}}
\AtBeginBibliography{\small}

\DefineBibliographyStrings{english}{
    backrefpage = {cited on page},
    backrefpages = {cited on pages},
}

\usepackage[x11names]{xcolor}
\usepackage{orcidlink}

\usepackage{subcaption}

\usepackage{graphicx}
\usepackage{hyperref}
\hypersetup{
    colorlinks=true,
    urlcolor=blue,
    linkcolor=blue,
    citecolor=SpringGreen3
}

\newcommand{\al}{\alpha}
\newcommand{\SSP}{\mathscr{P}}

\newcommand{\Rect}{\mathsf{Rect}}
\newcommand{\Slope}{\mathsf{Slope}}
\newcommand{\Diag}{\mathsf{Diag}}
\newcommand{\Cyl}{\mathsf{Cyl}}
\DeclareMathOperator{\per}{perimeter}

\usepackage{mathrsfs}
\newcommand{\VV}{\mathscr{V}}
\newcommand{\II}{\mathscr{I}}
\newcommand{\RR}{\mathscr{R}}
\newcommand{\vv}{\mathbf{v}}
\newcommand{\up}{{\mathrm{up}}}
\newcommand{\down}{{\mathrm{down}}}

\renewcommand{\ss}{{\zeta}}
\newcommand{\sss}{{\ss}}
\newcommand{\rr}{{\rho}}
\newcommand{\JJ}{{\lambda}}
\newcommand{\MM}{{M}}
\newcommand{\tail}{{c}}

\newcommand{\LBRW}{L_{n}^{\mathrm{BRW}}}

\newcommand{\dmn}{\llbracket 0,n\rrbracket^2}
\newcommand{\dmp}{\llbracket0,n-1\rrbracket^2}
\newcommand{\BB}{\mathscr{B}}

\newcommand{\norm}[1]{\left\lVert#1\right\rVert}

\DeclarePairedDelimiter\ceil{\lceil}{\rceil}
\DeclarePairedDelimiter\floor{\lfloor}{\rfloor}

\newcommand{\lb}{\llbracket}
\newcommand{\rb}{\rrbracket}

\newcommand{\wh}{\widehat}
\newcommand{\wt}{\widetilde}

\newcommand{\R}{\mathbb{R}}  
\newcommand{\Z}{\mathbb{Z}}

\newcommand{\N}{\mathbb{N}}

\newcommand{\e}{\varepsilon}
\renewcommand{\d}{\delta}
\renewcommand{\o}{\omega}

\DeclareMathOperator{\polylog}{polylog}

\newcommand{\ls}{\lesssim}
\newcommand{\gs}{\gtrsim}

\newcommand{\E}{\mathbb{E}}
\renewcommand{\P}{\mathbb{P}}

\DeclareMathOperator{\Var}{Var}

\newcommand{\1}{\mathbf{1}}%\mathbbm{1}

\newcommand{\Ber}{\mathrm{Bernoulli}}
\newcommand{\Bin}{\mathrm{Binomial}}
\newcommand{\Poi}{\mathrm{Poisson}}

\newcommand{\msf}{\mathsf}

\newcommand{\mrm}{\mathrm}

\newcommand{\cE}{\mathcal{E}}
\newcommand{\cF}{\mathcal{F}}

\newcommand{\cL}{\mathcal{L}}

\newcommand{\sA}{\mathsf{A}}

\newcommand{\sC}{\mathsf{C}}

\newcommand{\sP}{\mathsf{P}}

\numberwithin{equation}{section}

\theoremstyle{plain}
\newtheorem{theorem}{Theorem}[section]

\newtheorem{proposition}[theorem]{Proposition}

\newtheorem{lemma}[theorem]{Lemma}

\newtheorem{maintheorem}{Theorem}

\theoremstyle{remark}
\newtheorem{definition}[theorem]{Definition}
\newtheorem{remark}[theorem]{Remark}

\newtheorem{claim}{Claim}

\title{Last passage percolation in hierarchical environments}

\author{Shirshendu Ganguly \and Victor Ginsburg \and Kyeongsik Nam}

\address{Shirshendu Ganguly\\
Department of Statistics,
University of California, Berkeley, CA, USA} 
\email{sganguly@berkeley.edu}

\address{Victor Ginsburg \orcidlink{0000-0001-9399-6748}\\
Department of Mathematics, University of California, Berkeley, CA, USA}
\email{victor@math.berkeley.edu}

\address{Kyeongsik Nam\\
Department of Mathematical Sciences, KAIST, South Korea}
\email{ksnam@kaist.ac.kr}
\begin{document}

\begin{abstract}
    Last passage percolation (LPP) is a model of a directed metric and a zero-temperature polymer where the main observable is a directed path (termed as the geodesic) evolving in a random environment accruing as energy the sum of the random weights along itself.
    When the environment distribution has light tails and a fast decay of correlation, the random fluctuations of LPP are predicted to be explained by the celebrated Kardar--Parisi--Zhang (KPZ) universality theory.
    However, the KPZ theory is not expected to apply for many natural environments, particularly “critical” ones. A characteristic feature of the latter is a hierarchical structure where each level in the hierarchy resembles white noise at a different scale, often leading to logarithmic correlations.

    In this article, we initiate a novel study of LPP in such hierarchical environments by investigating two particularly interesting examples.
    The first is an independent and identically distributed (i.i.d.) environment but with a power-law distribution with an inverse quadratic tail decay which is conjectured to be the critical point for the validity of the KPZ scaling relation.
    The second is the Branching Random Walk which is a hierarchical approximation of the two-dimensional Gaussian Free Field, the universal scaling limit of various two-dimensional critical models.
    The second example may be viewed as a directed, high-temperature (weak coupling) version of Liouville Quantum Gravity, a model of random geometry driven by the exponential of a logarithmically correlated field.
   
    Due to the underlying fractal structure, LPP in such environments is expected to exhibit logarithmic correction terms with novel critical exponents. 
    While some discussions about such critical models appear in the non-rigorous literature, precise predictions about exponents seem to be missing, in part due to the difficulty of simulating these models at large enough scales.
    
    Developing a framework based on multi-scale analysis, we obtain bounds on such exponents and prove almost optimal concentration results in all dimensions for both models, in the process elucidating similarities and distinctions between them.
    Finally, as a byproduct of our analysis we answer a long-standing question of Martin \cite{MarLinearGrowthGreedy2002} concerning necessary and sufficient conditions for the linear growth of the geodesic energy in i.i.d. random environments.
\end{abstract}

\maketitle

\setcounter{tocdepth}{1}
\tableofcontents

\section{Introduction}\label{intro}
One of the central aims of the last several decades of research in probability
theory and statistical physics
has been to understand interface growth in disordered media (e.g. the growth of
crystals, forest fires, or bacteria colonies).
It is widely predicted that the random fluctuations of such growing interfaces
do not depend on the microscopic details of any particular model, but rather 
stem from an interplay between just a few universal mechanisms. 
This motivates the study of \emph{universality classes} of growth models, where
examples in a given universality class share a common fluctuation theory. Models of random geometry form an important class of examples where boundaries of metric balls play the role of random interfaces.
This paper is focused on one such canonical model of random geometry known as 
\emph{last passage percolation} (LPP). 

\begin{definition}[Last passage percolation]\label{def:lpp}
    Let $(X(v))_{v\in \Z^2}$ be a collection of random variables indexed by the vertices of the lattice $\Z^2$.
    We refer to the individual variables $X(v)$ as \emph{weights}.
    For $n\ge1$, let $L_n$ be the \emph{last passage time} from $(0,0)$ to $(n,n)$:
\begin{equation}\label{LPPdef}
        L_n \coloneqq \max_{\pi} \sum_{v\in\pi} X(v),
\end{equation}
    where the maximum is taken over up-right directed lattice paths that start
    at $(0,0)$ and end at $(n,n)$.
    A \emph{geodesic} is a path achieving the above maximum.
    We write $\Gamma_n$ to denote such a maximizer (one may fix a deterministic 
    rule to break ties in the case there are multiple geodesics).
    
    Note that the definition of last passage time naturally extends to any pair of ordered points (in the coordinate-wise ordering) $u\preceq v \in \Z^2$ by defining $L(u;v)$ as in \eqref{LPPdef} but with the maximum over up-right directed paths from $u$ to $v$ (so $L_n=L((0,0);(n,n))$).
    Thus LPP can be thought of as a model of a random metric space where the last passage time is the distance between points 
    (strictly speaking, it is not a metric as it satisfies the reverse triangle inequality on account of the maximization in \eqref{LPPdef}).
    \footnote{The related model of \emph{first passage percolation} is where the weights are positive and one minimizes the passage time over all paths (including non-directed paths), thereby obtaining a genuine metric space.}
\end{definition}

\begin{figure}[hb]
    \centering
    \begin{subfigure}[c]{0.4\textwidth}
            \includegraphics[width=\linewidth]{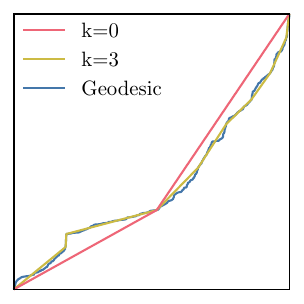}
    \end{subfigure}\hfill
    \begin{subfigure}[c]{0.55\textwidth}
            \includegraphics[width=\linewidth]{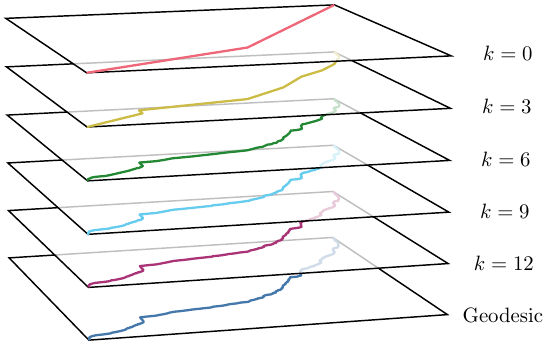}
    \end{subfigure}
    \caption{
        A simulation of the directed geodesic from $(0,0)$ to $(2^{15},2^{15})$
        in an environment of i.i.d. weights with power-law tails of exponent $2$.
        The left figure depicts the geodesic (blue) with two of its
        ``skeletons'' superimposed on top.
        The $k$\textsuperscript{th} skeleton is obtained from the geodesic by linearly
        interpolating between the weights that exceed $2^{15-k}$. 
        The right figure depicts the same situation as the left, but with more
        skeletons and with each path plotted on its own copy of the plane.
        (The skeletons corresponding to $k=6,9,12$ are visually quite similar to the geodesic, and their inclusion in the left figure would obscure the geodesic.)
    }\label{fig:skeletons}
\end{figure}

\noindent
Note that we have not yet specified the law or the correlation structure of the random field $(X(v))_{v\in \Z^2}$. While the focus in this article will be primarily on a class of fields exhibiting certain hierarchical multi-scale structures, we begin by reviewing how the behavior is expected to depend on the law of the field.

\subsection*{Kardar--Parisi--Zhang (KPZ) universality class} 
When the weights are sufficiently light-tailed and exhibit a fast decay of correlation (the simplest case being when the variables are independent and identically distributed), it is predicted that LPP belongs to the KPZ universality class \cite{KPZDynamicScalingGrowing1986}, exhibiting characteristic fluctuation behavior
\begin{align*}
    L_n-\E[L_n] \approx n^\chi,
    \qquad
    \max_{(x,y)\in\Gamma_n} |x-y|\approx n^\xi,
\end{align*}
with the universal exponents $\chi=1/3$ and $\xi=2/3$ satisfying the \emph{KPZ scaling relation}:
\begin{equation*}
    \chi=2\xi-1.
\end{equation*}
The quantity $\max_{(x,y)\in\Gamma_n} |x-y|$ measures the transversal
fluctuations of the LPP geodesic away from the Euclidean geodesic, the latter
being simply the diagonal line segment with endpoints $(0,0),(n,n)$.
Without delving into details, let us simply mention that the KPZ relation $\chi=2\xi-1$ stems from a characteristic competition between the fluctuation of the passage times
and the curvature exhibited by the \emph{limit shape}, i.e., the function 
$g$ on $(0,\infty)^2$
defined by 
\begin{align}\label{001}
    g(x,y)\coloneqq\limsup_{n\to\infty} \frac{L((0,0);(\floor{nx}, \floor{ny}))}{n},
\end{align}
where $\floor{\cdot}$ is the floor function.
By Kingman's subadditive ergodic theorem \cite{KinErgodicTheorySubadditive1968}, $g$ is a deterministic function which can be shown to be finite and concave for sufficiently light-tailed weights (e.g. \cite{MarLimitingShapeDirected2004}).
Thus by such concavity considerations, a transversal fluctuation of $n^{\xi}$ of the geodesic entails a loss of $n^{2\xi-1}$ in the last passage time which must compete with the $n^{\chi}$ order of fluctuations.
In \cite{ChaUniversalRelationScaling2013} Chatterjee established the KPZ relation in the setting of first passage percolation under strong assumptions on the existence of the exponents.
A more recent paper \cite{BSSRotationallyInvariantFirst2023} proves an unconditional related statement for rotationally invariant models.
We encourage the reader to consult the surveys
\cite{QuaIntroductionKPZ2011,
BGLecturesIntegrableProbability2015,
CorKardarParisiZhangUniversality2016,
ADH50YearsFirstpassage2017,
ComDirectedPolymersRandom2017,
GanRandomMetricGeometries2022,
ZygDirectedPolymersRandom2024}
and the references therein for more background on
the many facets of KPZ universality.
\\

\noindent
Staying with i.i.d. environments, the behavior is predicted to change as the variables become more heavy-tailed with power-law tails, as occasional large-valued weights are expected to be more pivotal in determining the last passage time.
That the heavy-tailed weights exhibit non-KPZ behavior  was first observed by Zhang \cite{ZhaGrowthAnomalyIts1990} and subsequently studied in several physics papers \cite{KruKineticRougheningExceptional1991,AFScalingSurfaceFluctuations1991,BKWSurfaceGrowthPowerLaw1991,LSSurfaceGrowthPowerlaw1992,BBPExtremeValueProblems2007,GLBRGroundstateStatisticsDirected2015} (see also the references therein).
The ``Flory argument'' heuristic allows one to derive fluctuation exponents for such heavy-tailed LPP.
To describe this we assume there exist constants $C,\al>0$ such that
\footnote{Here and in the sequel we use $\sim$ to denote asymptotic equivalence as $t\to\infty$, i.e. $A\sim B$ if $\lim_{t\to\infty}\frac{A}{B}=1$.}
$\P(X(v)>t)\sim Ct^{-\al}$.
As before, we postulate the existence of exponents $\chi=\chi(\al)$ and $\xi=\xi(\al)$ such that the random fluctuations of $L_n$ are of order $n^\chi$, and such that the maximum geodesic transversal fluctuation is of order $n^\xi$.
By definition of $\xi$, the geodesic $\Gamma_n$ typically lies within a cylinder $K$ of radius $\approx n^\xi$ centered around the diagonal $y=x$.
The Flory heuristic predicts that the largest-valued weights inside $K$ dictate the fluctuations and thus these large weights must compete with the parabolic loss incurred by the transversal fluctuations.
It is straightforward to see that (the interested reader might also want to refer to \cite{LLRExtremesRelatedProperties1983})
\begin{equation*}
    \max_{v\in K}X(v)\approx |K|^{\frac{1}{\al}} \approx n^{\frac{1+\xi}{\al}}.
\end{equation*}
This suggests that $\chi=\frac{1+\xi}{\al}$.
On the other hand, a transversal fluctuation of order $n^\xi$ incurs a parabolic loss of order $n^{2\xi -1}$,
leading to the relation
\[
    \frac{1+\xi}{\al}=2\xi-1.
\]
Rearranging, we obtain 
\begin{align}\label{eq:alpha-exponents-1}
    \chi = \frac{3}{2\al-1},\qquad\quad
    \xi = \frac{1+\al}{2\al-1}.
\end{align}
Setting $\al=5$  recovers the KPZ exponents $\chi=1/3$ and $\xi=2/3$.
While a rigorous analysis of any LPP model in this regime remains wide open,
on the basis of the above reasoning and substantial numerical evidence it is predicted in the physics literature that LPP belongs to the KPZ universality class for $\al>5$.
For $\al\in(2,5)$ the exponents $\chi,\xi$ in \eqref{eq:alpha-exponents-1} differ from their KPZ counterparts, with each $\al\in(2,5)$ corresponding to a distinct universality class.
Further, it is known that the limit shape exists if and only if $\al>2$ and hence the above heuristic is only valid in this regime (see Section \ref{intro:finite-second-moment} for an in-depth discussion of the finiteness of the limit shape which will also play an important role in this paper). 
For $\al \in (0,2)$, the largest weights dominate not only the fluctuations but also the leading order of the last passage time.
In this case the geodesic has a transversal fluctuation of order $n$ with last passage time of order $n^{2/\al}$ and hence
\begin{equation*}
    \chi = \frac{2}{\al},
    \qquad\quad \xi = 1.
\end{equation*}
This regime was treated rigorously by Hambly and Martin in their pioneering work \cite{HMHeavyTailsLastpassage2007}, and in fact they showed that the rescaled last passage time
$n^{-2/\al} L_n$
converges in distribution to an explicit non-degenerate random variable, and similarly for the rescaled geodesic $n^{-1}\Gamma_n$.
We encourage the reader to consult \cite{HMHeavyTailsLastpassage2007} for precise statements and for many other interesting results.

Turning next to the transition locations $\al=2$ and $\al=5$, 
it has been widely observed that ``critical'' statistical physics models are scale-invariant and exhibit fractal behavior, i.e. they look the same at all scales.
This manifests in extra correction terms, often logarithmic in nature, in the behavior of various observables.
In physics parlance, a change in scaling relations often coincides with the marginality of an associated renormalization group flow, leading to subtle changes in the universality behavior.
For example, consider the critical $d$-dimensional Ising model, whose observables exhibit power-law scaling behavior governed by universal critical exponents that depend on $d\ge 1$.
However, in the upper critical dimension $d=4$ (i.e. the model exhibits mean field behavior for $d>4$ and non-mean field behavior for $d<4$), the scaling is not purely power-law but rather carries additional polylogarithmic factors \cite{ADMarginalTrivialityScaling2021}.
In the context of LPP, since $\al=2$ and $\al=5$ mark transitions in the mechanisms governing fluctuations ($\al=2$ is the boundary for the validity of the KPZ scaling relation, and $\al=5$ is the boundary for the validity of the Flory heuristic that the largest weights dictate fluctuations), one may expect the LPP model in these regimes to exhibit power-law scaling with additional polylogarithmic factors.
Another recent line of work investigating similar random optimization problems where logarithmic factors appear in certain ``critical'' settings is \cite{DEHPMinimalSurfacesRandom2025,DEPMinimalSurfacesStrongly2025,OPWMinimizingCurvesBrownian2026}.
Whether or not the findings of this paper and the above articles can be connected to the broader renormalization group story remains a very interesting question.
To the best of our knowledge, predictions about the critical behavior and associated exponents in LPP are missing even from the non-rigorous literature (however see \cite{LSSurfaceGrowthPowerlaw1992,LSFractalsSurfaceGrowth1992,LSExactScalingSurface1993} for studies of related heavy-tailed growth models at criticality).
Note that the Flory argument presented above is only suited for predicting the exponents governing the leading-order polynomial fluctuations and is too coarse to probe polylogarithmic corrections at the critical points.
Similarly, although the predicted fluctuation exponents have been numerically verified in all regimes of $\al$
(see \cite{GLBRGroundstateStatisticsDirected2015} and the references therein),
a numerical study focusing on polylogarithmic corrections at criticality would be difficult given current computational resources.
\\

\noindent
In this article we build a framework to initiate the study of LPP driven by noise fields that exhibit a hierarchical multi-scale structure characteristic of critical models. We will soon see how the $\al=2$ heavy-tailed model falls in this class. Another well-known class of random fields exhibiting such a hierarchical structure is that of logarithmically correlated Gaussian fields, which includes, for instance, the two-dimensional Gaussian free field (a more elaborate discussion on such objects is presented in the next subsection and Section \ref{sec:brw0}). Part of this program is inspired by the recent explosion of activity in the study of models of Liouville quantum gravity (LQG) and pre-limiting random planar maps.
The main model in such settings is a first passage percolation (FPP) driven by the exponential of the Gaussian free field
\cite{GwyRandomSurfacesLiouville2020,DDGIntroductionLiouvilleQuantum2023,BPGaussianFreeField2025} 
which is expected to model a ``random'' Riemannian surface by quantum field theoretic considerations.
In this vein, LPP on a logarithmically correlated field may be viewed as a \emph{directed, high-temperature (weak coupling)} version of LQG (by linearizing the exponential).
While both FPP and LPP in i.i.d. light-tailed environments are expected to belong to the KPZ universality class, they may differ significantly when the underlying noise is correlated.
This is partly because in the i.i.d. light-tailed setting, even though FPP geodesics are not constrained to be directed, they still proceed roughly linearly, as winding around causes path lengths to typically increase (see \cite[Proposition 5.8]{KesAspectsFirstPassage1986}).
In contrast, for correlated environments, owing to clustering behavior of atypical weights, the non-directed geodesic may wander a lot in search of favorable regions causing its fractal dimension to be strictly greater than $1$. 
For instance, such a result for the non-directed LQG geodesic appears in \cite[Theorem 1.1]{FGRoughnessGeodesicsLiouville2024},
whereas in directed LQG (at least in the prelimit) the geodesic is forced to move in a Lipschitz fashion due to directedness.
This discrepancy also manifests in the related model of fractal percolation on the plane, where it is known that there is a regime where standard percolation occurs, but directed percolation never occurs (see \cite{CCDConnectivityPropertiesMandelbrots1988,ChaAbsenceDirectedFractal1995}).
\\

We next present a more detailed discussion on the structure of the random noise fields that will be considered in this work before stating the main results.

\subsubsection{Multi-scale environments}\label{intro:critical-points}
The following heuristic calculation gives some indication of self-similarity in the $\al=2$ heavy-tailed LPP model.
Suppose $n=2^m$ for some $m\in\N$, and 
partition $[0,n]^2\cap\Z^2$ into $2^k\times 2^k$ boxes for some $k\le m$. 
In each such box, the largest weights are typically of order $2^k$ and there are only $O(1)$ many of them.
Thus the entire noise field may be viewed approximately as a superimposition of $\log_2 n$ many independent Poisson point processes, one per dyadic scale, where the $k\textsuperscript{th}$ process has intensity $\frac{1}{2^{2k}}$ and each of its points has weight $2^{k}$.
The expert reader might at this point already see a resemblance with the Gaussian free field (GFF).
For the purposes of this discussion we will just consider the branching random walk (BRW) which may be a viewed as a coarsening of the GFF (a brief discussion about the GFF is presented in Section \ref{sec:brw0}).
In BRW the noise field is a sum of $\log n$ many independent fields where in the $k\textsuperscript{th}$ field, considering a decomposition into $2^k \times 2^k$ boxes as above, all vertices in a box are assigned the \emph{same} standard Gaussian value and the variables across boxes are independent.
Note that in the $\al=2$ heavy-tailed case, each $2^k\times 2^k$ box contains \emph{on average} a \emph{single} weight of size $2^k$, whereas in the BRW case \emph{every} vertex in a given $2^k\times 2^k$ box is associated to the \emph{same} variable of size $O(1)$.

It is not too difficult to see that, in either model, the last passage time \emph{per scale} is $O(n)$,  leading to an overall upper  bound of $O(n\log n)$ for the last passage time (a brief argument to this effect is presented in Section \ref{iop}).
Thus one of the main motivations of this paper is to understand the true polylogarithmic correction terms to the leading order of the last passage time in such hierarchical settings, as well as the fluctuation theory of the last passage time.
The geometric counterpart of this focuses on understanding whether the geodesic, restricted to the weights it collects from any given scale, looks similar to the geodesic for the LPP model that uses only weights of that scale (see Figure \ref{fig:skeletons}, for instance).
\begin{figure}[hb]
    \centering
    \begin{subfigure}[t]{0.45\textwidth}
        \centering
        \includegraphics[width=\linewidth]{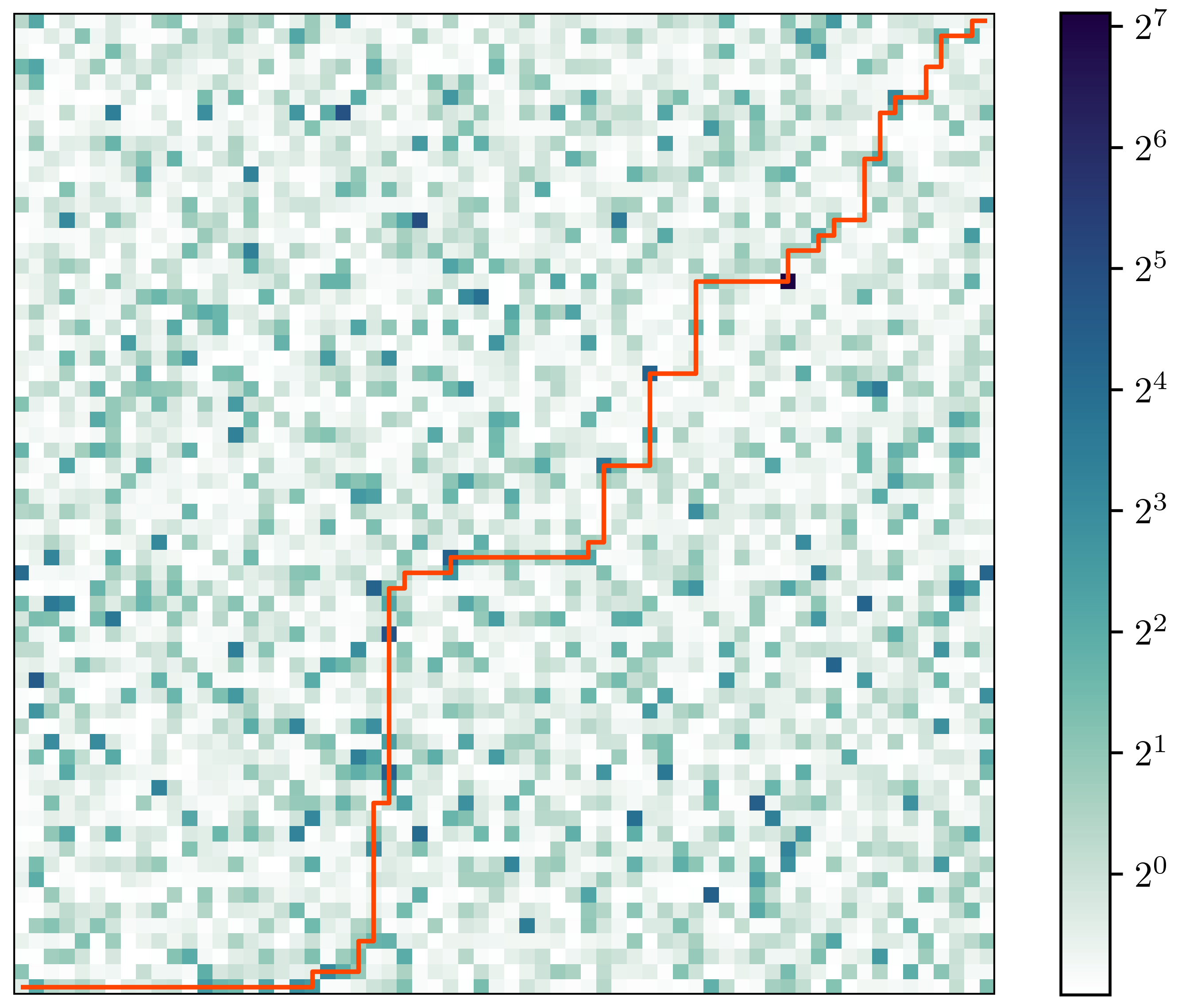}
        \caption{
            The geodesic (red) from $(0,0)$ to $(2^6, 2^6)$ is superimposed on a heatmap portraying the underlying i.i.d. weights.
        }
    \end{subfigure}
    \hfill
    \begin{subfigure}[t]{0.5\textwidth}
        \centering
        \includegraphics[width=\linewidth]{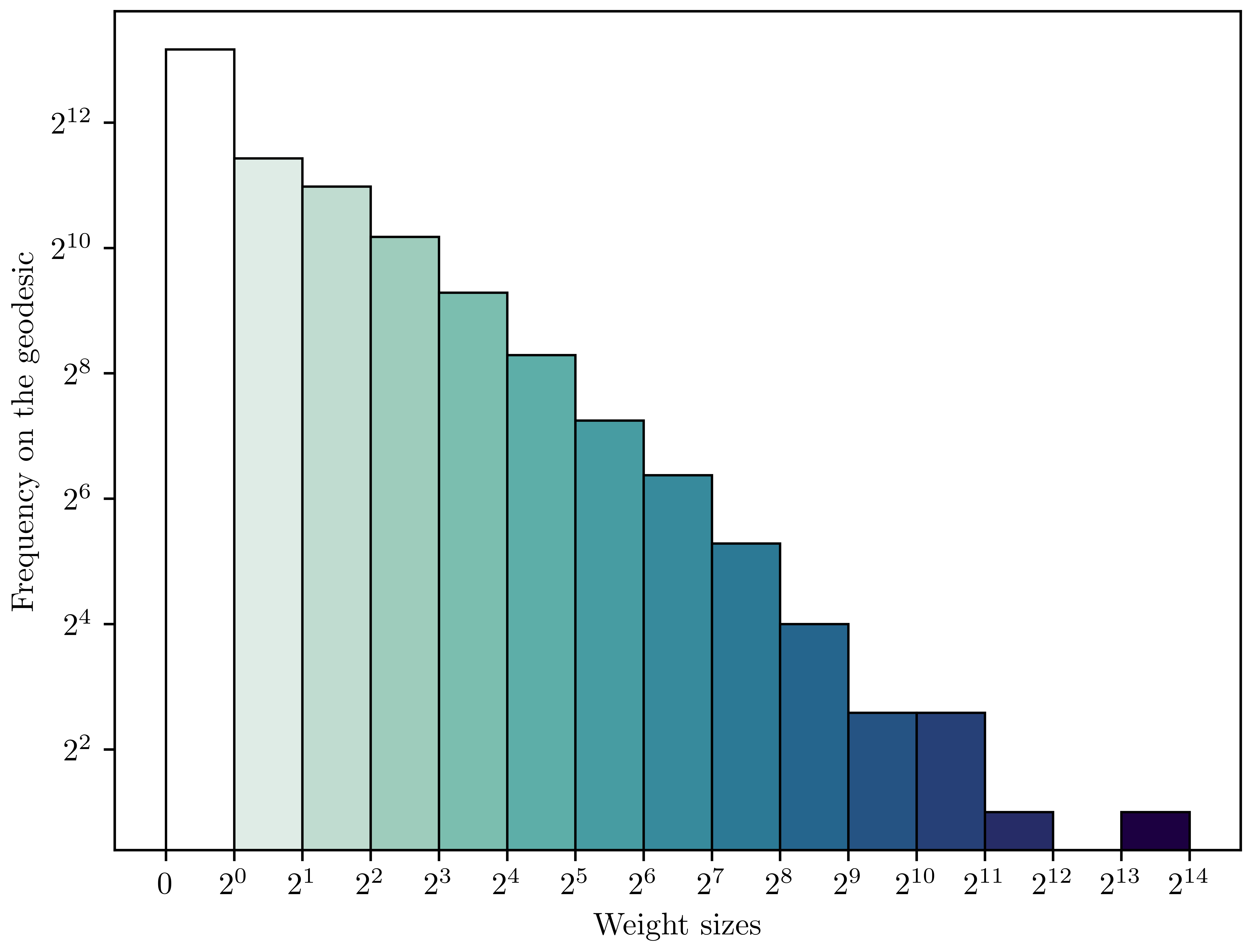}
        \caption{
            Log-log histogram of the weights on the geodesic from $(0,0)$ to $(2^{13}, 2^{13})$.
            The approximate linearity of the profile reflects the self-similarity of the critical LPP model.
        }
        \label{fig:iid-histogram}
    \end{subfigure}
    \caption{Simulations of the directed geodesic in an environment of i.i.d. weights with power-law tails of exponent $\al=2$.}
    \label{fig:heatmap-iid}
\end{figure}

We now proceed to our main results.
Our first result is an explicit lower bound for the logarithmic correction exponent for the last passage time in an environment of i.i.d. weights with power-law tails of exponent $\al=2$.
A subsequent result states an analogous bound for the last passage time on the branching random walk.
We will also present concentration results.
Further, a question of Martin about the dependence of the finiteness of the limit shape on the tails of the weights will be answered in the negative. 
Finally all our results have natural analogues in higher dimensions.

\subsection{Main results}
\begin{maintheorem}[Logarithmic correction lower bound for $\al=2$]\label{thm:main}
    Let $X$ be a non-negative random variable with a power-law tail of exponent $\al=2$, i.e. there exists $C>0$ such that as $t\to\infty$,
    \[
        \P(X>t)\sim Ct^{-2}.
    \]
    Consider LPP where the weights are i.i.d. copies of $X$.
    Then there exists $c>0$ such that for all sufficiently large $n$,
    \begin{align*}
        \P\left(
            L_n \ge c\,\frac{n(\log n)^{3/4}}{\log\log n}
        \right)
        \ge 1-e^{-(\log n)^{97}}.
    \end{align*}
\end{maintheorem}
The exponent $97$ appearing in Theorem \ref{thm:main} holds no particular significance; our proof implies a lower bound of $\smash{1-e^{-(\log n)^C}}$ for any large constant $C>0$, provided $c>0$ is modified appropriately.\\

As already indicated, an $O(n\log n)$ upper bound is not difficult to prove (a brief argument will be provided in Section \ref{iop}).  Thus  Theorem \ref{thm:main} leaves open the intriguing problem of identifying the true growth rate of $L_n$.\\ 

Despite the true growth rate remaining unknown, the next  result asserts that $L_n$ is concentrated.
\begin{maintheorem}[Concentration for $\al=2$]\label{thm:conc}
    In the setting of Theorem \ref{thm:main}, there exist $C,t_0>0$ such that for all $t\ge t_0$ and all $n$,
    \[
        \P\left(\left|\frac{L_n-\E[L_n]}{n}\right| > t\right) \le \frac{C}{t^{4/3}}.
    \]
    Note that by Theorem \ref{thm:main}, there exists $c>0$ such that
    \[
        \E[L_n] \ge c\,\frac{n(\log n)^{3/4}}{\log\log n}
    \]
    for all sufficiently large $n$ and hence $L_n$ is concentrated around its mean. 
\end{maintheorem}

A few remarks are in order.

\begin{remark}\label{rem:noconc}
    As mentioned above, \cite{HMHeavyTailsLastpassage2007} showed that for $\al\in(0,2)$, the rescaled last passage time $n^{-2/\al}L_n$ converges in distribution to a non-degenerate random variable in particular implying no concentration.
    This is a manifestation of the fact that for $\al\in(0,2)$, the fluctuation theory of the heavy-tailed LPP model is governed by the fluctuations of the $O(1)$ many largest weights.
    In contrast, Theorems \ref{thm:main} and \ref{thm:conc} indicate that at the critical point $\al=2$ the LPP fluctuation theory depends non-trivially on the fluctuations of weights of all scales, 
    resulting in a concentrated last passage time that grows strictly faster than the passage time in any given scale.
\end{remark}

\begin{remark}\label{lower1234}
    While the above result proves concentration at scale $n$, we expect that the true fluctuation scale is in fact smaller than $n$ by a polylogarithmic factor.
    A heuristic explanation for this is that at the critical point $\al=2$, we expect the geodesic transversal fluctuations to be smaller than $n$ by a polylogarithmic factor.
    Then, assuming the validity of the reasoning around \eqref{eq:alpha-exponents-1}, the last passage time fluctuations should be dictated by the largest weight in a strip around the diagonal of width $\frac{n}{\polylog n}$ which is typically $\frac{n}{\polylog n}$.
    In fact, later in Section \ref{sec:conc} we state and prove fluctuation lower bounds consistent with this (see Proposition \ref{prop:fluc-lower-bound} and the discussion preceding it).
    Namely, the fluctuation is at least $n$ with probability at least $\frac{1}{\polylog n}$.
\end{remark}

We next present our counterpart results for LPP on the branching random walk.

\subsubsection{Last passage percolation driven by logarithmically correlated Gaussian fields}\label{sec:brw0}
Recall that the 2D discrete Gaussian free field (GFF) on $\lb{0,n}\rb^2$, say with Dirichlet boundary conditions, is a random two-dimensional generalization of Brownian motion with the covariance structure given by the Green's function of the 2D simple random walk killed on hitting the boundary.
\footnote{For real numbers $x\le y$ we write $\lb x,y\rb\coloneqq[x,y]\cap\Z$.}
The Green's function in two dimensions decays logarithmically, making the GFF 
logarithmically correlated
(see e.g. \cite{BPGaussianFreeField2025} for background on the GFF). 
Beyond being a canonical directed model of geometry in a log-correlated environment, the model of LPP on the GFF is expected to be the scaling limit of natural optimization problems over paths on random surfaces arising from dimer models, lozenge tilings, etc., whose fluctuations are known to converge to the GFF
(see e.g. \cite{KenDominosGaussianFree2001,KenHeightFluctuationsHoneycomb2008,PetAsymptoticsUniformlyRandom2015},
and see \cite{GorLecturesRandomLozenge2021} for a much more comprehensive overview of the topic and the relevant literature).
Further, another intriguing perspective arises in the context of the Directed Landscape.
The latter is a universal scaling limit of LPP models in the KPZ universality class constructed in \cite{DOVDirectedLandscape2022}.
A key ingredient in its construction is a model of LPP driven by the parabolic Airy line ensemble \cite{CHBrownianGibbsProperty2014}.
The latter is the edge scaling limit of Dyson Brownian motion which describes the evolution of the eigenvalues of a random Gaussian Hermitian matrix as its entries perform independent Brownian motions.
While the edge of Dyson Brownian motion, the part that is relevant for the Directed Landscape, has Tracy--Widom fluctuations \cite{TWLevelspacingDistributionsAiry1994}, it is known that the bulk eigenvalues exhibit a log-correlated structure with fluctuations converging to the GFF (see e.g. \cite{BorCLTSpectraSubmatrices2010a}).
Thus LPP on the GFF admits possible connections to a putative bulk version of the Directed Landscape. (This was mentioned in a talk by Amol Aggarwal at the ``KPZ meets KPZ'' workshop at the Fields Institute in March 2024.)

While the GFF is a canonical example of a planar log-correlated field, in this article, for simplicity, we will restrict our attention to the simpler Gaussian branching random walk (BRW).
The BRW already captures many of the properties of the GFF, including, most importantly, a version of log-correlation. Further, the BRW is exactly hierarchical, which will be clear from the precise definition presented in the next paragraph, making its analysis simpler.

To define the BRW we need the following notation.
Let $n=2^m$ for some $m\in\N$.
For $k\in\lb 0,\log_2n\rb $, we denote by $\BB^{(k)}$ the partition of $\dmp$ into boxes of size $\smash{\frac{n}{2^k}\times \frac{n}{2^k}}$.
We associate to each box $B\in\bigcup_{k=0}^{\log_2n}\BB^{(k)}$ an independent standard Gaussian random variable $\xi_B$.
The BRW is the centered Gaussian field $Y=(Y(v))_{v\in\dmp}$
given by 
\[
    Y(v) \coloneqq \sum_{k=0}^{\log_2n} \xi_{B^{(k)}(v)},
\]
where $B^{(k)}(v)$ denotes the unique box in $\BB^{(k)}$ containing $v$.
We denote by $\LBRW$ the last passage time from $(0,0)$ to $(n-1,n-1)$ with respect to $(Y(v))_{v\in\dmp}$.

\begin{maintheorem}[Logarithmic correction lower bound for BRW]\label{thm:brw}
    There exist $c,c'>0$ such that for all sufficiently large $n$, 
    \begin{align}\label{eq:brw-hp}
        \P\left(\LBRW \ge c\,\frac{n(\log n)^{1/2}}{\log\log n}\right)\ge 
        1-e^{-c'\frac{\log n}{(\log\log n)^2}}
    \end{align}
    and
    \begin{align}\label{eq:brw-mean}
        \E[\LBRW] \ge c\,\frac{n(\log n)^{1/2}}{\log\log n}.
    \end{align}
\end{maintheorem}

Note that interestingly the above bound is weaker than the one for the $\al=2$ heavy-tailed case (Theorem \ref{thm:main}). We will comment on this later.\\

The next result establishes concentration. Again, interestingly distinct from the heavy-tailed $\al=2$ case, here we obtain \emph{matching} upper and lower bounds. 
\begin{maintheorem}[Concentration and fluctuations for BRW]\label{thm:conc-brw}
    As $n\to\infty$,
    \begin{align}\label{eq:brw-variance}
        \Var(\LBRW) = \Theta(n^2).
    \end{align}
    Moreover, there exist $c, t_0>0$ such that for all $t\ge t_0$ and all $n$, 
    \begin{align}\label{eq:brw-fluctuations}
        e^{- t^2}\le
        \P\left(\left|\frac{\LBRW - \E[\LBRW]}{n}\right| > t\right)
        \le e^{-c t^2},
    \end{align}
    and hence the fluctuations of $\LBRW$ around its mean are of order $\Theta(n)$.
    Note that this combined with \eqref{eq:brw-mean} implies that $\LBRW$ is concentrated around its mean.
\end{maintheorem}

The above results raise several fundamental questions about the sharp logarithmic correction factors as well as fluctuation theories  for LPP on BRW, GFF and other log-correlated fields. It is also of interest to see if comparison arguments in the spirit of Kahane's convexity inequality (e.g. \cite[Theorem 3.11]{LTProbabilityBanachSpaces1991}) may allow one to transfer results about LPP for branching random walk (or modified versions thereof) to the GFF.  We leave a closer investigation of such questions to future work.\\

\begin{figure}[ht]
    \centering
    \begin{subfigure}[h]{0.45\textwidth}
        \centering
        \includegraphics[width=\linewidth]{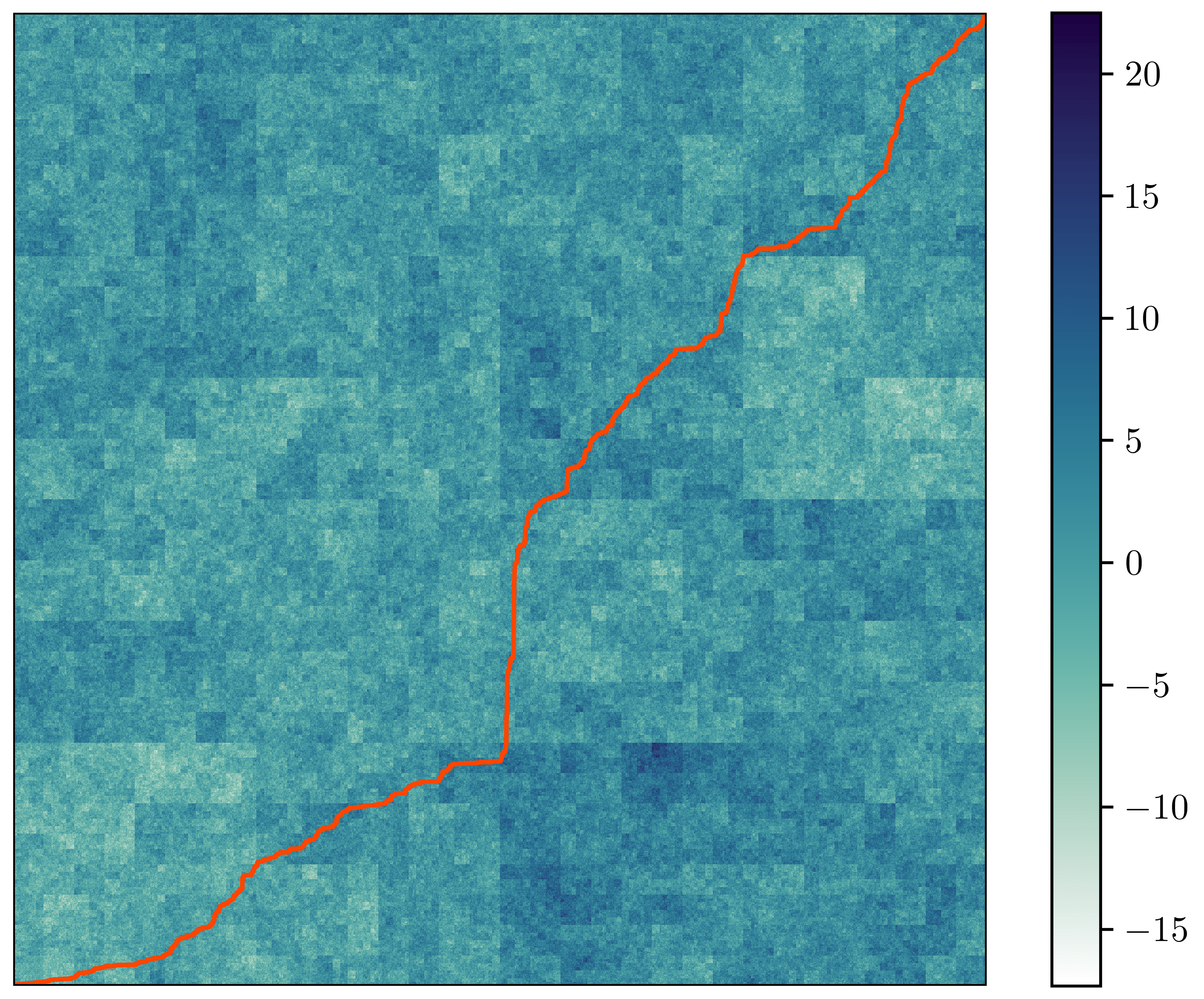}
        \caption{The geodesic (red) is superimposed on a heatmap portraying the         underlying branching random walk.}
    \end{subfigure}
    \hfill
    \begin{subfigure}[h]{0.5\textwidth}
        \centering
        \includegraphics[width=\linewidth]{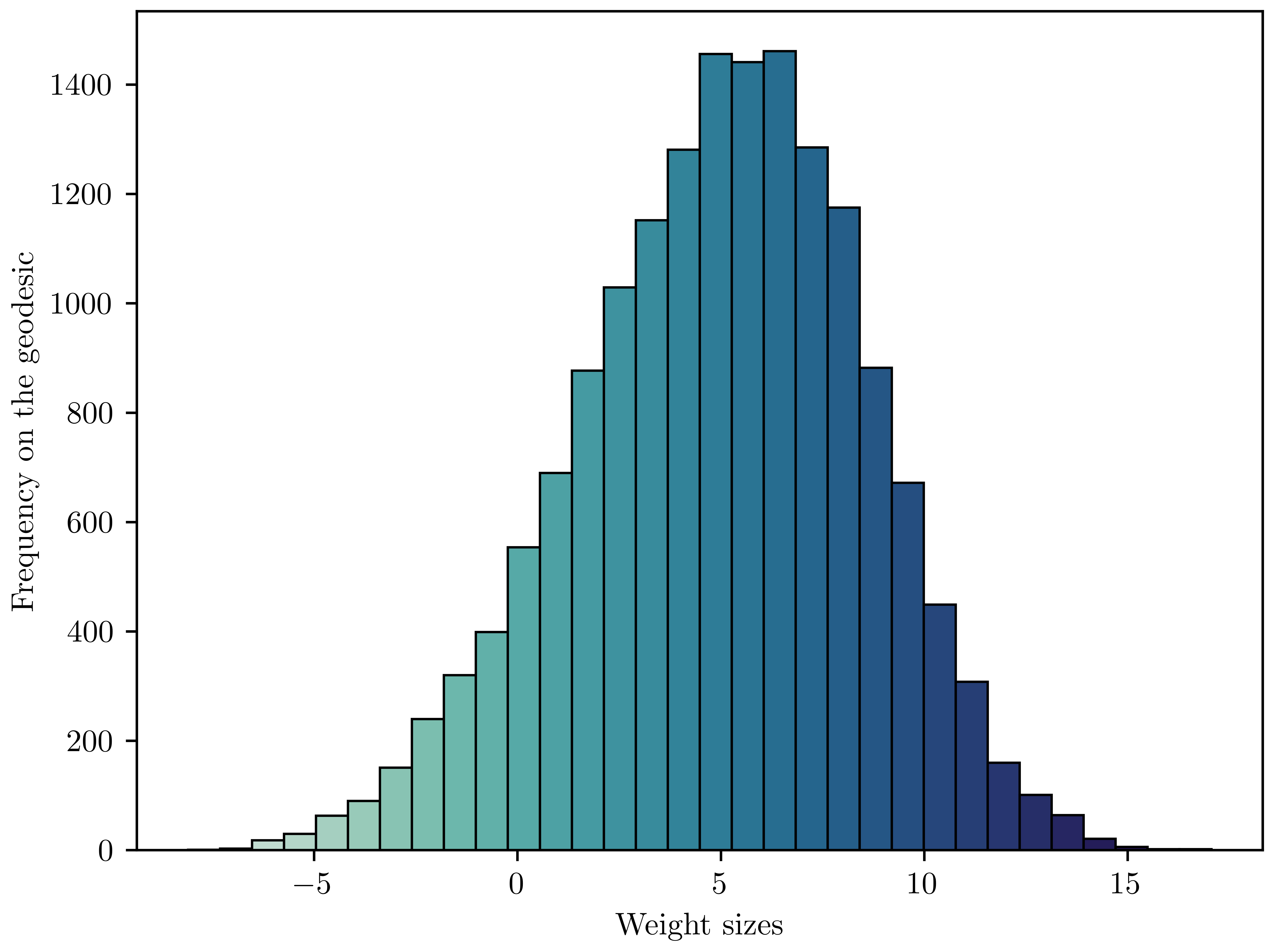}
        \caption{
            Histogram of the weights on the geodesic pictured on the left.
            Notice that the geodesic contains many negative weights.
        }
    \end{subfigure}
    \caption{
        A simulation of the directed geodesic from $(0,0)$ to
        $(2^{13}, 2^{13})$ on the branching random walk.
    }
    \label{fig:heatmap-brw}
\end{figure}

Our argument for Theorem \ref{thm:main} with minor modifications also allows us to construct a distribution with finite second moment and superlinear last passage time, thereby answering a question of Martin \cite{MarLinearGrowthGreedy2002}. Before stating the formal result we briefly review the pertinent literature on last passage limit shapes.

\subsubsection{A distribution with finite second moment and superlinear last passage time}\label{intro:finite-second-moment}
In the present discussion we only consider LPP models where the weights are i.i.d. copies of some non-negative random variable $X$. 
Recall that the \emph{limit shape} is the
deterministic function $g:(0,\infty)^2 \to [0,\infty]$ that measures the linear
growth rate of the last passage time (see \eqref{001} for a precise definition).
It is an important problem to identify the vertex weight distributions for which the corresponding LPP models satisfy $g<\infty$.
For instance, as explained in the discussion following \eqref{001}, the finiteness of $g$ is linked to the validity of the KPZ scaling relation $\chi=2\xi-1$.
This line of research began in the related setting of first passage percolation (FPP).
The problem of characterizing the weight distributions for which the FPP analogue of $g$ is finite 
was studied---and solved---in a number of groundbreaking works dating back several decades, see \cite{HWFirstPassagePercolationSubadditive1965,
    KinErgodicTheorySubadditive1968,
    WieWeakMomentConditions1980,
    CDLimitTheoremsPercolation1981,
    KesAspectsFirstPassage1986} and the references therein.
We will not discuss FPP further, but we encourage the interested reader to consult the excellent survey \cite{ADH50YearsFirstpassage2017}.

Unlike in FPP, 
there is no known characterization of the weight distributions for which the corresponding LPP model satisfies $g<\infty$.
However, significant progress towards a characterization was 
made in a series of works from the 1990s and 2000s.
The first result in this direction was obtained by 
Glynn and Whitt \cite{GWDeparturesManyQueues1991}, who showed that $g<\infty$ whenever $X$ has an exponential tail.
Subsequently, Cox, Gandolfi, Griffin, and Kesten \cite{CGGKGreedyLatticeAnimals1993}
(see also \cite{GKGreedyLatticeAnimals1994})
proved that $g<\infty$ under the weaker assumption that
\[
    \E\Bigl[X^2\max\{\log X,\, 0\}^{2+\epsilon}\Bigr] < \infty
    \qquad\text{for some } \epsilon>0.
\]
Finally, the state of the art was established by Martin \cite{MarLinearGrowthGreedy2002,MarLimitingShapeDirected2004}
(see also the survey \cite{MarLastpassagePercolationGeneral2006}),
who showed that $g<\infty$ under the assumption that
\begin{equation}\label{eq:martin_hypothesis}
    \int_0^\infty \P(X>t)^{1/2}\,\mrm{d}t < \infty.
\end{equation}
Note that \eqref{eq:martin_hypothesis} is only slightly stronger than $\E[X^2]<\infty$.
Moreover, \cite{CGGKGreedyLatticeAnimals1993}
showed that $g<\infty$ implies $\E[X^2]<\infty$, a result which, to our knowledge, 
is the strongest known necessary condition for $g<\infty$.
In view of the small discrepancy between the sufficient condition \eqref{eq:martin_hypothesis} and the necessary condition $\E[X^2]<\infty$, 
it is natural to ask if one of the two conditions is actually equivalent to $g<\infty$.
In fact, Martin asked the same question in \cite[Section 9]{MarLinearGrowthGreedy2002}.
While this question has not seen progress since \cite{MarLinearGrowthGreedy2002}, the next result, a quick corollary of the proof of Theorem \ref{thm:main}, shows that $\E[X^2]<\infty$ is not sufficient for $g<\infty$.

\begin{maintheorem}\label{thm:secondmoment}
    Let $X$ be a non-negative random variable such that $\P(X>t) \sim \frac{1}{t^2(\log t)^\beta}$ for some $\beta \in (1,\frac32)$.  
    Clearly $X$ has finite second moment.
    Moreover, for LPP where the weights are i.i.d. copies of $X$, 
    \[
        \frac{L_n}{n}\to \infty \quad\text{almost surely}.
    \]
\end{maintheorem}

\begin{remark} As will be apparent from the proof (see Section \ref{sec:secondmoment}), that the power of log in Theorem \ref{thm:main} is strictly greater than $1/2$  is crucial to deduce the above result. 
\end{remark}

Our methods and results also admit natural generalizations to LPP in higher dimensions.
Precise statements are deferred to  Section \ref{sec:higher-dimensions}.\\

We conclude this introduction by briefly reviewing the study of positive-temperature versions of LPP which naturally raises counterpart open problems in hierarchical settings.

\subsection*{Directed polymers in random environments}
In this model the weights are rescaled by an inverse temperature parameter $\beta_n\approx n^{-\gamma}$ for some $\gamma\ge0$.
The analogue of the last passage time is the \emph{free energy} which is the logarithm of the \emph{partition function}.
The latter is the average over all directed paths $\pi$ of its Gibbs weight $e^{\beta_n \sum_{v\in\pi}X(v)}$.
Thus choosing a small $\beta_n$ reduces the impact of the largest weights on the overall fluctuations of the polymer free energy.
In the seminal work \cite{AKQIntermediateDisorderRegime2014}, 
for i.i.d. weights with a finite exponential moment,
the authors established a scaling limit of the partition function, when $\beta_n = n^{-1/4}$, to the solution to the multiplicative  stochastic heat equation whose logarithm is the Cole--Hopf solution to the KPZ equation
(the authors also conjectured that the same result holds as long as the weights have finite sixth moment, which was later proved in \cite{DZHighTemperatureLimits2016}).
The corresponding continuum polymer measure was constructed in \cite{AKQContinuumDirectedRandom2014}.
Building a parallel theory for log-correlated or hierarchical fields seems to be a novel research direction that will be explored in forthcoming work of the first two authors \cite{GG}.
\\

Not unrelated to this paper, there have also been interesting developments in the study of directed polymers in heavy-tailed environments.
While we encourage the reader to
consult the recent survey \cite{ZygDirectedPolymersRandom2024} for a more 
comprehensive review, we list a few lines of work here. 
In \cite{DZHighTemperatureLimits2016}, Dey and Zygouras 
extended the zero-temperature Flory argument (presented earlier in this introduction) to derive conjectural fluctuation exponents 
$(\chi,\xi)$ for every pair $(\al,\gamma)$.
\cite{DZHighTemperatureLimits2016} also rigorously verified their predictions in
the regime where $\al>\frac12$ and where $\gamma$ is tuned so that the polymer 
is diffusive (i.e. $\xi=\frac12$).
Auffinger and Louidor \cite{ALDirectedPolymersRandom2011} treated the regime where $\al\in(0,2)$ and where $\gamma$ is tuned so that the polymer is fully delocalized (i.e. $\xi=1$), thus generalizing many of the results of \cite{HMHeavyTailsLastpassage2007} to positive temperature.
Berger and Torri \cite{BTDirectedPolymersHeavytail2019} later established the full phase diagram for $\al\in(0,2)$, and actually discovered more phase transitions depending on logarithmic corrections to the inverse temperature $\beta_n$ (see also \cite{BTWNondirectedPolymersHeavytail2022} for analogous results for non-directed polymers).
In a different direction, Berger and Lacoin \cite{BLScalingLimitDirected2021,BLContinuumDirectedPolymer2022} considered heavy-tailed environments for $\al\in(0,2)$, constructing continuum polymer models and proving convergence of pre-limiting discrete models (see also \cite{BCLStochasticHeatEquation2023} for a study of the multiplicative stochastic heat equation driven by L\'evy noise).

We could only find two comments in the rigorous literature regarding the heavy-tailed polymer with 
$\al=2$. The first is  \cite[Remark 2.3]{BLScalingLimitDirected2021} presenting a discussion on the subtlety of identifying an intermediate disorder scaling limit for $\al=2$. The second is \cite[Remark 1]{DZHighTemperatureLimits2016} where it was asserted that for $\al=2$ and $\xi=\frac12$, the polymer free energy is of order  {$\beta_n n$} and exhibits 
Gaussian fluctuations of order $n^{-1/2}$ for a particular choice of $\beta_n$ which up to a logarithmic factor scales like $n^{-3/4}$.
The growth rate of the free energy along with an application of Markov's inequality implies that under the polymer measure the energy of a typical path is $O(n)$, in contrast to the lower bound for the LPP geodesic in Theorem \ref{thm:main}.

\subsection*{Organization}
In Section \ref{iop} we outline our arguments.
In Section \ref{sec:preliminaries} we record notation and preliminary lemmas.
In Section \ref{sec:superlinear1} we prove Theorem \ref{thm:main}.
In Section \ref{sec:secondmoment} we prove Theorem \ref{thm:secondmoment}.
In Section \ref{sec:conc} we prove Theorem \ref{thm:conc} and a counterpart result in Proposition \ref{prop:fluc-lower-bound}.
In Section \ref{sec:brw} we prove Theorems \ref{thm:brw} and \ref{thm:conc-brw}.
In Section \ref{sec:higher-dimensions} we present the higher-dimensional analogues of our results.

\subsection*{Acknowledgements}
SG thanks James Martin for telling him about the question answered by Theorem \ref{thm:secondmoment} and Ewain Gwynne for discussions on the broad topic of models of directed metrics driven by the Gaussian free field. 
SG and VG also thank Kaihao Jing for helpful discussions.
We are grateful to the anonymous referees whose meticulous feedback improved the presentation.

SG was partially supported by NSF Career grant 1945172.
VG was supported by the NSF Graduate Research Fellowship Program under Grant No. DGE-2146752.
KN was supported by Samsung Science and Technology Foundation under Project Number SSTF-BA2202-02.
Part of this work was completed during KN's visit to UC Berkeley in the summer of 2024.
\\

In the next section we outline the key ideas underlying our arguments. 

\section{Proof ideas}\label{iop}
While Theorems \ref{thm:main} and \ref{thm:brw} are about lower bounds, we begin with a quick discussion on the straightforward upper bounds. 
For the BRW an $O(n\log n)$ upper bound is trivial since it is well known that the maximum of the BRW is $O(\log n)$ (see e.g. \cite{ZeiBranchingRandomWalks2016}), so we focus on the heavy-tailed $\al=2$ case. 
The heuristic in Section \ref{intro:critical-points} can be adapted to prove that $L_n = O(n\log n)$ with high probability. 
As mentioned, for i.i.d. $\al=2$ heavy-tailed weights, the largest weight in the box $\dmn$ is of order $n$.
It follows that with high probability the last passage time can be decomposed as:
\begin{align*}
    L_n &= \sum_{k=0}^{\Theta(\log n)}\sum_{v\in\Gamma_n}X(v)\1_{\frac{n}{2^{k}}<X(v)\le \frac{n}{2^{k-1}}}
    + \sum_{v\in\Gamma_n}X(v)\1_{0\le X(v) \le 1}\\
    &\le \sum_{k=0}^{\Theta(\log n)}\max_{\pi}\sum_{v\in\pi}X(v)\1_{\frac{n}{2^{k}}<X(v)\le \frac{n}{2^{k-1}}}
    + \max_\pi \sum_{v\in\pi}X(v)\1_{0\le X(v) \le 1},
\end{align*}
where all the maxima are over up-right directed lattice paths from $(0,0)$ to $(n,n)$.
We denote by $L^{(k)}_n$ the $k$\textsuperscript{th} summand above (including the term with weights $X(v)\1_{0\le X(v)\le 1}$).
By the independence of the weights, the distribution of $L^{(k)}_n$ is approximately that of $\smash{\frac{n}{2^k}}\cdot\cL_n^{(k)}$, where $\cL_n^{(k)}$ denotes the last passage time from $(0,0)$ to $(n,n)$ on a Poisson point process of intensity $\smash[b]{\frac{2^{2k}}{n^2}}$.
By the homogeneity of the Poisson point process, $\cL_n^{(k)}$ is equal in distribution to the last passage time from $(0,0)$ to $(2^k, 2^k)$ on a Poisson point process of intensity $1$ and hence the expectation of $\cL^{(k)}_n$ is $O(2^k)$.  
This added over $\log n$ many scales bounds the expectation of $L_n$ by $O(n\log n)$.
To get tail bounds one only has to work a bit harder.
It is well known (see e.g. \cite{TalNewLookIndependence1996,ADHammersleysInteractingParticle1995}) that there exist $C,t_0>0$ such that for all $n,k\ge 1$ and $t\ge t_0$,
\begin{equation*}
    \P\bigl(\cL^{(k)}_n > t 2^k\bigr)
    \le e^{-C t 2^k}
\end{equation*}
(see also \cite{DZIncreasingSubsequencesIID1999,SepLargeDeviationsIncreasing1998} for sharp large deviations estimates).
It follows by a union bound that
(we adopt the conventions that all logarithms have base $2$, and that $c,c',c''$ denote positive constants which may decrease from one line to the next)
\begin{equation}\label{upper1234}
    \begin{split}
        \P\left(
        \sum_{k=0}^{\Theta(\log n)}
        L^{(k)}_n > c n\log n
    \right)
    &\le \sum_{k=0}^{\ceil{ \log\log n}-1}
    \P\left(
        L^{(k)}_n > \frac{c' n \log n}{2^k}
    \right)
    + \sum_{k=\ceil{\log\log n}}^{\Theta(\log n)}
    \P\bigl(
        L^{(k)}_n > c''n
    \bigr)\\
    &\approx \sum_{k=0}^{\ceil{\log\log n}-1}
    \P\bigl(
        \cL^{(k)}_n > c'\log n
    \bigr)
    + \sum_{k=\ceil{\log\log n}}^{\Theta(\log n)}
    \P\bigl(
        \cL^{(k)}_n > c''2^k
    \bigr)\\
    &\le \ceil{\log\log n}e^{-c'\log n}
    + e^{-c''\log n},
    \end{split}
\end{equation}
and therefore $L_n = O(n\log n)$ with high probability.

We now move on to the main ideas in the proof of Theorem \ref{thm:main}. For simplicity, in the present discussion we will work with the proxy Poisson superimposition model already alluded to above. That is, there are $\log n$ many independent Poisson point processes, where the $k\textsuperscript{th}$ one, describing the points at scale $k$, i.e., having weight $\frac{n}{2^k}$, has intensity $\frac{2^{2k}}{n^2}$. 
Note that the typical separation between points at the $k\textsuperscript{th}$ {scale} is of order $\frac{n}{2^k}$.

Our basic approach is a multi-scale construction where we build a path with large weight (henceforth called a \emph{heavy path}) scale by scale.
That is, we will inductively specify the points from the $k\textsuperscript{th}$ scale through which the heavy path passes.
It is instructive to view the restriction of a path to points coming from scales $\le k$ as the ``skeleton'' of the path at scale $k$, see Figure \ref{fig:skeletons}.

To explain the basic idea it suffices to consider points of just two consecutive scales given by point processes $\omega$ and $\wt{\omega}$ on $[0,n]^2$. We will take $\omega$ to be a Poisson process of intensity $1$ with all its points having weight $1$ and $\wt{\omega}$ to be a Poisson process of intensity $4$ whose points each have weight $1/2$. We will thus think of the noise as having two levels, Level 1 and Level 2 corresponding to $\omega$ and $\wt{\omega}$ respectively.

It is well known that for a Poisson process of intensity $1$ on a rectangular box of side lengths $x$ and $y$, the last passage time between the bottom-left and top-right corners is comparable to $\sqrt{xy}$ (see e.g. \cite{ADHammersleysInteractingParticle1995}).
This leads to a tradeoff that we must address in our construction: attempting to maximize the number of points from $\omega$ causes the corresponding skeleton to zig-zag a lot, leading to a loss in the weight accumulated from $\wt \omega$.
To see this, consider a sequence {$\SSP$} of ordered points (also viewed as a directed path)
\begin{equation}\label{skeleton1}
    (0,0)=(x_0,y_0) \preceq (x_1,y_1) \preceq \ldots \preceq (x_m,y_m) \preceq (x_{m+1},y_{m+1})=(n,n),
\end{equation}
where $(x_1,y_1),\dots,(x_m,y_m)\in\omega$.
Thus {$\mathscr{P}$} is a candidate for the skeleton at Level 1 of the heavy path.
Next suppose we want to add points in $\wt\omega$ to {$\SSP$}.
To maintain directedness, these added points must satisfy
\[
    (x_i,y_i)\preceq (\wt x_{i,1},\wt y_{i,1}) \preceq \ldots  \preceq (\wt x_{i,m_i}, \wt y_{i,m_i}) \preceq (x_{i+1}, y_{i+1})
\]
for all $i\in \lb 0, m\rb$, where the points $(\wt x_{i,j}, \wt y_{i,j})$ belong to $\wt\omega$.

From the preceding discussion it seems that, given {$\SSP$}, the maximum passage time collected from $\wt\omega$ should be comparable to 
\begin{equation}\label{energy}
 \frac{1}{2} \cdot 2\cdot \sum_{i=0}^m\sqrt{(x_{i+1}-x_{i})(y_{i+1}-y_{i})}
\end{equation}
(although the pre-factors cancel each other out, we kept them to emphasize that the factor $\frac12$ comes from the weight of the points in $\wt\omega$ and the factor $2$ from the fact that $\wt\omega$ has intensity $4$, not $1$).

Now, since $\sum_{i=0}^m (x_{i+1}-x_{i})=\sum_{i=0}^m (y_{i+1}-y_{i})=n$, the sum in \eqref{energy} is at most $n$ with equality achieved only if $x_i=y_i$ for all $i$.
This implies that unless the points in {$\SSP$} literally land on the diagonal, the constraint enforced by {$\SSP$} 
leads to some loss in the passage time contributed by $\wt\omega$.
Of course, the points in {$\SSP$} are random, belonging to the point process $\omega$, and will never lie on the diagonal.
This leads to a tradeoff between the number of points collected from $\omega$ and from $\wt\omega$, or more generally (as in the proof of Theorem \ref{thm:main}) between the number of points collected from a given scale and from all the subsequent scales. 

A strategy in light of this is to demand that {$\SSP$} does not venture too far from the diagonal, and more importantly that the linear segments $(x_i,y_i) \to (x_{i+1},y_{i+1})$ have slopes close to one (note that this needs to happen at \emph{all scales}).
This leads to the intriguing question of identifying the optimal threshold for ``closeness'' that ensures that the passage time under such a constraint is large enough while simultaneously producing slopes that allow subsequent scales to contribute.

The following algebra is instructive.
Suppose we impose the slope bounds
\begin{equation}\label{constraint}
    1-\d
    \le \frac{y_{i+1}-y_i}{x_{i+1}-x_i}
    \le 1+\d
\end{equation}
for some $\d=o(1)$ and all $i\in\lb 0,m\rb$.
Applying the worst-case lower bound for every $i$ would then yield
\begin{equation}\label{thin}
    \begin{split}
        \sum_{i=0}^m \sqrt{(x_{i+1}-x_i)(y_{i+1}-y_i)}
        = \sum_{i=0}^m (x_{i+1}-x_i)\sqrt{\frac{y_{i+1}-y_i}{x_{i+1}-x_i}}
        &\ge \sqrt{1-\d}\sum_{i=0}^m (x_{i+1}-x_i)\\
        &\sim \left(1-\frac{\d}{2}\right)n.
    \end{split}
\end{equation}
Thus in one step a loss of a factor of $(1-\frac{\delta}{2})$ is incurred and hence after $k$ many iterations the loss is $(1-\frac{\delta}{2})^k$.
Since we would ideally want to perform $\log n$ many iterations, as there are $\log n$ many Poisson layers superimposed on each other, this suggests that we may take $\delta=\frac{1}{\log n}$.
This leads to the task of estimating the maximum passage time coming from $\omega$ (which is simply $m$ in \eqref{skeleton1}) subject to the constraint \eqref{constraint}, as a function of $\delta$.
To do this we consider a cylinder around the diagonal (i.e. the line $y=x$ joining $(0,0)$ and $(n,n)$) of width $\sqrt{\delta}$ (see Figure \ref{fig:iop}).
Split it into smaller cylinders of length $\frac{1}{\sqrt \delta}$, so there are $n\sqrt{\delta}$ many of them.
\begin{figure}[th]
    \centering
    \includegraphics[width=0.35\textwidth]{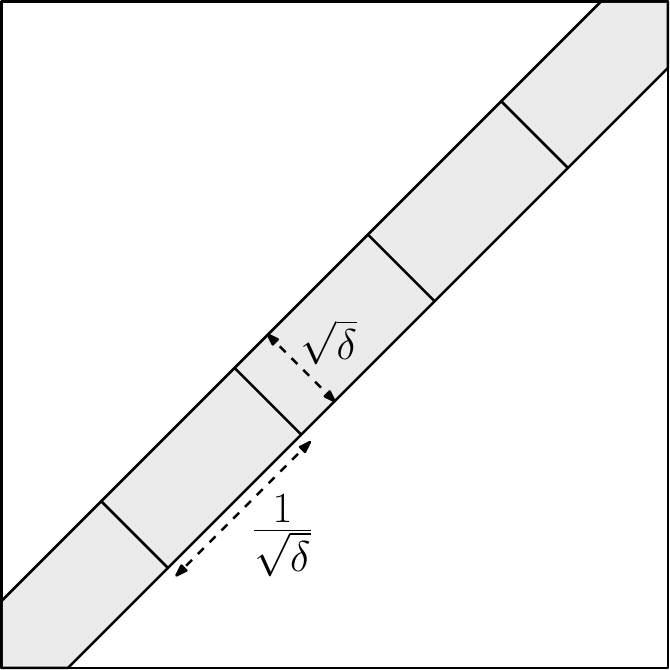}
    \caption{Depicted are $[0,n]^2$ (black square) and the cylinder (shaded gray) around the diagonal of width $\sqrt{\delta}$, divided into smaller cylinders of length $\frac{1}{\sqrt{\d}}$.}\label{fig:iop}
\end{figure}
There are two reasons for taking the smaller cylinders to have side lengths $\sqrt{\delta}$ and $\frac{1}{\sqrt \delta}$.
The first is that for two typical points in consecutive such cylinders, the slope of the line segment joining them satisfies \eqref{constraint}.
The second is that $\omega$, being a Poisson process of intensity $1$, typically has a point in a constant fraction of such cylinders due to their volumes being $1$.
Now forming {$\SSP$} by selecting one point from each cylinder leads to an overall passage time of order ${n}{\sqrt \delta}=\frac{n}{\sqrt{\log n}}$ 
collected from $\omega$. (Note that in our application, the Poisson processes have intensities of the form $\frac{2^{2k}}{n^2}$ and hence are much sparser than $\omega$ and $\wt\omega$.
To extend the above calculations to that setting the cylinders have to be dilated by a factor of $\frac{n}{2^k}$.)
Instead of collecting the maximum passage time of order $n$ per scale as in the proof of the {$O(n\log n)$} upper bound, the above strategy picks up only $\frac{n}{\sqrt{\log n}}$ per scale.
Repeating for $\log n$ many scales leads to a lower bound of $n \sqrt{\log n}$.
Note that this falls short of the $n(\log n)^{3/4}$ bound claimed by Theorem \ref{thm:main}. It turns out that the $n \sqrt{\log n}$ bound would also be insufficient to prove Theorem \ref{thm:secondmoment}.

We remedy this by observing that \eqref{thin} may be improved by 
leveraging \emph{cancellations} between the slopes of the various line segments. Namely, by accounting for the constraints $\sum_{i=0}^m(x_{i+1}-x_i)=\sum_{i=0}^m(y_{i+1}-y_i)=n$, we show that
\begin{equation}\label{thin1}
    \sum_{i=0}^m \sqrt{(x_{i+1}-x_i)(y_{i+1}-y_i)}
    \ge \bigl(1-\d^2\bigr)n.
\end{equation}
This allows us to choose $\delta=\frac{1}{\sqrt{\log n}}$ and still perform $\log n$ many iterations, thereby obtaining a contribution of $\frac{n}{(\log n)^{1/4}}$ per scale for an overall passage time of ${n}{(\log n)^{3/4}}$.

It is worth noting that the path we construct has transversal fluctuations of order $\frac{n}{(\log n)^{1/4}}$ (ignoring lower-order corrections), which in light of Remark \ref{lower1234} raises the intriguing question of identifying the true geodesic transversal fluctuations.

On a final technical note, in the actual implementation of the above strategy, we do not add points from \emph{every} scale to our path, but instead work with a sequence of scales that have a separation of order $\log\log n$.
This is needed to ensure that between each consecutive pair of points in the skeleton at a given scale, there is room for enough cylinders to yield useful concentration estimates.\\

The argument for the BRW setting (Theorem \ref{thm:brw}) is related but more complicated.
This is because while both noise environments are hierarchical, the BRW has long-range correlations on account of the same values being shared across multiple vertices, whereas in the heavy-tailed model the vertex weights are i.i.d..
Thus, in the BRW, the path must spend a lot of time in a given box to pick up a large amount of weight, unlike in the heavy-tailed setting where as long as the path passes through a large weight, its passage time becomes large regardless of its trajectory away from that weight.
Another restriction of a similar nature is that various candidates for the heavy path at a given scale must be well-separated, again due to the long-range correlations, to gain access to the larger-valued weights arising from fluctuations.
Both these features reduce the choices for candidate heavy paths, leading to the weaker bound in Theorem \ref{thm:brw} compared to Theorem \ref{thm:main}.

To keep the exposition light, we postpone further discussions on the proof of Theorem \ref{thm:brw} to Section \ref{sec:brw} beyond mentioning that in the multi-scale construction, unlike the heavy-tailed case, the choice of the separation of scales is dictated by an optimization problem.
The reasoning for the latter will be apparent from the discussion presented in Section \ref{sec:brw}.

Multi-scale analysis of geodesics in log-correlated environments was developed in the context of Liouville quantum gravity and the prelimiting Liouville FPP in 
\cite{DGFirstPassagePercolation2017,DGUpperBoundsLiouville2019,DDLiouvilleFirstpassagePercolation2019,DDDFTightnessLiouvilleFirst2020}.
There are two main differences between the above settings and ours.
The first is that these, being models of FPP, involve an optimization over \emph{all} paths, whereas we restrict to directed paths.
As already alluded, while for light-tailed i.i.d. environments both FPP and LPP are expected to be in the KPZ universality class, in log-correlated 
settings significant differences are expected to emerge.
Secondly, in models of Liouville quantum gravity the weights are given by the exponential of the underlying log-correlated field, in particular making them positive, whereas in our setting negative weights must be accounted for.
Very recently, a multi-scale path construction involving elements similar to the construction presented in this paper appeared in a paper of Ding, Gwynne, and Zhuang \cite{DGZPercolationThickPoints2026} on connectivity properties of thick point sets of high-dimensional log-correlated fields with applications to high-dimensional Liouville FPP.
A more in-depth discussion on the potential counterpart results for LPP in high-dimensional log-correlated environments
will be presented later in Remark \ref{rem:dgz}.
\\

The concentration result stated in Theorem \ref{thm:conc} is a consequence of the Efron--Stein inequality (recorded below as Proposition \ref{efron-stein}) and a dyadic scale argument adapted from \cite{MarLinearGrowthGreedy2002}.
A similar strategy can be applied in the BRW setting to establish Theorem \ref{thm:conc-brw} with a sub-optimal upper bound in \eqref{eq:brw-fluctuations}; this was done in the previous arXiv version of this paper.
Here we instead prove Theorem \ref{thm:conc-brw} by leveraging the Gaussianity of the BRW via the Borell--TIS inequality (a special case of which is recorded in Proposition \ref{borell-TIS}, see also \cite[Theorem 2.1.1]{ATRandomFieldsGeometry2007}).
\\

We end this discussion by briefly mentioning that the above theme of maximizing passage times under certain slope constraints enforced on the geometry of the maximizing path as encountered in \eqref{constraint} has previously appeared in work of Berger and Torri \cite{BTEntropycontrolledLastPassagePercolation2019,BTHammersleysLastPassagePercolation2021}.

\section{Preliminaries}\label{sec:preliminaries}

In this section we record our notation and collect some elementary geometric facts which will be in use throughout many of the arguments. 

\subsection{General notation}\label{sec:notation}

Throughout the paper we denote by $C,C',c,c'$, etc., finite positive constants
that do not depend on any parameters except possibly the joint law of the vertex weights of the LPP model under consideration.
Uppercase constants $C$ may increase from one line to the next,
and lowercase constants $c$ may decrease from one line to the next.

We adopt the usual Landau asymptotic notation.
For quantities $A,B\in\R$ depending on some parameter (usually $n\in\N$), we 
write $A=O(B)$ if $|A|\le CB$ for some constant $C>0$ as in the previous paragraph.
We write $A=\Theta(B)$ if $A=O(B)$ and $B=O(A)$.
We write $A=o(B)$ if $\frac{A}{B}\to 0$ as the limit is taken in some parameter
(usually $n\to\infty$), and $A\sim B$ if $A=B(1+o(1))$, i.e. if $\frac{A}{B}\to 1$.
We will sometimes also write $A\ls B$ or $B\gs A$ if $A=O(B)$, 
and $A\ll B$ or $B\gg A$ if $A=o(B)$.

For a measurable set $A\subset\R^2$, we denote by $|A|$ the area (Lebesgue measure) of $A$.
However, if $A$ is understood to be a finite set (e.g. if $A=B\cap\Z^2$ for some bounded $B\subset \R^2$), then we denote by $|A|$ the cardinality of $A$.
In our applications this should not cause confusion.
For real numbers $x\le y$ we write $\lb x,y\rb\coloneqq[x,y]\cap\Z$.

We take all logarithms to have base $2$.

\subsection{Plane geometry}
For $a=(a_1,a_2),b=(b_1,b_2)\in\R^2$, we write $a\preceq b$ if $a$ is weakly less than $b$ coordinate-wise, and $a\prec b$ if $a$ is strictly less than $b$ coordinate-wise. 
\begin{definition}[Rectangles,  slopes, and cylinders]\label{def:geometry}
We introduce the following notation:
\begin{enumerate}[label=(\roman*)]
    \item For $a,b\in\R^2$ with $a\prec b$, we denote by $\Rect(a,b)$ the rectangle whose bottom-left and top-right corners are $a$ and $b$ 
    respectively.\\

    \noindent In what follows we fix $a\prec b$ and write $R \coloneqq \Rect(a,b)$.\\

    \item We denote by $\Diag(R)$ the diagonal of $R$ between the bottom-left and top-right corners of $R$.
    In other words, $\Diag(R)$ is the line segment with endpoints $a$ and $b$.
    \item We denote by $\Slope(R)$ the slope of $\Diag(R)$.
    \item We will often abuse notation by writing $\Slope(a,b) \coloneqq \Slope(\Rect(a,b))$, 
    so that $\Slope(a,b)$ equals the slope of the line segment with endpoints $a$ and $b$:
    \[
        \Slope(a,b) = \frac{b_2-a_2}{b_1-a_1}.
    \]
    \item 
    For $r\in(0,1)$ we define
    \begin{equation*}
        \Cyl_r(R) \coloneqq \bigl\{ (x,y)\in R : |y-\Slope(R)x| \le r(b_2-a_2) \bigr\},
    \end{equation*}
    see Figure \ref{fig:geometry}.
    A straightforward calculation shows that  
    \footnote{For instance, $\Cyl_r(R)$ can be expressed as the disjoint union of a parallelogram of area $2r(1-r)|R|$ and two right triangles of area $\frac12r^2|R|$.}
    \[
        |\Cyl_r(R)| = (2-r)r|R|.
    \] 
    In particular, since $r\in(0,1)$ we get
    \begin{align}\label{eq:area_cyl}
        |\Cyl_r(R)|\ge r|R|.
    \end{align}
\end{enumerate}
\end{definition}

\begin{figure}[ht]
    \centering
    \includegraphics[width=0.2\linewidth]{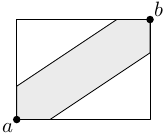}
    \caption{The black rectangle is $\Rect(a,b)$. The shaded gray region is the cylinder $\Cyl_r(\Rect(a,b))$, for some $r\in(0,1)$.}\label{fig:geometry}
\end{figure}

The following lemma asserts a lower bound for the number of lattice points within a cylinder in terms of the area of that cylinder.
This will allow us to work almost exclusively with the area in the proof of Theorem \ref{thm:main}, thereby significantly simplifying our notation.

\begin{lemma}[Number of lattice points compared to area]\label{lem:lattice_count}
    Fix a rectangle $R=\Rect(a,b)$ and $r\in(0,1)$.
    Suppose that
    $\Cyl_r(R)\cap\Z^2$
    is nonempty, which for instance holds whenever
    \begin{equation}\label{08508}
        r\cdot \min\{b_1-a_1,\ b_2-a_2\} \ge 10.
    \end{equation}
    Then
    \begin{align}\label{130}
        \bigl|\Cyl_r(R)\cap\Z^2\bigr|
         \ge \bigl|\Cyl_r(R)\bigr| - 100\max\{b_1-a_1,\ b_2-a_2\}.
    \end{align}
\end{lemma}
While Lemma \ref{lem:lattice_count} may be proved using elementary arguments, below we present a proof involving the application of a special case of the classical Pick's theorem (e.g. \cite[p. 209]{CoxIntroductionGeometrySecond1989}).
\begin{proposition}[Pick's theorem]\label{prop:pick}
    Let $\sP$ be a  convex polygon whose vertices lie on the lattice $\Z^2$.
    Then we have that
    \[
        \left|\msf{P}\right| = 
        \left|\msf{P}\cap\Z^2\right| - \frac12\bigl|(\partial\msf{P})\cap\Z^2\bigr|-1.
    \]
    In particular,
    \[
        \left|\msf{P}\cap\Z^2\right|
        \ge \left|\msf{P}\right|.
    \]
\end{proposition}

\begin{proof}[Proof of Lemma \ref{lem:lattice_count}]
    Note that $\sC\coloneqq \Cyl_r(R)$ is a convex hexagon; we label its vertices $v_1,\dots,v_6$.
    Let $\wh\sC$ be the convex hull of the points obtained by rounding each $v_i$ to the nearest lattice point in
    $\sC\cap\Z^2$
    (with ties broken in some arbitrary way).
    Note that $\sC\cap\Z^2$ is non-empty by hypothesis. 
    By convexity, $\wh{\sC}\subset\sC$.
    Moreover, observe that
    \[
        \sup_{a\in\partial\sC}\mrm{dist}(a, \widehat{\sC}) =
        \sup_{a\in\partial\sC}\inf_{b\in \partial\widehat{\sC}}\|a-b\|_1 \le 2,
    \]
    and thus it follows from Fubini's theorem that
    \[
        |\sC|-|\widehat{\sC}| \le 2\cdot \per(\sC).
    \]
    The perimeter is in turn upper bounded by
    \[
        \per(\sC)\le 50\max\{b_1-a_1, \ b_2-a_2\}
    \]
    (the coefficient $50$ is certainly not optimal).
    We deduce that
    \begin{align*}
        \bigl|\sC\cap\Z^2\bigr| \ge \bigl|\widehat{\sC}\cap\Z^2\bigr|
        &\ge |\widehat{\sC}|\\
        &\ge |\sC| - 100\max\{b_1-a_1,\ b_2-a_2\},
    \end{align*}
    where the second inequality holds by Pick's theorem.
    We are done.
\end{proof}

\begin{remark}
    It is not difficult to prove an analogue of \eqref{130} with the inequality reversed
    (for instance by using Pick's theorem in a less crude manner).
    However we do not pursue this here as we only need the lower bound in our applications.
\end{remark}

The next lemma, illustrated in Figure \ref{fig:slopebound}, will be used throughout the proof of Theorem \ref{thm:main}.
Here---and nowhere else in the paper---we violate the notational conventions of Section \ref{sec:notation} by using $c$ to denote a point in $\R^2$.
This local abuse of notation should not cause any confusion.
\begin{lemma}[Worst-case slope bounds]\label{lem:slopebound}
    Fix $a,b,c,d\in\R^2$ with $a\prec b\prec c\prec d$.
    Assume that
    \[
        \Slope(a,b)=\Slope(b,c)=\Slope(c,d)
    \]
    and that
    \begin{align}\label{eq:separation-hyp}
        c_i-b_i \ge \max\{b_i-a_i,\ d_i-c_i\}\quad\text{for } i\in\{1,2\},
    \end{align}
    i.e. the rectangles are similar and the middle one has (weakly) larger sides than the other two. 
    Then for all $r\in(0,1)$, for all $v\in\Cyl_r(\Rect(a,b))$, and for all $w\in \Cyl_r(\Rect(c,d))$, we have that
    \begin{align}\label{eq:slopebound}
        \frac{1}{1+2r} \le \frac{\Slope(v,w)}{\Slope(b,c)} \le 1+2r.
    \end{align}
\end{lemma}

\begin{figure}[ht]
    \includegraphics[width=0.45\linewidth]{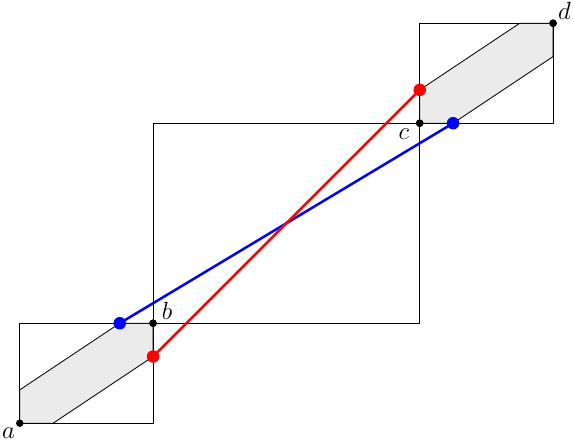}
\caption{
    An illustration of Lemma \ref{lem:slopebound}.
    The rectangles $\Rect(a,b), \Rect(b,c)$ and $\Rect(c,d)$ have identical slopes, and $\Rect(b,c)$ has (weakly) larger side lengths than the other two rectangles.
    Shaded in gray are the cylinders $\Cyl_r(\Rect(a,b))$ and $\Cyl_r(\Rect(c,d))$.
    Among all lines having one endpoint in each cylinder, the line with the largest slope is drawn in red, and the line with the smallest slope is drawn in blue.
    A direct calculation shows that the red line has slope $(1+2r)\Slope(b,c)$ and the blue line has slope $\frac{1}{1+2r}\Slope(b,c)$, which are the upper and lower bounds asserted in Lemma \ref{lem:slopebound}.
}\label{fig:slopebound}
\end{figure}

Lemma \ref{lem:slopebound} is elementary, so for brevity we only sketch a proof.
Since the hypotheses and conclusions of the lemma are invariant under rescaling of the form $(x,y)\mapsto (tx,y)$, it suffices to treat the case where all the slopes are equal to $1$, i.e. the rectangles are squares.
For this case the validity of the lemma becomes apparent upon changing to the coordinate system $(x',y')\coloneqq(x-y,x+y)$.

\section{Critical heavy-tailed last passage percolation}
\label{sec:superlinear1}
From now until the end of Section \ref{sec:conc} we will exclusively consider the $\al=2$ heavy-tailed LPP model.
We state the outcome of our multi-scale construction (see Section \ref{iop}) in the upcoming proposition which immediately implies Theorem \ref{thm:main}.
The rest of this section is devoted to the proof of the proposition, except for Section \ref{sec:secondmoment} which contains the proof of Theorem \ref{thm:secondmoment}.

To precisely state the relevant properties of the multi-scale construction it is convenient to use the following terminology.
\begin{definition}[Scales]\label{def:kscale}
    For $k\ge0$, we say that a vertex $v\in\Z^2$ is of \emph{scale $k$} if
    \[
        \frac{n}{2^k} < X(v) \le \frac{n}{2^{k-1}}.
    \]
\end{definition}
Note that $k=0$ corresponds to the largest values; this indexing convention is informed by the iterative nature of our multi-scale construction where we choose vertices from one scale at a time, beginning with the largest weights and ending with the smallest weights (details to follow).

For future reference we note that there exists
$\tail>0$
such that for all sufficiently large $n$,  all $k\le \frac{1}{10}\log  n$, and all $v\in\Z^2$,
\begin{align}\label{eq:tail}
    \P(v \ \textup{is of scale $k$}) \ge \tail \frac{2^{2k}}{n^2}.
\end{align}

The next proposition asserts that with high probability there exists an up-right directed lattice path from $(0,0)$ to $(n,n)$ that contains many vertices of scale $k$, for many values of $k$.
We write
\begin{align}\label{eq:s}
    \ss  \coloneqq  100 \log  \log n,
    \qquad \MM  \coloneqq  \left\lfloor \frac{\log  n}{10\ss}\right\rfloor,
    \qquad  \JJ  \coloneqq  (\log n)^{1/4}.
\end{align}
The parameter $\ss$ is the separation between the scales we will consider, and $\MM$ is the total number of scales involved in our multi-scale construction (we consider only $\frac{\log n}{10\ss}$ many scales instead of $\frac{\log n}{\ss}$ many for technical reasons that will become clear later).
This separation will be necessary to obtain desirable concentration estimates.
The parameter $\JJ$ will play a similar role as $\frac{1}{\sqrt{\d}}$ from Section \ref{iop}.

\begin{proposition}\label{prop:V}
    For all $n\ge1$ large enough such that $\MM\ge1$, there exist random sets of vertices
    \[
        \VV^{(1)}\subset \VV^{(2)}\subset \cdots\subset\VV^{(\MM)}\subset\dmn
    \]
    with the following properties.
    \begin{enumerate}[label=(\roman*)]
        \item \label{V1} 
        There exists an up-right directed lattice path from $(0,0)$ to $(n,n)$ that passes through every vertex in $\VV^{(\MM)}$.
        Equivalently,  $\VV^{(\MM)}$ is totally ordered with respect to $\prec$.
        \item \label{V2} 
        For all $\ell\in\lb1,\MM\rb$, the set $\VV^{(\ell)}\setminus\VV^{(\ell-1)}$ consists entirely of vertices of scale $\ell\ss$ (see Definition \ref{def:kscale}).
        Here, we write $\VV^{(0)} \coloneqq \bigl\{(0,0), (n,n)\bigr\}$.
        \item \label{V3} There exists $c>0$ such that for all sufficiently large $n$,
        \[
            \P\left(
                \bigl|\VV^{(\ell)}\setminus\VV^{(\ell-1)}\bigr| \ge c\frac{2^{\ell\ss}}{\JJ}\ \textup{for all}\ \ell\in\lb1,\MM\rb
            \right)
            \ge 1-e^{-(\log n)^{97}}.
        \]
    \end{enumerate}
\end{proposition}

As indicated above, Theorem \ref{thm:main} follows directly from  Proposition \ref{prop:V}:
\begin{proof}[Proof of Theorem \ref{thm:main}]
By Proposition \ref{prop:V}\ref{V1} and the non-negativity of the weights (Definition \ref{def:lpp}), we have that 
\[
    L_n \ge \sum_{v\in\VV^{(\MM)}} X(v)
    \ge \sum_{\ell=1}^\MM \sum_{v\in\VV^{(\ell)}\setminus\VV^{(\ell-1)}} X(v).
\]
By Proposition \ref{prop:V}\ref{V2}--\ref{V3} and Definition \ref{def:kscale}, the following lower bound holds for all $\ell\in\lb1,\MM\rb$ with 
probability at least $1-e^{-(\log n)^{97}}$:
\[
    \sum_{v\in\VV^{(\ell)}\setminus\VV^{(\ell-1)}} X(v) \ge c\frac{2^{\ell\ss}}{\JJ} \cdot \frac{n}{2^{\ell\ss}} = c\,\frac{n}{\JJ}.
\]
Combining the last two displays yields the lower bound
\[
    L_n 
    \ge \MM\cdot c\,\frac{n}{\JJ}
    \ge c\,\frac{n(\log n)^{3/4}}{\log\log n}.
\]
This completes the proof of Theorem \ref{thm:main}.
\end{proof}

The rest of this section is organized as follows.
In Section \ref{sec:multiscale} we construct the vertex sets $\VV^{(\ell)}$.
In Section \ref{sec:properties} we deduce a number of properties from the definition
of the construction, including Proposition \ref{prop:V}\ref{V1}--\ref{V2}.
In Section \ref{sec:vertex_sets_large} we prove Proposition \ref{prop:V}\ref{V3}, thus completing the proof of Proposition \ref{prop:V}.
Finally, in Section \ref{sec:secondmoment} we show how these arguments can be extended to prove Theorem \ref{thm:secondmoment}.

\subsection{Multi-scale construction}\label{sec:multiscale}
In this section we construct the vertex sets $\VV^{(\ell)}$ of Proposition \ref{prop:V}.
Since our construction is inductive, it will be convenient to define the set
\begin{align}\label{eq:base}
    \VV^{(0)} \coloneqq \bigl\{(0,0), (n,n)\bigr\}
\end{align}
which will serve as the base case for our induction.
We now assume that, for some 
$\ell\in\lb 0, \MM-1 \rb$, 
we have specified the distribution of a random totally ordered set of vertices
\[
    \VV^{(\ell)}=
    \left\{
        \vv^{(\ell)}_1\prec \vv^{(\ell)}_2\prec\dots\prec \vv^{(\ell)}_{|\VV^{(\ell)}|}
    \right\}
\]
(for notational simplicity we regard $\VV^{(0)}$ as a degenerate random set).
Given this data, we will construct $\VV^{(\ell+1)}$.
Before continuing we introduce some notation.
We write
\[
    \RR^{(\ell)} \coloneqq \left\{\Rect\bigl(\vv^{(\ell)}_i,\;\vv^{(\ell)}_{i+1}\bigr) : i\in\lb 1,\, |\VV^{(\ell)}| - 1 \rb
    \right\},
\]
as well as
\[
    R^{(\ell)}_i  \coloneqq  \Rect\bigl(\vv^{(\ell)}_i,\;\vv^{(\ell)}_{i+1}\bigr),\qquad i\in\lb 1,\, |\VV^{(\ell)}| - 1 \rb.
\]
We also define
\begin{equation}\label{eq:rho}
    \rr \coloneqq (\log \log n)^{1/2}.
\end{equation}
We will use $\rr$ to slightly widen the various cylinders encountered in the construction (recall the proof sketch in Section \ref{iop}), thereby simplifying the later analysis.
For these purposes the exact value $\rr=(\log\log n)^{1/2}$ is unimportant, and it can be replaced with any other $1\ll\rr\ls (\log \log n)^{1/2}$ (or even a sufficiently large constant) without affecting any of the upcoming arguments.

We are now ready to construct the set $\VV^{(\ell+1)}$.
Though we have strived to keep the notation light, the multi-scale nature of our
construction would make it difficult to prove Proposition \ref{prop:V} without some
notational overhead.
Before getting into details, let us briefly describe the strategy.
For each rectangle $R^{(\ell)}_i\in\RR^{(\ell)}$, {we consider certain sub-rectangles along the diagonal $\Diag(R^{(\ell)}_i)$.}
We add vertices of scale $(\ell+1)\ss$ from within the sub-rectangles.
To ensure that the slopes between the newly added vertices are still close to $1$, we will further demand that they lie in certain cylinders as indicated in Figure \ref{fig:construction} (recall also the proof sketch in Section \ref{iop}).
\\

We now make this precise.
Fix a rectangle $R^{(\ell)}_i\in\RR^{(\ell)}$.
We describe the sub-rectangles of $R^{(\ell)}_i$ alluded to above now.

We cover $\Diag(R^{(\ell)}_i)$ by rectangles $R^{(\ell)}_{i,j}$ of the same aspect ratio as $R^{(\ell)}_i$ and of area $\sA>0$ (to be specified momentarily), such that the $R^{(\ell)}_{i,j}$ have their bottom-left and top-right corners on $\Diag(R^{(\ell)}_i)$ and such that their interiors are pairwise disjoint 
(see Figure \ref{fig:construction}, Step 1).
Let $m$ be the number of such rectangles.
Then $m$ satisfies 
\[
    \frac{\bigl|R^{(\ell)}_i\bigr|}{m^2} = \sA,\quad \text{or equivalently}\quad m=\frac{1}{\sqrt{\sA}}\sqrt{\bigl|R^{(\ell)}_i\bigr|}.
\]
In view of the proof strategy outlined in Section \ref{iop}, we would like to set $\sA=\JJ^2\frac{n^2}{2^{2(\ell+1)\ss}}$, so that with high probability each cylinder $\Cyl_{\rr/\JJ^2}(R^{(\ell)}_{i,j})$, having area at least 
$\frac{\rr}{\JJ^2}\sA$
(recall \eqref{eq:area_cyl}), will contain a vertex of scale $(\ell+1)\ss$.
However, to avoid divisibility issues we instead work with the least $\sA\ge \JJ^2\frac{n^2}{2^{2(\ell+1)\ss}}$ for which the desired sub-rectangles exist.
\footnote{Note that in principle the area of $R^{(\ell)}_i$ could be so small that its diagonal does not admit such a covering. However, our construction ensures that this never happens (see Lemma \ref{lem:apriori}\ref{apriori_rect}), and we will not discuss this possibility further.}
Equivalently, we set 
\begin{align}\label{3095}
    m = m_i^{(\ell)}
    \coloneqq
    \left\lfloor\frac{2^{(\ell+1)\ss}}{\JJ n}\cdot \sqrt{\Bigl|R^{(\ell)}_i\Bigr|}\right\rfloor.
\end{align}
Thus, each rectangle $R^{(\ell)}_{i,j}$ for $j\in\lb 1,m_i^{(\ell)}\rb$ has the same slope as $R^{(\ell)}_i$,
and has area $\Bigl|R_{i,j}^{(\ell)}\Bigr|=\sA$ satisfying the following lower bound:
\begin{align}\label{308}
    \Bigl|R_{i,j}^{(\ell)}\Bigr|
    \ge
    \JJ^2\frac{n^2}{2^{2(\ell+1)\ss}}.
\end{align}

\begin{figure}[hbtp]
    \centering
    \includegraphics[width=\linewidth]{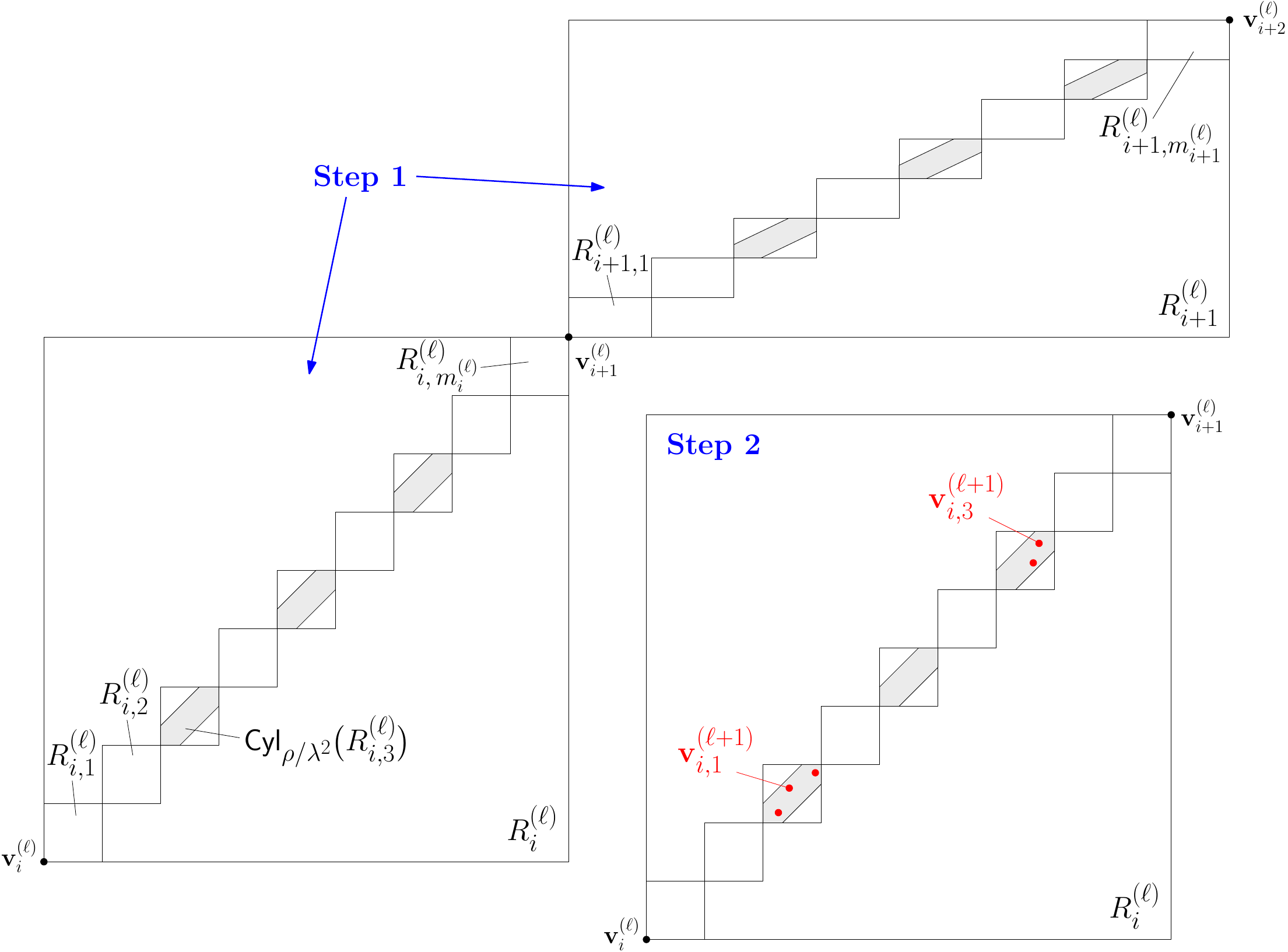}
    \caption{The construction of $\VV^{(\ell+1)}$.\\
            \textbf{Step 1:} Depicted are the rectangles $R^{(\ell)}_i$ and $R^{(\ell)}_{i+1}$.
            By definition, their bottom-left and top-right corners are consecutive vertices in $\VV^{(\ell)}$. 
            The set $\VV^{(\ell+1)}$ is constructed by adding vertices from each such rectangle to $\VV^{(\ell)}$ according to the following procedure.
            For notational simplicity we explain the procedure only for $R^{(\ell)}_i$.
            The construction begins by fixing congruent sub-rectangles $R^{(\ell)}_{i,j}$ along the diagonal of $R^{(\ell)}_i$, each with area $\JJ^2\frac{n^2}{2^{2(\ell+1)\ss}}(1+o(1))$.
            The number of such sub-rectangles is denoted by $m^{(\ell)}_i$ (as drawn, $m^{(\ell)}_i=9$, and $m^{(\ell)}_{i+1}=8$).
            Shaded in gray are the cylinders $\Cyl_{\rr/\JJ^2}(R^{(\ell)}_{i,{2j+1}})$ for $j\in\bigl\lb 1,\,\floor{\frac12 m^{(\ell)}_i}-1\bigr\rb$ 
            (the purpose of omitting every other cylinder will be explained later).\\
            \textbf{Step 2:}
            Note that this step is only depicted for $R^{(\ell)}_i$.
            Each shaded cylinder contains some number of vertices of scale $(\ell+1)\ss$, drawn in red.
            For each cylinder $\Cyl_{\rr/\JJ^2}(R^{(\ell)}_{i,2j+1})$ containing a vertex of scale $(\ell+1)\ss$, we select one such vertex $\vv^{(\ell+1)}_{i,j}$.
            The resulting set of vertices is denoted $\VV^{(\ell+1)}_i$.
            Finally, $\VV^{(\ell+1)}$ is defined to be the union of $\VV^{(\ell)}$ and all the sets $\VV^{(\ell+1)}_i$.
    }\label{fig:construction}
\end{figure}

Given the above, we define the random index set
\begin{equation}\label{301}
    \II_i^{(\ell+1)}  \coloneqq 
    \left\{
        j\in\bigl\lb 1,\,\floor{\tfrac{1}{2} m_i^{(\ell)}}-1\bigr\rb
        :
        \Cyl_{\rr/\JJ^2}\bigl(R^{(\ell)}_{i,2j+1}\bigr)
        \textup{ contains a vertex of scale $(\ell+1)\ss$}
    \right\}
\end{equation}
(see Figure \ref{fig:construction}, Steps 1 and 2).
That is, we are looking at every odd-indexed rectangle whose corresponding cylinder contains a vertex of scale $(\ell+1)\ss$ (excluding the first rectangle, and also excluding the last rectangle when $m^{(\ell)}_i$ is odd).
This choice of index set is to ensure that the slopes between points selected from consecutive rectangles are sufficiently close to $1$, as will be quantified more precisely later on.
For each $j\in \II_i^{(\ell+1)}$, we select one vertex $\vv^{(\ell+1)}_{i,j}$ of scale $(\ell+1)\ss$ from $\Cyl_{\rr /\JJ^2} (R^{(\ell)}_{i,2j+1})$ 
(see Figure \ref{fig:construction}, Step 2).
The precise manner in which the vertices $\vv^{(\ell+1)}_{i,j}$ are selected from the cylinders is irrelevant to the later analysis; we can, for instance, work with any deterministic rule, such as picking the vertex smallest in the lexicographic ordering.
Regardless of the selection procedure, we obtain a collection of vertices
\begin{equation}\label{eq:Vi}
    \VV^{(\ell+1)}_i
     \coloneqq 
    \left\{
        \vv^{(\ell+1)}_{i,j} : j \in \II_i^{(\ell+1)}
    \right\}.
\end{equation}
Notice that $\VV^{(\ell+1)}_i$ is totally ordered: it has been constructed by selecting at most one vertex from each of the rectangles $R^{(\ell)}_{i,j}$, and the latter are ordered by definition (see Figure \ref{fig:construction}, Step 1).

We repeat the above procedure for every rectangle $R^{(\ell)}_i\in\RR^{(\ell)}$, and at last define
\begin{align}\label{eq:def_V}
    \VV^{(\ell+1)} \coloneqq \VV^{(\ell)}\cup\bigcup_{i=1}^{|\RR^{(\ell)}|}\VV^{(\ell+1)}_i.
\end{align}
This ends the construction.

The upcoming Sections \ref{sec:properties} and \ref{sec:vertex_sets_large} are devoted to proving that the above construction indeed satisfies the claims in Proposition \ref{prop:V}.
We begin with some notation to streamline the analysis.
For a given rectangle $R\in\RR^{(\ell)}$ we will need to refer to the unique rectangle in $\RR^{(\ell-1)}$ that contains $R$; the following notation offers a convenient way to do so.
The reader may find it helpful to look at Figure \ref{fig:area} below for an illustration of the objects defined here, and may prefer to skip ahead to Section \ref{sec:properties} until this notation is first used (we will make a disclaimer whenever we use it).

For $R_i^{(\ell)}\in\RR^{(\ell)}$, we denote by $\RR^{(\ell+1)}_i$ the set of rectangles in $\RR^{(\ell+1)}$ that are contained in $R_i^{(\ell)}$.
In symbols,
\begin{align}\label{eq:Ri}
    \RR^{(\ell+1)}_i  \coloneqq  \left\{R\in\RR^{(\ell+1)} : R\subset R_i^{(\ell)}\right\}.
\end{align}
By construction the sets $\RR^{(\ell+1)}_1, \RR^{(\ell+1)}_2,\dots$ are pairwise disjoint, and satisfy
\begin{align}\label{eq:RR-decomp}
    \RR^{(\ell+1)} = \bigcup_{i=1}^{|\RR^{(\ell)}|}
    \RR^{(\ell+1)}_i.
\end{align}
Indeed, $\RR^{(\ell+1)}_i$ is the collection of all rectangles whose bottom-left and top-right corners are 
consecutive vertices in $\VV^{(\ell+1)}_i$, 
together with two more rectangles: one whose bottom-left corner is the bottom-left corner of $R^{(\ell)}_i$ and whose top-right corner is the ``first'' vertex in $\VV^{(\ell+1)}_i$, 
and one whose bottom-left corner is the ``last'' vertex in $\VV^{(\ell+1)}_i$ and whose top-right corner is the top-right corner of $R^{(\ell)}_i$ (see Figures \ref{fig:construction} and \ref{fig:area}).
In symbols,
\begin{equation}\label{eq:Ri2}
    \begin{split}
        \RR^{(\ell+1)}_i =
    &\biggl\{
        \Rect\Bigl(\vv^{(\ell+1)}_{i,j},\; \vv^{(\ell+1)}_{i,j+1}\Bigr)
        : j\in\II^{(\ell+1)}_i\setminus\{\sup\II^{(\ell+1)}_i\}
    \biggr\}\\
    &\qquad\quad\cup
    \biggl\{
        \Rect\Bigl(\vv^{(\ell)}_{i},\; \vv^{(\ell+1)}_{i,1}\Bigr)
    \biggr\}
    \cup
    \biggl\{
        \Rect\Bigl(\vv^{(\ell+1)}_{i,\, \sup\II^{(\ell+1)}_i},\; \vv^{(\ell)}_{i+1}\Bigr)
    \biggr\}.
    \end{split}
\end{equation}

\subsection{Properties of the vertex sets}\label{sec:properties}
We now deduce some properties of the vertex sets $\VV^{(\ell)}$ 
from the construction described in the previous subsection.

First we prove Proposition \ref{prop:V}\ref{V1}--\ref{V2}: 
\begin{proof}[Proof of Proposition \ref{prop:V}\ref{V1}--\ref{V2}]
    Recall that $\VV^{(0)} \coloneqq \bigl\{(0,0),(n,n)\bigr\}$ is totally ordered.
    Also, as discussed above, each set $\VV^{(\ell+1)}_i$ is totally ordered.
    Proposition \ref{prop:V}\ref{V1} now follows by induction along with the observations that 
    \[
        \VV^{(\ell+1)}_i\subset R^{(\ell)}_i = \Rect\bigl(\vv^{(\ell)}_i, \vv^{(\ell)}_{i+1}\bigr),
    \]
    and that each rectangle $R^{(\ell)}_i$ has corners given by consecutive vertices in the totally ordered set $\VV^{(\ell)}$ (the latter is totally ordered by the induction hypothesis).
    
    Every vertex in $\VV^{(\ell+1)}\setminus\VV^{(\ell)}$ is of scale $(\ell+1)\ss$ by construction (see \eqref{301} and \eqref{eq:def_V}), which is the content of
    Proposition \ref{prop:V}\ref{V2}.
\end{proof}

Next we prove the following lemma, which gives a priori bounds for the areas and slopes of the rectangles demarcated by consecutive vertices of any given scale.
\begin{lemma}[A priori estimates]\label{lem:apriori}
    The following is true almost surely.
    For all $\ell\in\lb1,\MM\rb$, every rectangle $R^{(\ell)}\in\RR^{(\ell)}$ enjoys the following two properties:
    \begin{enumerate}[label=(\roman*)]
        \item\label{apriori_rect}
        \emph{(Large area).}
        \[
            \bigl|R^{(\ell)}\bigr| \ge \JJ^2\frac{n^2}{2^{2\ell\ss}}.
        \]
        \item\label{apriori_slope}
        \emph{(Controlled slope).}
        There exists $R^{(\ell-1)} \in \RR^{(\ell-1)}$ such that $R^{(\ell)} \subset R^{(\ell-1)}$ and 
        \begin{equation*}
            \frac{1}{1+2\frac{\rr}{\JJ^2}}
            \le \frac{\Slope(R^{(\ell)})}{\Slope(R^{(\ell-1)})}
            \le 1+2\frac{\rr}{\JJ^2}.
        \end{equation*}

        Thus, iterating yields the bounds
        \begin{align}\label{eq:iteratedslope}
            \left(1+2\frac{\rr}{\JJ^2}\right)^{-\ell} \le 
            \Slope\bigl(R^{(\ell)}\bigr)
            \le \left(1+2\frac{\rr}{\JJ^2}\right)^{\ell}.
        \end{align}
        Note that since $\ell\le \MM \le C \frac{\log n}{\log\log n}$ and $\frac{\rr}{\JJ^2}=\left(\frac{\log\log n}{\log n}\right)^{1/2}$,
        these last bounds are roughly of the form
        $e^{\pm C\left(\frac{\log n}{\log\log n}\right)^{1/2}}$
        and hence are sub-polynomial in $n$ which will be crucial in our later applications.
    \end{enumerate}
\end{lemma}

\begin{proof}
    We first verify the assertions of
    the lemma
    for $\ell=1$.
    Recalling that $\RR^{(0)}=\{R^{(0)}_1\}=\{[0,n]^2\}$,
    we see from \eqref{308} that the rectangles $R^{(0)}_{1,j}$ are \emph{squares} of area
    \[
        \Bigl|R^{(0)}_{1,j}\Bigr|\ge\JJ^2\frac{n^2}{2^{2\ss}}
    \]
    whose bottom-left and top-right corners lie on $\Diag([0,n]^2)$.

    Since we restricted to odd indices in \eqref{301}, it follows that every rectangle $R^{(1)}_i\in\RR^{(1)}$ contains a square $R^{(0)}_{1,j}$ for some $j$ (see Figure \ref{fig:area}).
    In particular,
    \[
        \Bigl|R^{(1)}_i\Bigr| \ge \Bigl|R^{(0)}_{1,j}\Bigr| \ge \JJ^2\frac{n^2}{2^{2\ss}},
    \]
    which is Lemma \ref{lem:apriori}\ref{apriori_rect} for $\ell=1$.
    Similarly, the restriction to odd indices ensures separation between consecutive vertices in $\VV^{(\ell+1)}_i$, thus allowing an application of Lemma \ref{lem:slopebound} (note the condition \eqref{eq:separation-hyp}).
    Hence, since $\Slope(R^{(0)}_1)=\Slope([0,n]^2)=1$, it follows that
    \[
        \frac{1}{1+2\frac{\rr}{\JJ^2}}\le \Slope(R^{(1)}_i)\le 1+2\frac{\rr}{\JJ^2}
    \]
    (see Figure \ref{fig:slope}).
    That $R^{(0)}_1 = [0,n]^2$ contains $R^{(1)}_{1,j}$ is by definition.
    This proves Lemma \ref{lem:apriori}\ref{apriori_slope} for $\ell=1$.

    Suppose now that the conclusions of Lemma \ref{lem:apriori} are true for some $\ell\in\lb1,\MM-1\rb$.
    We will use the argument of the previous paragraph to arrive at the same conclusions for $\ell+1$.
    Fix $R^{(\ell+1)}\in\RR^{(\ell+1)}$.
    By the discussion below \eqref{eq:RR-decomp},
    there exists a unique $i$ such that $R^{(\ell+1)}\subset R^{(\ell)}_i$ 
    (i.e. $R^{(\ell+1)}\in\RR^{(\ell+1)}_i$). 
    Then by \eqref{301}, the bottom-left and top-right corners of $R^{(\ell+1)}$ belong respectively to rectangles $R^{(\ell)}_{i,j_1}$ and $R^{(\ell)}_{i,j_2}$, for some $j_2-j_1\ge 2$ (see Figure \ref{fig:area}).
    In particular, $R^{(\ell+1)}$ contains some $R^{(\ell)}_{i, j}$ and thus
    \begin{align*}\label{398}
        \Bigl|R^{(\ell+1)}\Bigr| \ge \Bigl|R^{(\ell)}_{i,j}\Bigr| 
        \overset{\eqref{308}}{\ge} \JJ^2\frac{n^2}{2^{2(\ell+1)\ss}},
    \end{align*}
    which is just Lemma \ref{lem:apriori}\ref{apriori_rect}.
    By the same reasoning, the corners of the rectangles $R^{(\ell)}_{i,j_1}, R^{(\ell)}_{i,j_2}$ satisfy the hypotheses of Lemma \ref{lem:slopebound}, since by definition $\Slope(R^{(\ell)}_{i,j})=\Slope(R^{(\ell)}_i)$ for all $j$.
    It follows by Lemma \ref{lem:slopebound} that
\begin{equation}\label{relativeslope}
        \frac{1}{1+2\frac{\rr}{\JJ^2}}
        \le \frac{\Slope(R^{(\ell+1)})}{\Slope(R_{i,j}^{(\ell)})}
        \le 1+2\frac{\rr}{\JJ^2}
\end{equation}
    (see Figure \ref{fig:slope}), yielding Lemma \ref{lem:apriori}\ref{apriori_slope}.
\end{proof}

\begin{figure}[ht]
    \centering
    \begin{subfigure}[t]{0.45\linewidth}
        \includegraphics[width=\linewidth]{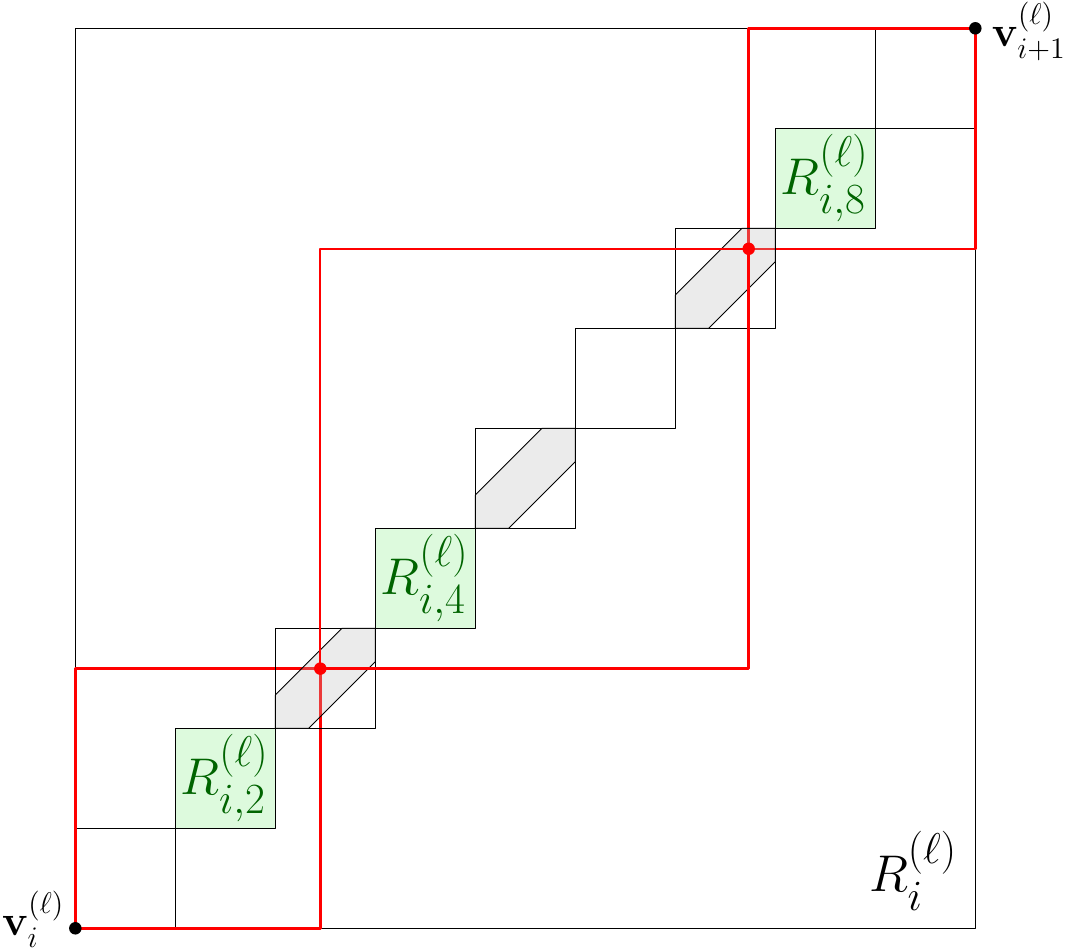}
        \caption{
            Depicted in red are the rectangles belonging to $\RR^{(\ell+1)}_i$.
            In words, these are the rectangles demarcated by consecutive vertices in $\VV^{(\ell+1)}_i\cup\bigl\{
            \vv^{(\ell)}_i,\vv^{(\ell)}_{i+1}
            \bigr\}$.
            Each red rectangle contains some $R^{(\ell)}_{i,j}$ (shaded green).
            This implies a lower bound for the area of each red rectangle.
        }\label{fig:area}
    \end{subfigure}\hfill
    \begin{subfigure}[t]{0.45\linewidth}
        \includegraphics[width=\linewidth]{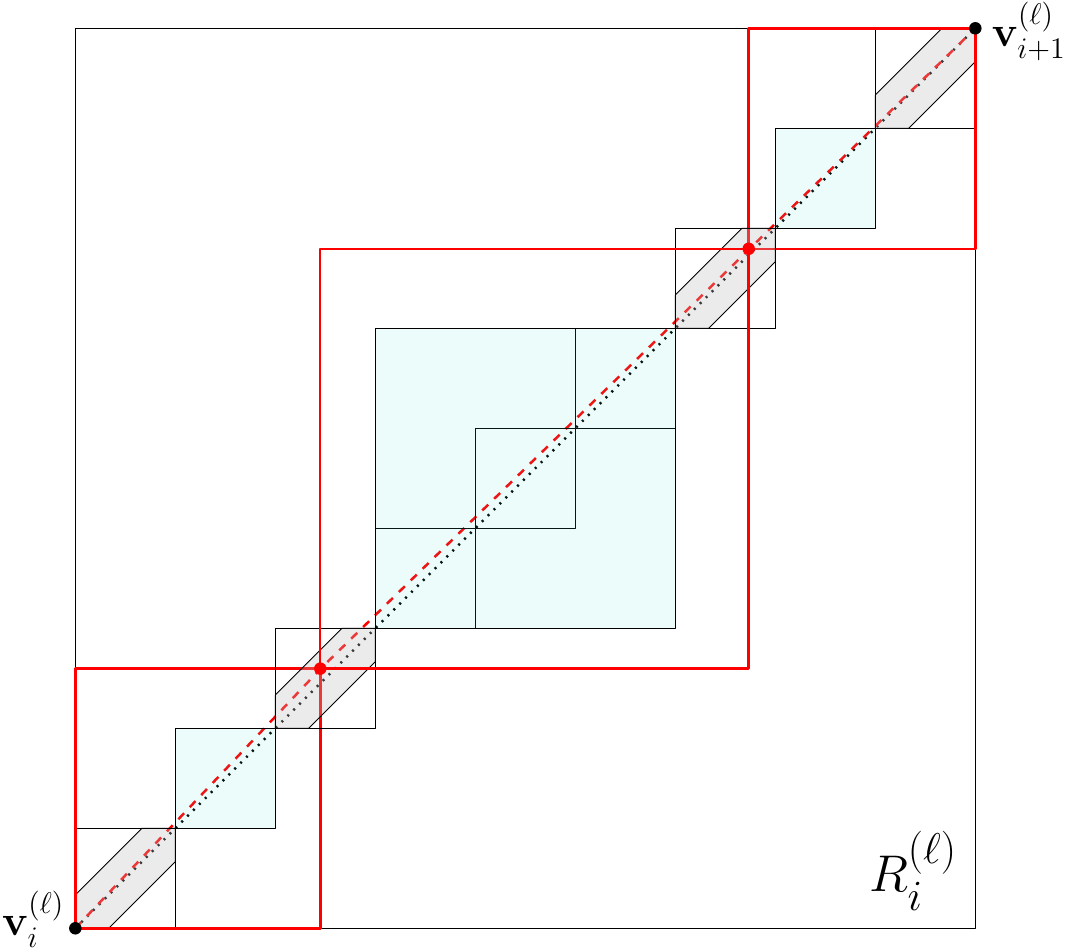}
        \caption{
            The dotted black line is the diagonal of $R^{(\ell)}_i$, and the dashed red lines are the diagonals of the rectangles in $\RR^{(\ell+1)}_i$ (red).
            We want to compare the slopes of these lines via Lemma \ref{lem:slopebound}.
            The bottom-left and top-right corners of each red rectangle lie in cylinders (shaded gray) contained in some $R^{(\ell)}_{i,j_1},R^{(\ell)}_{i,j_2}$ that are separated by a third rectangle (shaded blue) with the same aspect ratio and possibly larger side lengths, as required to apply Lemma \ref{lem:slopebound}.
        }\label{fig:slope}
    \end{subfigure}
    \caption{
        The proof of Lemma \ref{lem:apriori}.
    }\label{fig:apriori}
\end{figure}

\subsection{Vertex sets are large}\label{sec:vertex_sets_large}

In this subsection we prove Proposition \ref{prop:V}\ref{V3}.
For the reader's convenience we recall the parameters defined previously in \eqref{eq:s}, \eqref{eq:rho}:
\begin{align}
    \label{eq:recalled}
    \ss  \coloneqq  100\log\log n,
    \qquad \MM \coloneqq \left\lfloor\frac{\log  n}{10\ss}\right\rfloor,
    \qquad \JJ \coloneqq (\log n)^{1/4},
    \qquad \rr \coloneqq 
    (\log\log n)^{1/2}.
\end{align}
Also recall from \eqref{3095} that 
\begin{align*}
    m_i^{(\ell)}
    \coloneqq\left\lfloor\frac{2^{(\ell+1)\ss}}{\JJ n}\cdot \sqrt{\Bigl|R^{(\ell)}_i\Bigr|}\right\rfloor.
\end{align*}

We begin by establishing the following proposition, which asserts that with high probability a constant fraction of the cylinders $\Cyl_{\rr/\JJ^2}\bigl(R_{i,2j+1}^{(\ell)}\bigr)$ contain a vertex of scale $(\ell+1)\ss$,
for all $\ell\in\lb 0, \MM-1 \rb$ and all $i\in
\lb 1,|\RR^{(\ell)}|\rb$.
\begin{proposition}[Many cylinders contain large weights]\label{prop:largefrac}
    For all sufficiently large $n$,
    \begin{align*}
        \P\biggl(
            \bigl|\VV^{(\ell+1)}_i\bigr|\ge 
            \tfrac{1}{8} m_i^{(\ell)}
            \ \textup{for all} \ \ell\in\lb 0, \MM-1 \rb
            \ \textup{and all}\ i\in\lb 1, |\RR^{(\ell)}| \rb
        \biggr)
        \ge 1-e^{-(\log n)^{97}}.
    \end{align*}
\end{proposition}
The coefficient $\frac18$ holds no particular significance; for example, the proof can be easily modified to establish an identical result but with any $c<\frac12$ in place of $\frac18$.

To prove the above proposition we will reveal the vertex weights one scale at a time.
That is, we will condition on the vertices of scales $k\le\ell\ss$, and then sample conditionally the remaining vertices to check whether they are of scale $(\ell+1)\ss$.
Towards this, for $\ell\in\lb 0,\MM\rb$ we define the set
\begin{align}\label{largeset}
    S_\ell \coloneqq
    \left\{v\in\dmn : 
        X(v) > \frac{n}{2^{\ell\ss}}
    \right\}.
\end{align}
Note that conditional on $S_\ell$, the vertices of scale $(\ell+1)\ss$ in $\dmn$ are uniformly distributed in $\dmn\setminus S_\ell$.
Thus to show that any given cylinder
$\Cyl_{\rr/\JJ^2}\bigl(R_{i,2j+1}^{(\ell)}\bigr)$
contains a vertex of scale $(\ell+1)\ss$ with high probability, it will be useful to first prove that $S_\ell$ typically is not very large, so that after $S_\ell$ is revealed each cylinder still has ample room for vertices of scale $(\ell+1)\ss$.
This is where we will use the fact that $\MM$ was defined in \eqref{eq:s} to be $\floor{\frac{\log n}{10\ss}}$ and not $\floor{\frac{\log n}{\ss}}$.

\begin{lemma}[Most weights are not large]\label{lem:most-not-large}
     There exists $c>0$ such that for all sufficiently large $n$ and all $\ell\in\lb 0,\MM\rb$,
    \[
        \P\bigl(|S_\ell| > n^{0.5}\bigr) \le e^{-c n^{0.5}\log n}.
    \]
\end{lemma}
\begin{proof}
Note that $\frac{n}{2^{\MM\ss}}\ge n^{0.9}$ by definition of $\MM,\ss$ (recalled in \eqref{eq:recalled} above).
Now by the $\al=2$ power-law tail of $X$, we have $\P(X>n^{0.9}) \le Cn^{-1.8}$ for all sufficiently large $n$, for some $C>0$.
Therefore we have the stochastic domination
    \[
        \P\bigl(|S_\ell|>t\bigr)
        \le \P\bigl(|S_\MM|>t\bigr)
        \le \P\Bigl(
            \Bin\bigl((n+1)^2, \;\;Cn^{-1.8}\bigr)
            > t
        \Bigr).
    \] 
    It follows that 
    \begin{align*}
        \P\bigl(|S_\ell|>n^{0.5}\bigr)
        &\le \binom{(n+1)^2}{\ceil{n^{0.5}}}
        \left(Cn^{-1.8}\right)^{\ceil{n^{0.5}}}\\
        &\le \left(\frac{C'n^2 e}{\ceil{n^{0.5}}}\right)^{\ceil{n^{0.5}}}
        \left(Cn^{-1.8}\right)^{\ceil{n^{0.5}}}\\
        &= \left(n^{-0.3 + o(1)}\right)^{\ceil{n^{0.5}}}\\
        &\le e^{-cn^{0.5}\log n},
    \end{align*}
    where we used the standard bound 
    $\binom{a}{b} \le {\left(\frac{ae}{b}\right)}^b$.
    This finishes the proof.
\end{proof}

With the above lemma in place we can now finish the proof of Proposition \ref{prop:largefrac}.

\begin{proof}[Proof of Proposition \ref{prop:largefrac}]
Fix a rectangle $R_i^{(\ell)}\in\RR^{(\ell)}$ and an index $j\in\lb 1, m_i^{(\ell)} \rb$. 
As indicated we will first prove a lower bound on the
number of points in each cylinder $\bigl(\Cyl_{\rr/\JJ^2}(R^{(\ell)}_{i,j})\cap\Z^2\bigr)\setminus S_{\ell}$, where $S_\ell$ was defined in \eqref{largeset}.

We aim to apply Lemma \ref{lem:lattice_count} to the cylinder $\Cyl_{\rr/\JJ^2}(R^{(\ell)}_{i,j})$.
To this end, note that for all $\ell \le M$, the side lengths of the rectangles $R^{(\ell)}_{i,j}$ are at least polynomially large in $n$.
Indeed, by \eqref{308} and Lemma \ref{lem:apriori}\ref{apriori_slope} (area and slope bounds),
each side length is lower bounded by
\begin{align}\label{689}
    \min
    \Biggl\{\sqrt{\Bigl|R_{i,j}^{(\ell)}\Bigr|\cdot\Slope(R_i^{(\ell)})},
    &\quad\sqrt{\Bigl|R_{i,j}^{(\ell)}\Bigr|\cdot\frac{1}{\Slope(R_i^{(\ell)})}}\Biggr\}\nonumber\\
    &\ge \JJ\frac{n}{2^{(\ell+1)\ss}} \cdot \left(1 + 2\frac{\rr}{\JJ^2}\right)^{-\ell/2}
    \nonumber\\
    &\ge \JJ\frac{n}{2^{\MM\ss}} \cdot \left(1 + 2\frac{\rr}{\JJ^2}\right)^{-\MM/2}
    \nonumber\\
    &\ge \JJ  n^{0.9}
    \cdot
    \left(1 + C\frac{(\log\log n)^{1/2}}{(\log n)^{1/2}}\right)^{-C'\frac{\log n}{\log\log n}}
    \nonumber\\
    &=\JJ n^{0.9-o(1)}
    \nonumber\\
    &\ge n^{0.8}
\end{align}
for all sufficiently large $n$,
where the last two inequalities are obtained by substituting in the definitions of $\ss,\MM,\rr,\JJ$ (recalled in \eqref{eq:recalled} above).
Thus, since $\smash{\frac{\rr}{\JJ^2}}$ decays only polylogarithmically in $n$, the condition \eqref{08508} is clearly satisfied and hence we can apply Lemma \ref{lem:lattice_count} to the cylinder $\Cyl_{\rr/\JJ^2}(R^{(\ell)}_{i,j})$.
Doing so, we learn that for all sufficiently large $n$,
\begin{align}\label{690}
\Bigl|\Cyl_{\rr/\JJ^2} (R^{(\ell)}_{i,j})\cap\Z^2\Bigr|
&\overset{\hphantom{\eqref{308}}}{\ge}
\Bigl|\Cyl_{\rr/\JJ^2} (R^{(\ell)}_{i,j})\Bigr|
- 100n\nonumber\\
&\overset{\eqref{eq:area_cyl}}{\ge}\frac{\rr}{\JJ^2}\Bigl|R^{(\ell)}_{i,j}\Bigr|-100n\nonumber\\
&\overset{\eqref{308}}{\ge}\rr\frac{n^2}{2^{2(\ell+1)\ss}}- 100n\nonumber\\
&\overset{\hphantom{\eqref{308}}}{\gs} \rr\frac{n^2}{2^{2(\ell+1)\ss}},
\end{align}
where in the first line we used that the side lengths of $R^{(\ell)}_{i,j}$ are at most $n$ (because $R^{(\ell)}_{i,j}\subset[0,n]^2$), and in the last line we used that $\rr\frac{n^2}{2^{2(\ell+1)\ss}}\ge (\log \log n)^{1/2}\,n^{1.8}$.
By the same reasoning, the lower bound \eqref{690} is much larger than the high probability $n^{0.5}$ upper bound for $|S_\ell|$.

Next we derive a uniform lower bound for the number $m_i^{(\ell)}$ of rectangles $R^{(\ell)}_{i,j}$.
By \eqref{3095} we have
\[
    m_i^{(\ell)} \ge \frac{2^{(\ell+1)\ss}}{\JJ n}\cdot \sqrt{\bigl|R^{(\ell)}_i\bigr|} -1.
\]
For $\ell\ge1$ we have the lower bound $\sqrt{\bigl|R^{(\ell)}_i\bigr|} \ge \JJ \frac{n}{2^{\ell\ss}}$ (Lemma \ref{lem:apriori}\ref{apriori_rect}),
while for $\ell=0$ we have $\sqrt{\bigl|R^{(0)}_1\bigr|}=n$.
Therefore, for all sufficiently large $n$, we have that
\begin{align}\label{256}
    \lfloor\tfrac{1}{2} m_i^{(\ell)}\rfloor-1
    \ge \left\lfloor \frac{2^{\ss-1}}{\JJ} - \frac12 \right\rfloor -1
    \ge
    (\log n)^{99}
\end{align}
for all $\ell\in\lb 0, \MM-1 \rb$ and all $i\in\lb 1, |\RR^{(\ell)}| \rb$, where we plugged in the values $\ss=100\log \log n$ and  $\JJ=(\log n)^{1/4}$ (recall that all logarithms have base $2$ by convention).

We will now analyze the law of the number of cylinders containing a vertex of scale $(\ell+1)\ss$ conditional on $S_\ell$.
First we note that, conditional on $S_\ell$, the weights $X(v)$ with $v\not\in S_\ell$ are i.i.d. with common distribution $\P\bigl(X(v)\in\cdot\,\big|\,X(v)\le\frac{n}{2^{\ell\ss}}\bigr)$.
Under this conditioning, the probability of being of scale $(\ell+1)\ss$ is 
$\P\left(\frac{n}{2^{(\ell+1)\ss}} < X \le \frac{n}{2^{(\ell+1)\ss-1}}\;\Big|\;X\le \frac{n}{2^{\ell\ss}}\right)\ge \tail\, \frac{2^{2(\ell+1)\ss}}{n^2}$ for all sufficiently large $n$.
It follows that
for all sufficiently large $n$,
\begin{align}\label{691}
    \P\biggl(
        \Cyl_{\rr/\JJ^2}(R^{(\ell)}_{i,j})
        \textup{ contains a vertex of scale }
        &(\ell+1)\ss\;\bigg|\;S_\ell
    \biggr)\nonumber\\
    &\overset{\hphantom{\eqref{690}}}{\ge} 1-\left(
        1-\tail\,\frac{2^{2(\ell+1)\ss}}{n^2}
    \right)^{\bigl|\bigl(\Cyl_{\rr/\JJ^2}(R^{(\ell)}_{i,j})\cap\Z^2\bigr)\,\setminus\, S_\ell\bigr|}\nonumber\\
    &\overset{\eqref{690}}{\ge}
    1-\left(
        1-\tail\, \frac{2^{2(\ell+1)\ss}}{n^2}
    \right)^{c'\rr\frac{n^2}{2^{2(\ell+1)\ss}} \,-\, |S_\ell|}.
\end{align}
As indicated above, on the high probability event that $|S_\ell|\le n^{0.5}$, the $-|S_\ell|$ term in the exponent in \eqref{691} can be absorbed into the constant $c'$ (because $\rr\frac{n^2}{2^{2(\ell+1)\ss}} \ge (\log\log n)^{1/2}n^{1.8} \gg n^{0.5}$).
Thus, when $|S_\ell|\le n^{0.5}$, we can lower bound \eqref{691} by
\begin{align}\label{692}
    1 - e^{-c\rr}
    = 1-e^{-c(\log\log n)^{1/2}}.
\end{align}

Next, writing $N \coloneqq \lfloor\frac12 m_i^{(\ell)}\rfloor-1$, we have by definition that
\begin{align}\label{694}
    \bigl|\VV^{(\ell+1)}_i\bigr|
    = \sum_{j=1}^{N}
    \1_{\left\{\Cyl_{\rr/\JJ^2}(R^{(\ell)}_{i,2j+1})  
    \textup{ contains a vertex of scale }
    (\ell+1)\ss\right\}}\,.
\end{align}
By the above estimates 
\eqref{691}, \eqref{692},
it follows that 
\begin{align}\label{693}
    \P\left(\bigl|\VV^{(\ell+1)}_i\bigr| \ge 
    \tfrac{1}{2}N\right)
    &= \E\left[\P\left(\bigl|\VV^{(\ell+1)}_i\bigr| \ge \tfrac{1}{2}N\;\Big|\;S_\ell, N\right)\right]\nonumber\\
    &\ge\E\left[
        \P\Bigl(
            \Bin\bigl(
                N,\;\;
                1-e^{-c(\log\log n)^{1/2}}
            \bigr)
            \ge 
            \tfrac12 N
            \;\Big|\;N
        \Bigr)
        \right]
        - \P\bigl(|S_\ell|> n^{0.5}\bigr),
\end{align}
where we used the fact that the summed indicator functions in \eqref{694} are conditionally independent given $S_\ell$.
Applying the simple estimate
$\P\bigl(\Bin(k, 1-p)\ge \frac{1}{2}k\bigr)\ge 1 - 2^{k}p^{k/2}$ conditionally,
we get that the first term on the RHS of \eqref{693} is lower bounded by
\begin{align*}
    \E\left[
        1-2^N e^{-\frac{c}{2}N(\log\log n)^{1/2}}
    \right]
    &\overset{\hphantom{\eqref{256}}}{\ge} \E\left[
        1-e^{-\frac{c}{4}N(\log\log n)^{1/2}}
    \right]\\
    &\overset{\eqref{256}}{\ge} 1-e^{-\frac{c}{4}(\log n)^{99}(\log\log n)^{1/2}},
\end{align*}
where the first inequality is valid for all sufficiently large $n$.
Finally, using the estimate for $\P\bigl(|S_\ell|> n^{0.5}\bigr)$ from Lemma \ref{lem:most-not-large}, we obtain
\begin{align*}
    \P\left(
        \bigl|\VV^{(\ell+1)}_i\bigr| \ge \tfrac12 N
    \right)
    &\ge 1-e^{-\frac{c}{4} (\log n)^{99}(\log\log n)^{1/2}}
    -e^{-c' n^{0.5}\log n}\\
    &\ge  1-e^{-(\log n)^{99}}.
\end{align*}

Next, to avoid cluttering subsequent notation with additive terms, we will use the crude bound $N=\floor{\frac12 m_i^{(\ell)}}-1 \ge \frac14 m_i^{(\ell)}$, which is valid for all sufficiently large $n$, all $\ell\in\lb 0, \MM-1 \rb$, and all $i\in\lb 1, |\RR^{(\ell)}| \rb$ (e.g. by \eqref{256}).
Then a union bound over $\ell,i$ shows that the probability in the statement of Proposition \ref{prop:largefrac} is lower bounded by
\begin{align*}
    1-(n+1)^4(\log n)\, e^{-(\log n)^{99}} \ge 1-e^{-(\log n)^{97}},
\end{align*}
where we upper bounded the cardinality of $\RR^{(\ell)}$ by the square of the cardinality of $\dmn$, and used that $\MM\le \log n$.
\end{proof}

So far we have not exploited cancellations between the slopes of the
rectangles $R^{(\ell)}\in\RR^{(\ell)}_i$, which was indicated to be crucial in
Section \ref{iop} (see \eqref{thin1}).
This is indeed required to prove Proposition \ref{prop:V}\ref{V3}.
The following lemma will serve as the key input.

\begin{lemma}[Slope cancellations]\label{lem:cancellations}
    The following is true for all sufficiently small $\delta>0$.

    Fix real numbers $x,y>0$.
    Fix also an integer $m\ge1$ and real numbers $x_1,\dots,x_m,y_1,\dots,y_m>0$ satisfying
    \[
        x=\sum_{j=1}^m x_j,\qquad y=\sum_{j=1}^m y_j.
    \]
    Suppose that for all $j\in\lb 1,m\rb$,
    \begin{align*}
        \frac{1}{1+\delta}\cdot \frac{y}{x}\le  \frac{y_j}{x_j} \le (1+\delta)\frac{y}{x}.
    \end{align*}
    Then we have
    \begin{align*}
        \sum_{j=1}^m\sqrt{x_j y_j}\ge  (1- \delta^2) \sqrt{xy}.
    \end{align*}
\end{lemma}
\begin{proof}
    By writing
    \begin{align*}
        x_j = \frac{x_j}{x} x,
        \qquad y_j = \frac{y_j}{y} y,
    \end{align*}
    it suffices to treat the case $x=y=1$.
    The lemma is then an immediate consequence of the following claim.
    \begin{claim}
        For all sufficiently small $\delta>0$, if $z,w>0$ satisfy
        \begin{align*}
            \frac{1}{1+\delta} \le \frac{z}{w}\le 1+\delta,
        \end{align*}
        then
        \begin{align}\label{34198}
            \sqrt{zw}\ge (1-\delta^2)\frac{z+w}{2}.
        \end{align}
    \end{claim}
    We now prove the claim.
    Since the inequalities in the claim are invariant under the rescaling $(z,w)\mapsto(tz,tw)$ for any $t>0$, we can assume  that $w=1$ without loss of generality.
    Then, a Taylor expansion on the left-hand side of \eqref{34198} yields
    \begin{align*}
        \sqrt{z} &= 1 + \tfrac{1}{2}(z-1) - \tfrac{1}{8}(z-1)^2 + O(\d^3)\\
        &\ge 1+\tfrac{1}{2}(z-1) - \tfrac{1}{8}\d^2 + O(\d^3),
    \end{align*}
    where the inequality is obtained by minimizing $-\smash{\frac{1}{8}}(z-1)^2$ over $z\in[\frac{1}{1+\d}, 1+\d]$.
    On the other hand, we can rewrite the right-hand side of \eqref{34198} as 
    \[
        (1-\d^2)\frac{z+1}{2}
        = 1+\tfrac{1}{2}(z-1) - \d^2 - O(\d^3).
    \]
    The claim follows by comparing the last two displays and noting that $-\smash{\frac{1}{8}}\d^2 > -\d^2+O(\d^3)$ for all small $\d>0$.
\end{proof}

This allows us to finish the proof of Proposition \ref{prop:V}\ref{V3}.
\begin{proof}[Proof of Proposition \ref{prop:V}\ref{V3}]
    By definition of $\VV^{(\ell+1)}$ (see \eqref{eq:def_V}), we have that
    \[
        \bigl|\VV^{(\ell+1)}\setminus\VV^{(\ell)}\bigr| = \sum_{i=1}^{|\RR^{(\ell)}|}
        \bigl|\VV^{(\ell+1)}_i\bigr|.
    \]
    It follows by Proposition \ref{prop:largefrac} that
    \[
        \P\left(
            \bigl|\VV^{(\ell+1)}\setminus\VV^{(\ell)}\bigr|
            \ge \frac18\sum_{i=1}^{|\RR^{(\ell)}|}m_i^{(\ell)}
        \quad \text{for all} \ \ell\in\lb 0, \MM-1 \rb    
        \right)
        \ge 1-e^{-(\log n)^{97}}.
    \]
    Therefore, to prove Proposition \ref{prop:V}\ref{V3}, it suffices to establish the existence of a constant $c>0$ such that for all sufficiently large $n$, it is deterministically true that
    \begin{align*}
        \sum_{i=1}^{|\RR^{(\ell)}|}m_i^{(\ell)}
        \ge c\,\frac{2^{(\ell+1)\ss}}{\JJ}\qquad
        \text{for all } \ell\in\lb 0, \MM-1 \rb.
    \end{align*}

    To this end, notice that \eqref{3095} implies the deterministic lower bound
    \[
        \sum_{i=1}^{|\RR^{(\ell)}|}m_i^{(\ell)}
        \ge 
        \frac{2^{(\ell+1)\ss}}{2\JJ n}
        \cdot
        \sum_{i=1}^{|\RR^{(\ell)}|}
        \sqrt{\Bigl|R^{(\ell)}_i\Bigr|},
    \] 
    where we used that $\floor{x}\ge x/2$ for all $x\ge1$ (note that $m^{(\ell)}_i\gg 1$ by \eqref{256}).
    We lower bound the sum on the right-hand side:
    \begin{align*}
        \sum_{i=1}^{|\RR^{(\ell)}|}
        \sqrt{\Bigl|R^{(\ell)}_i\Bigr|}
        &\;\overset{\eqref{eq:Ri},\, \eqref{eq:RR-decomp}}{=}\;
        \sum_{i=1}^{|\RR^{(\ell-1)}|}
        \sum_{R^{(\ell)}_j\in\RR^{(\ell)}_i}\sqrt{\Bigl|R^{(\ell)}_j\Bigr|}\\
        &\;\overset{\hphantom{\eqref{eq:Ri}, \eqref{eq:RR-decomp}}}{\ge}\;
        \sum_{i=1}^{|\RR^{(\ell-1)}|}
        \left(1-4\frac{\rr^2}{\JJ^4}\right)\sqrt{\Bigl|R^{(\ell-1)}_i\Bigr|},
    \end{align*}
    where the inequality is obtained by applying Lemma \ref{lem:cancellations} with $\d=2\frac{\rr}{\JJ^2}$, with $x,y$ given respectively by the horizontal side length and vertical side length of $R^{(\ell-1)}_i$, and with $x_j,y_j$ given respectively by the horizontal and vertical side length of $R^{(\ell)}_j\in\RR^{(\ell)}_i$.
    To see that this application of Lemma \ref{lem:cancellations} is justified, observe that
    $\sum_j x_j = x$ and $\sum_j y_j=y$ by the identity \eqref{eq:Ri2} and the discussion preceding it,
    and observe that we can choose $\smash{\d=2\frac{\rr}{\JJ^2}}$ due to the slope bounds of Lemma \ref{lem:apriori}\ref{apriori_slope} and the fact that $\smash{2\frac{\rr}{\JJ^2}}=o(1)$ 
    (recall  \eqref{eq:s}, \eqref{eq:rho}).

    Now, combining the last two displays and repeatedly applying Lemma \ref{lem:apriori}\ref{apriori_slope} and Lemma \ref{lem:cancellations}, we get
    \begin{align*}
        \sum_{i=1}^{|\RR^{(\ell)}|}
        m_i^{(\ell)}
        &\overset{\hphantom{\eqref{eq:base}}}{\ge}\; 
        \frac{2^{(\ell+1)\ss}}{2\JJ n}
        \left(1-4\frac{\rr^2}{\JJ^4}\right)^{\ell} 
        \sqrt{\Bigl|R_1^{(0)}\Bigr|}\\
        &\overset{\eqref{eq:base}}{=}\;
        \frac{2^{(\ell+1)\ss}}{2\JJ}
        \left(1-4\frac{\rr^2}{\JJ^4}\right)^{\ell}.
    \end{align*}
    Finally, using that $\JJ=(\log n)^{1/4}$, that $\rr=(\log\log n)^{1/2}$, and that $\ell\le\MM\le C\frac{\log n}{\log\log n}$,
    we obtain
    \begin{equation}\label{160983}
        \sum_{i=1}^{|\RR^{(\ell)}|}
        m_i^{(\ell)}
        \ge 
        \frac{2^{(\ell+1)\ss}}{2\JJ}
        \left(
            1-4\frac{\log \log n}{\log n}
        \right)^{C\frac{\log n}{\log\log n}}
        \ge c\,
        \frac{2^{(\ell+1)\ss}}{\JJ}
    \end{equation}
    for all sufficiently large $n$.
    This completes the proof of Proposition \ref{prop:V}\ref{V3}.
\end{proof}

In the next subsection we show how a straightforward adaption of the proof of Theorem \ref{thm:main} yields Theorem \ref{thm:secondmoment}.
\subsection{A distribution with finite second moment and superlinear last passage time}\label{sec:secondmoment}
As discussed just above (see also Section \ref{iop}), the key insight underlying the proof
of Theorem \ref{thm:main} is that accounting for the cancellations between the slopes of
rectangles at a given scale allows us to repeat the inductive construction underlying
Proposition \ref{prop:V} $\log n$ many times before losing control over the slopes, whereas
naively applying the worst-case slope bound at every scale would lead to a loss
of control after only $(\log n)^{1/2}$ many repetitions.
In what follows we will show that this improvement over the naive 
$(\log n)^{1/2}$ also allows us to construct a weight distribution with
finite second moment and superlinear last passage time, as asserted in Theorem \ref{thm:secondmoment}.

To make the distribution have a finite second moment we add a logarithmic factor in the tail.
Thus, we fix a non-negative random variable $X$ for which there exist constants $\beta \in (1,\,\frac32)$ and $C>0$ such that
as $t\to\infty$,
\[
    \P(X>t)\sim Ct^{-2}(\log t)^{-\beta}.
\]
Note that $\E[X^2] < \infty$, since
\begin{align*}
    \E[X^2]
    = \int_0^\infty 2t\, \P(X>t)\,\mrm{d}t 
    \ls
    \int_{2}^\infty \frac{1}{t(\log t)^{\beta}}\,\mrm{d}t 
    <\infty.
\end{align*}
Theorem \ref{thm:secondmoment} is therefore an immediate corollary of the following result.
\begin{proposition}\label{prop:real-secondmoment}
Consider LPP where the weights are i.i.d. copies of the above random variable $X$. 
Then there exists $c>0$ such that for all sufficiently large $n$, 
    \begin{align*}
        \P\left(L_n \ge c\,\frac{n(\log n)^{\frac{3}{4} - \frac{\beta}{2}}}{\log\log n}\right)
        \ge 1-e^{-(\log n)^{97}}.
    \end{align*}
\end{proposition}
\begin{proof}
The arguments are the same as in the proof of Theorem \ref{thm:main}, so we will be brief to avoid repetition. 
As in \eqref{eq:s}, we write
\[
    \ss \coloneqq 100\log \log n
    \qquad \MM \coloneqq \left\lfloor\frac{\log  n}{10\ss}\right\rfloor,
    \qquad\JJ \coloneqq (\log n)^{1/4}.
\]
For $k\ge0$, we say that a vertex $v$ is of \emph{scale $k$} if 
\begin{align*}
    \frac{n}{2^k(\log\frac{n}{2^k})^{\beta/2}}
    < X(v) \le
    \frac{n}{2^{k-1}(\log\frac{n}{2^k})^{\beta/2}}.
\end{align*}
A straightforward calculation reveals that there exists $c>0$ such that for all sufficiently large $n$ and all $k\le \frac{1}{10}\log n$,
\begin{align*}  
    \P(v \ \text{is of scale}\ k)  \ge c\,\frac{2^{2k}}{n^2}.
\end{align*}
One can check that the proof of Proposition \ref{prop:V} (multi-scale construction) simply relied on the tail estimate \eqref{eq:tail}.
Since that estimate continues to hold here the conclusions of Proposition \ref{prop:V} remain valid in the current setting as well.
We reuse the notation $\VV^{(0)}\subset\VV^{(1)}\subset\cdots\subset\VV^{(\MM)}$ for the random vertex sets thus constructed.

We can now argue as in the proof of Theorem \ref{thm:main} (see below Proposition \ref{prop:V}).
By Proposition \ref{prop:V} and the non-negativity of the weights, for all sufficiently large $n$ the following estimates hold with probability at least $\smash{1-e^{-(\log n)^{97}}}$:
    \begin{align*}
        L_n
        \ge \sum_{v\in\VV^{(\MM)}} X(v)
        &\ge \sum_{\ell=1}^\MM \sum_{v\in\VV^{(\ell)}\setminus\VV^{(\ell-1)}} X(v)\\
        &\gs \sum_{\ell=1}^\MM \frac{2^{\ell\ss}}{\JJ} \cdot \frac{n}{2^{\ell\ss}(\log \frac{n}{2^{\ell\ss}})^{\beta/2}}\\
        &\ge \MM\cdot \frac{n}{\JJ(\log n)^{\beta/2}}\\
        &\gs \frac{n(\log n)^{\frac{3}{4} - \frac{\beta}{2}}}{\log\log n}.
    \end{align*}
This finishes the proof.
\end{proof}
\noindent Observe that, as alluded to earlier, the exponent $3/4$ in Theorem \ref{thm:main} being strictly greater than $1/2$ was crucial in the construction of the above example: otherwise, we could not have deduced superlinear growth for any $\beta>1$.\\

Next we prove our results on concentration and fluctuations of the $\al=2$ heavy-tailed last passage time, namely Theorem \ref{thm:conc} and a fluctuation lower bound in Proposition \ref{prop:fluc-lower-bound} as indicated in Remark \ref{lower1234}.

\section{Concentration and fluctuations}\label{sec:conc}
We prove Theorem \ref{thm:conc} using the Efron--Stein inequality (e.g. \cite[Theorem 3.1]{BLMConcentrationInequalitiesNonasymptotic2013}):
\begin{proposition}[Efron--Stein inequality]\label{efron-stein}
    Let $Z=f(W_1,\dots,W_q)$ be a square-integrable function of independent random variables $W_1,\dots,W_q$.
    For $i\in\lb 1,q\rb$, let
    \[
        Z'_i \coloneqq f(W_1,\dots,W_{i-1},W_i',W_{i+1},\dots,W_q),
    \]
    where $W_i'$ is an independent copy of $W_i$. 
    Then
    \[
        \Var(Z) \le \sum_{i=1}^q \E\Bigl[\max\bigl\{Z-Z'_i,\, 0\bigr\}^2\Bigr].
    \]
\end{proposition}
We briefly sketch the proof of the Efron--Stein inequality.
For $i\in\lb 1,q\rb$ let $\cF_i\coloneqq \sigma(W_1,\dots,W_i)$, and let $\cF_0$ be the trivial $\sigma$-algebra.
Denote by $\E_i$ the conditional expectation given $\cF_i$.
By $L^2$-orthogonality of martingale increments,
\begin{align*}
    \Var(Z) = \sum_{i=1}^q \E\Bigl[
        \Bigl(
            \E_i[Z]
            -\E_{i-1}[Z]
        \Bigr)^2
    \Bigr].
\end{align*}
Writing $\E^{(i)}$ for the conditional expectation given $\sigma(W_j:j\ne i)$, we have by independence
\begin{align*}
    \E_{i-1}[Z]
    = \E_i\bigl[
        \E^{(i)}[Z]
        \bigr].
\end{align*}
Applying Cauchy--Schwarz to $\E_i$, we get
\begin{align*}
    \E\Bigl[
        \Bigl(
            \E_i[Z]
            -\E_{i-1}[Z]
        \Bigr)^2
    \Bigr]
    \le \E\Bigl[
        \Bigl(Z-\E^{(i)}[Z]\Bigr)^2
    \Bigr]
    = \E\Bigl[\Var^{(i)}(Z)\Bigr]
    =
    \E\Bigl[\max\bigl\{Z-Z'_i,\, 0\bigr\}^2\Bigr],
\end{align*}
from which the Efron--Stein inequality follows.

Note that the above argument applies verbatim in the more general setting where each $W_i$ is a random element of some measurable space $\Omega_i$, and $W_1,\dots,W_q$  are independent.
This observation will be used in the proof of Proposition \ref{prop:fluc-lower-bound} below, where we will control the variance of a Poissonian hierarchical model by resampling Poisson point processes restricted to small boxes.

We proceed now to the proof of Theorem \ref{thm:conc}.

\begin{proof}[Proof of Theorem \ref{thm:conc}]
    Fix $K\ge 1$, to be specified later, and let $\wt{L}_n$ be the last passage time from $(0,0)$ to $(n,n)$ with respect to the truncated weights
    \[
        \wt{X}(v) \coloneqq X(v)\1_{X(v)\le K^2n}.
    \]
    Note that $\P\bigl(X(v)\ne \wt{X}(v)\bigr)\ls \frac{1}{K^4n^2}$.
    We will now show that $\Var(\wt{L}_n) \ls K^2n^2$
    by viewing $\wt{L}_n$ as a function of the weights $\bigl(\wt{X}(v)\bigr)_{v\in\dmn}$ and applying the Efron--Stein inequality.

    For any vertex $v\in\dmn$ we denote by $\wt{L}'_{n,v}$ the last passage time obtained by replacing the weight $\wt{X}(v)$ with an independent copy $\wt{X}'(v)$.
    Suppose that $\wt{L}_n - \wt{L}'_{n,v} \ge 0$.
    Then, writing $\Gamma$ for a geodesic with respect to $\wt{L}_n$
    (Definition \ref{def:lpp}), we have that
    \begin{align}\label{resample}
        0\le \wt{L}_n-\wt{L}_{n,v}'
        &\le \sum_{w\in\Gamma}\wt{X}(w)
        - \left(\sum_{w\in\Gamma\setminus\{v\}}\wt{X}(w) + \wt{X}'(v)\1_{v\in\Gamma}\right)
        \le \wt{X}(v) \1_{v\in\Gamma},
    \end{align}
    where for future notational simplicity we used the bound $\wt{X}(v)-\wt{X}'(v)\le \wt{X}(v)$ (since $\wt{X}'(v)$ is non-negative).
    Combining the above estimate with the Efron--Stein inequality, we get
    \begin{align*}
        \Var(\wt{L}_n)
        &\le \sum_{v\in\dmn}\E\Bigl[\max\bigl\{\wt{L}_n-\wt{L}'_{n,v},\,0\bigr\}^2\Bigr]\\
        &\le \sum_{v\in\dmn}\E\Bigl[\wt{X}(v)^2\1_{v\in\Gamma}\Bigr]\\
        &\le\E\left[\max_{\pi}\sum_{v\in\pi}\wt{X}(v)^2\right],
    \end{align*}
    where the maximum in the last line is over all up-right directed lattice paths from $(0,0)$ to $(n,n)$.
    The last line is the expected last passage time with respect to the squared weights $\wt{X}(v)^2$, which we upper bound next.
    A general upper bound was obtained in \cite{MarLinearGrowthGreedy2002}
    relying the fact that for LPP with i.i.d. $\Ber(p)$ weights,
    the expected last passage time from $(0,0)$ to $(n,n)$ is $O(np^{1/2})$.
    Using this and decomposing the weights into scales (e.g. dyadically), one gets
    \begin{align*}
        \E\left[\max_{\pi}\sum_{v\in\pi}\wt{X}(v)^2\right]
        &\ls 
        n \int_0^\infty \P\Bigl(\wt{X}(v)^2 > t\Bigr)^{1/2}\,\mrm{d}t
    \end{align*}
    (for details, see the proof of \cite[Theorem 2.3]{MarLinearGrowthGreedy2002}).
  It is straightforward to see that 
    \begin{align*}
        \int_0^\infty \P\Bigl(\wt{X}(v)^2 > t\Bigr)^{1/2}\,\mrm{d}t
        &=\int_0^{(K^2n)^2} \P\Bigl(X(v)>t^{1/2}\Bigr)^{1/2}\,\mrm{d}t\\
        &\ls \int_0^{(K^2n)^2}\frac{1}{t^{1/2}}\,\mrm{d}t\\
        &=K^2n.
    \end{align*}
    Combining the last three displays shows that 
    \begin{equation}\label{varbound}
    \Var(\wt{L}_n) \ls K^2n^2
    \end{equation}
    as claimed.
    Then, by Chebyshev's inequality, there exists $C>0$ such that for all $t>0$ and all $K,n\ge1$,
    \begin{align*}
        \P\left(\left|\frac{\wt{L}_n-\E[\wt{L}_n]}{Kn}\right| > t\right)
        \le \frac{C}{t^2}.
    \end{align*}

It remains to compare  $\smash[b]{\frac{\wt L_n-\E[\wt{L}_n]}{Kn}}$ to  $\smash[b]{\frac{L_n-\E[L_n]}{Kn}}$.
    By a union bound,
    \begin{equation*}
        \begin{split}
            \P\bigl(L_n\ne \wt{L}_n\bigr)
        &\le \P\bigl(
                \text{there exists} \ v\in \dmn \ \text{such that} \ X(v)>K^2n
            \bigr)
        \ls \frac{1}{K^4}.
        \end{split}
    \end{equation*}
    Next we compare the means $\E[L_n]$ and $\E[\wt{L}_n]$:
    \begin{align*}
        0 \le \E[L_n]-\E[\wt{L}_n]
        &\le \E\left[
            \max_{\pi}\sum_{v\in\pi}X(v)\1_{X(v)>K^2n}
        \right]\\
        &\le \sum_{k=\floor{\log (K^2n)}}^\infty 
        \E\left[\max_{\pi}\sum_{v\in\pi}X(v)\1_{X(v)\in[2^k, 2^{k+1}]}\right]\\
        &\le \sum_{k=\floor{\log (K^2n)}}^\infty 2^{k+1}\cdot \E\biggl[\Bigl|\Bigl\{
            v\in\dmn : X(v) \in [2^k, 2^{k+1}]
            \Bigr\}\Bigr|\biggr]\\
        &\le C\sum_{k=\floor{\log (K^2n)}}^\infty 2^{k+1}\cdot \frac{n^2}{2^{2k}}\\
        &\le C \frac{n}{K^2},
    \end{align*}
    where the constant $C$ does not depend on $K$.
    It follows that there exists $t_0\ge 1$ such that for all $t\ge t_0$ and all $n,K\ge 1$,
    \begin{align*}
        \P\left(\left|\frac{L_n-\E[L_n]}{Kn}\right| > t\right)
        &\le 
        \P\left(\left|\frac{\wt{L}_n-\E[\wt{L}_n]}{Kn}\right| > t-\frac{C}{K^3}\right)
        + \P\bigl(L_n\ne \wt{L}_n\bigr)\\
        &\ls \frac{1}{t^2} + \frac{1}{K^4}.
    \end{align*}
    Setting $K=t^{1/2}$ and making the change of variables $t^{3/2}\mapsto t$ completes the proof of Theorem \ref{thm:conc}.
\end{proof}

We end this section with a discussion on fluctuation lower bounds, including the upcoming Proposition \ref{prop:fluc-lower-bound}.
Recall from Remark \ref{lower1234} that we expect the true fluctuation order of $L_n$ to be lower than $n$ by a polylogarithmic correction.
Note that this is consistent with the following simple observation:
\begin{align}\label{eq:lower-via-max}
    \P\bigl(L_n-\E[L_n]\ge n\log n\bigr)
    &\ge \P\left(\max_{v\in\dmn} X(v) \ge Cn\log n\right)
    \gs \frac{1}{(\log n)^2},
\end{align}
where we used that $\E[L_n]=O(n\log n)$ (recall Section \ref{iop}) and  $C>0$ is taken to be sufficiently large.
Note that except for \eqref{eq:lower-via-max}, in our analysis of the $\al=2$ heavy-tailed LPP model we have completely ignored the fluctuation behavior coming from weights taking values $\gg n$, but rather have focused on the underlying hierarchical structure.
Returning our focus to the latter, the upcoming proposition shows that a lower bound 
consistent with the above 
holds even after conditioning on
$\max_{v\in\dmn}X(v)  = O(n)$ (note that this is the typical behavior of the maximum).
It is in this regime that the fluctuation behavior of $L_n$ should approximate that of the hierarchical Poissonian model from Section \ref{iop}, so for simplicity we focus on the latter in the rest of this discussion.
Thus, we consider independent Poisson point processes $\omega^{(0)},\omega^{(1)},\dots$ where $\omega^{(k)}$ is of intensity $\frac{2^{2k}}{n^2}$, and define the last passage time by
\[
    \cL_n\coloneqq\max_\pi \sum_{k=0}^{\log  n}
    \frac{n}{2^k}\bigl|\omega^{(k)}\cap\pi\bigr|,
\]
where the maximum is over continuous increasing paths $\pi$ from $(0,0)$ to $(n,n)$.

\begin{proposition}\label{prop:fluc-lower-bound}
    There exists $C>0$ such that for all sufficiently large $n$,
    \[
        \P\bigl(\cL_n - \E[\cL_n] \ge n\bigr) \ge \frac{1}{(\log n)^C}.
    \]    
\end{proposition}
The guiding intuition in the proof of the above, as also in the upcoming analysis for the branching random walk model, is that the $k=0$ scale alone contributes an order $n$ amount to the fluctuations, as the points have weight $n$ and thus their inclusion or 
exclusion in a path changes its passage time by $n$. However since we expect the geodesic to not completely delocalize but instead deviate from the diagonal by only an $O\left(\frac{n}{\polylog n}\right)$ amount (recall Remark \ref{lower1234}), the probability that the points of weight $n$ land close enough to the geodesic for the latter to collect them is only $\frac{1}{\polylog n}$, leading to a probability lower bound of that order.

\begin{proof}[Proof of Proposition \ref{prop:fluc-lower-bound}]
Consider the last passage time defined in the same manner as $\cL_n$, but without the $0$\textsuperscript{th} scale:
\[
    \wh{\cL}_n \coloneqq \max_\pi \sum_{k=1}^{\log  n} \frac{n}{2^k}\bigl|\omega^{(k)}\cap\pi\bigr|.
\]
We will show that $\wh{\cL}_n\ge \E[\cL_n]-Cn$ and $\cL_n\ge \wh{\cL}_n + (C+1)n$ simultaneously with probability at least $\frac{1}{\polylog n}$.
First, notice that 
\begin{align*}
    \cL_n \ge \wh{\cL}_n \ge \cL_n - n\cdot |\o^{(0)}\cap[0,n]^2|,
\end{align*}
and thus, since $|\o^{(0)}\cap[0,n]^2|$ is a Poisson random variable of mean $1$,
\begin{align*}
    \E[\cL_n]
    \ge \E[\wh{\cL}_n]
    \ge \E[\cL_n] - n.
\end{align*}
A similar argument as in \eqref{varbound} using the Efron--Stein inequality implies that $\Var(\wh{\cL}_n)\ls n^2$.
Postponing the details of this variance bound to the next paragraph, it follows by Chebyshev's inequality that there exists $C_0>0$ such that
\begin{align}\label{08084308}
    \P\Bigl(\wh{\cL}_n \ge \E[\cL_n] - C_0n\Bigr) \ge \frac12
\end{align}
for all sufficiently large $n$.

We now show that $\Var(\wh{\cL}_n)\ls n^2$.
Partition $[0,n]^2$ into $1\times 1$ boxes.
For each $1\times 1$ box $B$ and each $k\in\lb 1,\log n\rb$, let $\wh{\cL}_{n,k,B}'$ be the last passage time obtained by replacing the restricted point process $\o^{(k)}\cap B$ with an i.i.d. copy.
Let $\wh{\Gamma}$ be a geodesic with respect to $\wh{\cL}_n$.
Arguing as in \eqref{resample}, on the event that $\wh{\cL}_{n} - \wh{\cL}_{n,k,B}'\ge 0$, we have the bounds
\begin{align*}
    0 \le \wh{\cL}_{n}-\wh{\cL}_{n,k,B}'
    \le \frac{n}{2^k}\left|\o^{(k)} \cap B\cap \wh{\Gamma}\right|.
\end{align*}
It follows by the Efron--Stein inequality (see the discussion below Proposition \ref{efron-stein}) that
\begin{align*}
    \Var(\wh{\cL}_{n})
    &\le 
    \sum_{k=1}^{\log n}
    \frac{n^2}{2^{2k}}
    \sum_{B} 
    \E\left[
        \left|\o^{(k)} \cap B\cap \wh{\Gamma}\right|^2
    \right],
\end{align*}
where the inner sum is over all $1\times 1$ boxes $B$ partitioning $[0,n]^2$.
We bound the RHS as follows:
\begin{align*}
    \sum_{k=1}^{\log n}
    \frac{n^2}{2^{2k}}
    \sum_{B} 
    \E\left[
        \left|\o^{(k)} \cap B\cap \wh{\Gamma}\right|^2
    \right]
    &\le
    \sum_{k=1}^{\log n}
    \frac{n^2}{2^{2k}}
    \E\left[
        \left(\sum_B \left|\o^{(k)}\cap B\cap \wh{\Gamma}\right|\right)^2
    \right]\\
    &= \sum_{k=1}^{\log n}
    \frac{n^2}{2^{2k}}
    \E\left[
        \left|\o^{(k)}\cap \wh{\Gamma}\right|^2
    \right]\\
    &\le \sum_{k=1}^{\log n}
    \frac{n^2}{2^{2k}}
    \E\left[
        \max_{\pi}\left|\o^{(k)}\cap \pi\right|^2    
    \right].
\end{align*}
By the large deviations estimates for Poissonian LPP discussed in Section \ref{iop}, the above second moment is $O(2^k)$ for all $k\in\lb 1,\log n\rb$.
It follows that $\Var(\wh{\cL}_n) \ls \sum_{k=1}^{\log n} \frac{n^2}{2^k} \le n^2$ as claimed.

Next we will estimate the contribution to $\wh{\cL}_n$ coming from small boxes.
Write $s\coloneqq \frac{n}{\log n}$ and partition $[0,n]^2$ into $s\times s$ boxes 
whose sides are parallel to the coordinate axes.
We now show that with high probability, every such box contributes at most $O(n)$
to the last passage time $\wh{\cL}_n$.
Let $R$ be any such box.
Note that by directedness and the non-negativity of the weights, the sum of weights collected by the geodesic of $\wh{\cL}_n$ in $R$ is upper bounded by
\[
    \wh{\cL}_R \coloneqq \max_\pi \sum_{k=1}^{\log  n}
    \frac{n}{2^k} \bigl|\omega^{(k)}\cap \pi\bigr|,
\]
where the maximum is over continuous increasing paths between the bottom-left
and top-right corners of $R$.
It therefore suffices to upper bound $\wh{\cL}_R$.
Towards this, note that for any $m\le \log n$,
\[
    \wh{\cL}_R \le \sum_{k=1}^{m-1}\underbrace{\max_\pi \frac{n}{2^k}\bigl|\omega^{(k)}\cap\pi\bigr|}_{\wh{\cL}_R^{(k)}}
    + \underbrace{\max_\pi \sum_{k=m}^{\log n} \frac{n}{2^k}\bigl|\omega^{(k)}\cap  \pi\bigr|}_{\wh{\cL}_R^{(\ge m)}}.
\]
Choosing $m\coloneqq\log\log n$, we will show that both terms are $O(n)$ using different arguments.
For notational simplicity we assume that $m\in\N$ (the general case follows by rounding).
Note that ${\wh{\cL}_R^{(\ge m)}}$ has the exact same distribution as $\cL_s$, i.e. the last passage time between $(0,0)$ and $(s,s)$ with $\log s$ many layers of the Poisson noise present (recall that all logarithms have base $2$ by convention).
Thus using the same arguments as in \eqref{upper1234} with $n$ replaced by $s$, we get that
\[
    \P\left(
        \wh{\cL}^{(\ge m)}_R
        > Cs\log s 
    \right)
    \le e^{-c\log s}.
\]
Then since $s\log s =O(n)$ and $\log n = O(\log s)$,
we get that
\[
    \P\left(
        \wh{\cL}^{(\ge m)}_R
        > Cn 
    \right)
    \le e^{-c \log n}.
\]
Now a union bound over the $(\log  n)^2$ many boxes $R$ shows that, with high probability, $\wh{\cL}^{(\ge m)}_R\ls n$ for all $R$ simultaneously.

It remains to bound the contribution from the scales $k\in\{1,\dots,m-1\}$.
Consider $\bigl|\omega^{(k)}\cap R\bigr|$ which is a Poisson random variable of mean
$\frac{2^{2k}}{n^2}\cdot\frac{n^2}{(\log  n)^2} = 2^{2(k-m)}$.
Since $k-m<0$, we have the simple tail estimate
\begin{align*}
    \P\left(\bigl|\omega^{(k)}\cap R\bigr| \ge i\right)
    &= e^{-2^{2(k-m)}}\sum_{j=i}^\infty\frac{2^{2j(k-m)}}{j!}
    \ls 2^{2i(k-m)}\quad \text{for all } i\ge 1.
\end{align*}
In particular,
\[
    \P\left(\wh{\cL}^{(k)}_R \ge \frac{n}{2^k}\cdot \frac{2m}{m-k}\right)
    \le \P\left(\bigl|\omega^{(k)}\cap R\bigr| \ge \frac{2m}{m-k}\right)
    \ls 2^{-4m}.
\]
Summing this bound over $k\in\{1,\dots,m-1\}$ and over all $(\log  n)^2=2^{2m}$ many boxes $R$ yields
\[
    \P\left(\text{there exists } R \text{ such that } 
        \sum_{k=1}^{m-1}\wh{\cL}^{(k)}_R 
    \ge \sum_{k=1}^{m-1}\frac{n}{2^k}\cdot \frac{2m}{m-k}\right)
    \ls m2^{-2m} = o(1).
\]
Next, observe that
\begin{align*}
    \sum_{k=1}^{m-1}\frac{n}{2^k}\cdot \frac{2m}{m-k}
    &= \frac{2nm}{2^m}\sum_{k=1}^{m-1}\frac{2^{m-k}}{m-k}
    \ls \frac{nm}{2^m} \cdot \frac{2^{m-1}}{m-1}
    \le n.
\end{align*}
We conclude that $\wh{\cL}_R \le C_1n$ for all $R$ simultaneously with probability $1-o(1)$, for some constant $C_1>0$.
This fact, along with \eqref{08084308}, implies that the event
\[
    \cE\coloneqq\bigl\{\wh{\cL}_R \le C_1n\text{ for all partition boxes } R\bigr\}\cap\bigl\{\wh{\cL}_n \ge \E[\cL_n] - C_0n\bigr\}
\]
satisfies $\P(\cE)\ge \frac12-o(1) \ge \frac14$ for all sufficiently large $n$.

Next we consider the effect of the $0$\textsuperscript{th} scale.
We again partition $[0,n]^2$ into $s\times s$ boxes (recall that $s= \frac{n}{\log n}$).
Let $\wh{\Gamma}$ be a geodesic with respect to $\wh{\cL}_n$.
One can show that, simply by virtue of $\wh{\Gamma}$ being a continuous increasing path from $(0,0)$ to $(n,n)$, there must exist two points $a,b\in\wh\Gamma$ with $a\prec b$ 
such that the rectangle $R_*\coloneqq\Rect(a,b)$ has area at least $s^2/4$ 
and intersects at most four of the $s\times s$ partition boxes (see Figure \ref{fig:lowerbound}).
Then, on the event $\cE$, the Poisson points of scales $k\ge 1$ in $R_*$ together contribute weight at most $4C_1n$ to the last passage time $\wh{\cL}_n$.
On the other hand, using that $\omega^{(0)}$ is independent of $R_*, \cE$, and using that $|R_*|\ge s^2/4$, we obtain
(writing $J\coloneqq \ceil{C_0+4C_1+1}$)
\begin{align*}
    \P\Bigl(\bigl|\omega^{(0)}\cap R_*\bigr| = J,\quad \omega^{(0)}\cap R_* \text{ is ordered},\quad \cE\Bigr)
    &\ge \frac{1}{J!}\, \P\left(\Poi\left(\frac{1}{n^2}\cdot \frac{s^2}{4}\right) = J\right)\cdot \P(\cE) \\
    &\ge \frac{1}{(J!)^2}\,e^{-2^{-2(m+1)}}2^{-2J(m+1)}\cdot \frac14\\
    &\ge \frac{1}{(\log n)^{C}}
\end{align*}
for some $C>0$.
It follows that with probability at least $\frac{1}{\polylog n}$, the geodesic $\wh{\Gamma}$ can be rerouted to collect $J$ many points from $\omega^{(0)}$ thereby increasing its passage time by $Jn$, at the cost of possibly decreasing its passage time coming from the scales $k\ge1$ by at most $4C_1n$, and simultaneously $\wh{\cL}_n \ge \E[\cL_n]-C_0n$.
In other words, with probability at least $\frac{1}{\polylog n}$ the following estimates hold:
\begin{align*}
    \cL_n &\ge \wh{\cL}_n + Jn-4C_1n\\
    &\ge \E[\cL_n]-C_0n + Jn-4C_1n\\
    &\ge \E[\cL_n] + n.
\end{align*}
This proves the proposition.
\end{proof}
\begin{figure}[ht]
    \centering
    \includegraphics[width=0.4\textwidth]{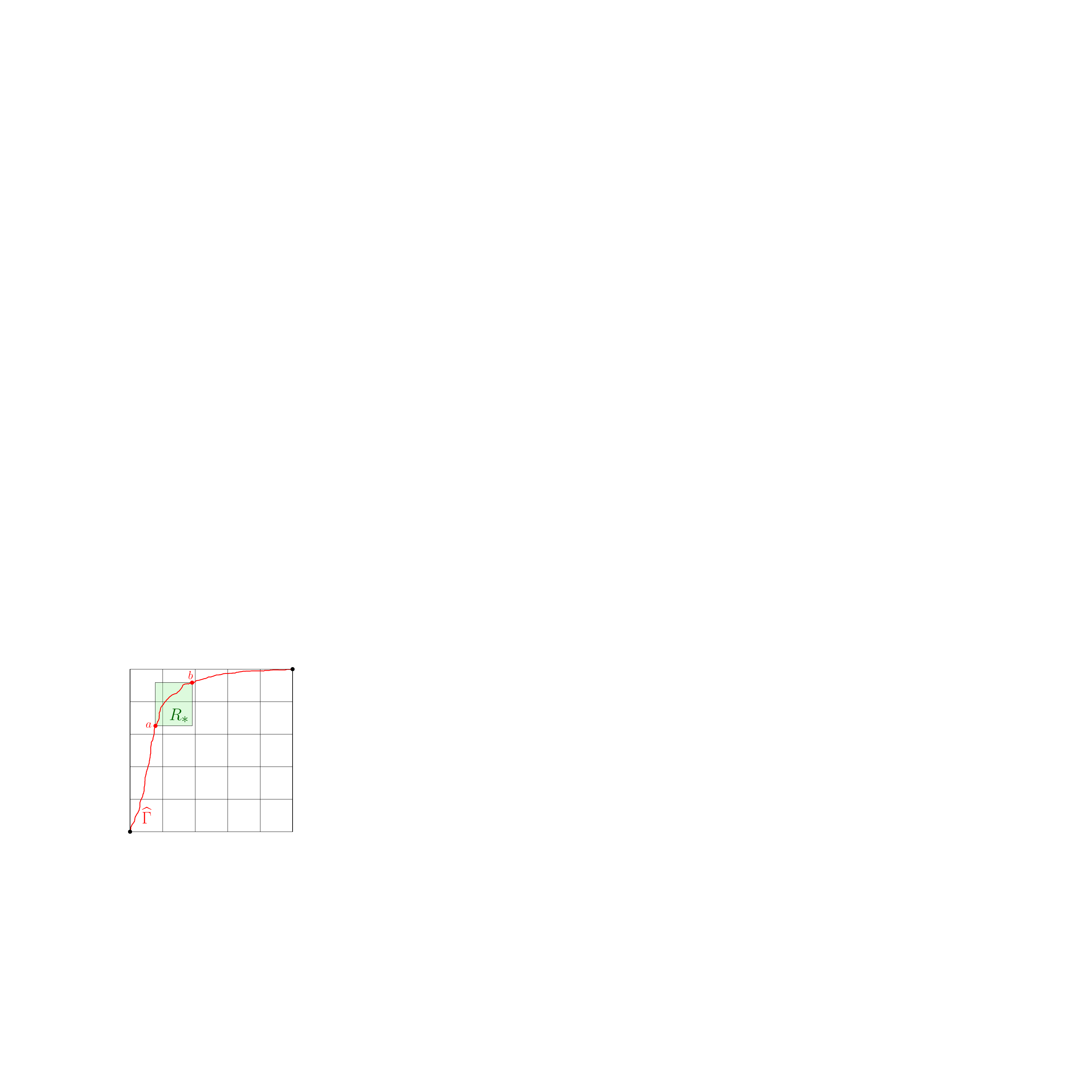}
    \caption{Depicted is $[0,n]^2$ (large black square) partitioned into $s\times s$ boxes (smaller black squares).
    Since the geodesic $\wh{\Gamma}$ (red) is directed, there exist points $a,b\in\wh{\Gamma}$ with $a\prec b$ such that the rectangle $R_*\coloneqq\Rect(a,b)$ (green) has area at least $s^2/4$ and is contained in at most four of the $s\times s$ partition boxes. }\label{fig:lowerbound}
\end{figure}

Next we will implement our multi-scale approach for the branching random walk and prove Theorems \ref{thm:brw} and \ref{thm:conc-brw}.
Along the way we will highlight some of the differences from the hierarchical structure present in the $\al=2$ heavy-tailed environment.

\section{Last passage percolation on the branching random walk}\label{sec:brw}

Throughout this section we use the notation of Section \ref{sec:brw0}.
Our first aim is to prove Theorem \ref{thm:conc-brw}, which recall asserts that $\LBRW$ has sub-Gaussian fluctuations of order $\Theta(n)$.
We will use the following standard results on maxima of Gaussian processes:
\begin{proposition}[Fluctuations of maxima of Gaussian processes]
    \label{borell-TIS}
    Let $S$ be a finite set and let $(W(s))_{s\in S}$ be a centered Gaussian process on $S$.
    Let $\sigma^2 \coloneqq \max_{s\in S} \Var(W(s))$.
    By the Gaussian Poincar\'{e} inequality (e.g. \cite[Theorem 3.20]{BLMConcentrationInequalitiesNonasymptotic2013}), we have
    \begin{align*}
        \Var\left(\max_{s\in S} W(s)\right)
        \le \sigma^2.
    \end{align*}
    Moreover, by the Borell--TIS inequality (e.g. \cite[Theorem 2.1.1]{ATRandomFieldsGeometry2007}), we have that for all $t\ge 0$,
    \begin{align*}
        \P\left(\left|\frac{\max_{s\in S} W(s) - \E[\max_{s\in S} W(s)]}{\sigma}\right| > t\right)
        \le 2e^{-t^2/2}.
    \end{align*}
\end{proposition}

We proceed to the proof of Theorem \ref{thm:conc-brw}.

\begin{proof}[Proof of Theorem \ref{thm:conc-brw}]
    First we prove the upper bounds in \eqref{eq:brw-variance}, \eqref{eq:brw-fluctuations}.
    Let $\Pi_n$ be the set of up-right directed paths $\pi$ from $(0,0)$ to $(n-1,n-1)$, and define
    \begin{align*}
        W_n(\pi) \coloneqq  \sum_{v\in\pi} Y_n(v)
        \quad\text{for } \pi\in\Pi_n,
    \end{align*}
    where $Y_n$ denotes the BRW on $\dmp$.
    Thus $(W_n(\pi))_{\pi\in\Pi_n}$ is a centered Gaussian process on $\Pi_n$, the maximum of which is the last passage time $\LBRW$.
    We will upper bound the fluctuations of $\LBRW$ using 
    Proposition \ref{borell-TIS}.

    Towards applying Proposition \ref{borell-TIS},
    we claim that $\max_{\pi\in\Pi_n}\Var(W_n(\pi))\ls n^2$.
    To see this, fix $\pi\in\Pi_n$ and observe that by directedness, for every $k\in\lb 0,\log n\rb$ and $B\in\BB^{(k)}$, we have that $\left|B\cap\pi\right|\ls \frac{n}{2^k}$.
    Moreover, $\left|B\cap\pi\right|\ne 0$ for at most $O(2^k)$ many $B\in\BB^{(k)}$.
    Therefore, by the independence of the Gaussian variables across boxes, we get that
    \begin{align*}
        \max_{\pi\in\Pi_n}\Var(W_n(\pi))
        =
        \max_{\pi\in\Pi_n}\Var\left(\sum_{k=0}^{\log n} \sum_{B\in\BB^{(k)}} \left|B\cap\pi\right|\xi_B\right)
        &= \max_{\pi\in\Pi_n}
        \sum_{k=0}^{\log n} \sum_{B\in\BB^{(k)}}\left|B\cap\pi\right|^2\\
        &\ls \sum_{k=0}^{\log n} 2^k \cdot \frac{n^2}{2^{2k}}\\
        &\ls n^2
    \end{align*}
    as claimed.
    As a consequence of this pathwise $O(n^2)$ variance bound, the upper bounds in \eqref{eq:brw-variance} and \eqref{eq:brw-fluctuations} now follow from Proposition \ref{borell-TIS}.

    The lower bounds are straightforward since by definition the coarsest scale affects all directed paths by a constant shift in their weights of order $n$.
    To see the details, recall that $\BB^{(0)}$ consists of just a single box, namely $\dmp$.
    Write $\xi_0\coloneqq\xi_{\dmp}$.
    Then we can rewrite the last passage time $L=\LBRW$ as
    \begin{align}\label{1001}
        L = (2n-1)\xi_0
        + \underbrace{\max_{\pi\in\Pi_n} \sum_{k=1}^{\log n} \sum_{B\in\BB^{(k)}}
        \left|B\cap\pi\right| \xi_B}_{\wh{L}},
    \end{align}
    where we used that every directed path $\pi\in\Pi_n$ contains exactly $2n-1$ many vertices,
    and where $\xi_0$ and $\wh{L}$ are independent.
    Independence implies that
    \[
        \Var(L) = \Var((2n-1)\xi_0) + \Var(\wh{L})
        = (2n-1)^2 + \Var(\wh{L}),
    \]
    which implies the variance lower bound in \eqref{eq:brw-variance}.
    Moreover, since $\Var(L)\ls n^2$, it follows that $\Var(\wh{L})\ls n^2$.
    Note that $\E[L]=\E[\wh{L}]$, since $\xi_0$ has mean zero.
    Therefore, by Chebyshev's inequality,
    \[
        \P\left(\left|\frac{\wh{L}-\E[L]}{n}\right|\le C\right) 
        \ge \frac12
        \qquad\text{for some } C>0.
    \]
    Combining this with \eqref{1001} and again using that $\xi_0$ and $\wh{L}$ are
    independent, we obtain
    \begin{align*}
        \P\left(\left|\frac{L-\E[L]}{n}\right|>t\right)
        &\ge \frac{1}{2}\cdot \P\bigl(\xi_0 > C+t\bigr).
    \end{align*}
    Since $\xi_0\sim N(0,1)$, the lower bound asserted in \eqref{eq:brw-fluctuations} follows.
    This proves Theorem \ref{thm:conc-brw}.
\end{proof}

\subsection*{Last passage time lower bound}

We turn now to proving Theorem \ref{thm:brw}, which recall asserts that
\begin{align*}
    \P\left(\LBRW \ge c\,\frac{n(\log n)^{1/2}}{\log\log n}\right)
    \ge 1-e^{-c'\frac{\log n}{(\log\log n)^{2}}}
\end{align*}
and that $\E[\LBRW]\gs \frac{n(\log n)^{1/2}}{\log\log n}$.
We will modify the multi-scale construction underlying Theorem \ref{thm:main} to build a path $\pi$ with the property that 
for many scales $k\in\{0,\dots,\log  n\}$, the sum $\sum_{B\in\BB^{(k)}}\left|B\cap \pi\right|\xi_B$ has a large positive mean, implying the lower bound for $\E[\LBRW]$.
The concentration established in Theorem \ref{thm:conc-brw} will then yield the probability bound.

The construction involves considering a tree of possible path trajectories, see Figure \ref{fig:brw}. 
As alluded to earlier, for reasons that will be clear shortly, we will only consider scales at a certain separation $\sss$.
The $\ell\textsuperscript{th}$ level of the tree corresponds to the trajectory or the skeleton of the path $\pi$ at scale $\ell\sss$. Given the latter, we decide the skeleton at the $(\ell+1)\textsuperscript{th}$ level in the following way. Inside each box of size $\frac{n}{2^{\ell\sss}}$ intersecting the trajectory decided so far,  
we construct two \emph{disjoint} sequences of boxes $B\in\BB^{(\ell+1)\sss}$, 
each encoding a possible trajectory (at resolution $\smash{\frac{n}{2^{(\ell+1)\sss}}}$) for the path $\pi$.
The trajectory maximizing the sum $\sum_{B\in\BB^{(\ell+1)\sss}}\left|B\cap \pi\right|\xi_B$ is chosen to be the true trajectory of $\pi$. 
Note that the contribution to the overall weight of the path at this scale by any one of the candidate trajectories is simply a sum of i.i.d. mean zero Gaussians (one for each box $B$) and hence has overall mean zero.
Taking the best of the two trajectories is what gives us a non-trivial mean.
However, to make the trajectories pass through disjoint randomness we will need to sacrifice some control over the slope of the path at each scale, more than what we had to in the 
heavy-tailed $\al=2$ setting.
This is manifested in the fact that, unlike the $\al=2$ case where the scale separation $\sss$ was simply to ensure concentration properties lending flexibility in its choice, in the BRW case choosing an optimal separation is crucial to get the best possible {exponent for the logarithmic corrections in the lower bound. The exponent} nonetheless turns out to be smaller (namely $\frac12$) than was obtained in Theorem \ref{thm:main}.
We start by redefining some of the quantities from \eqref{eq:recalled}.
Let  
\begin{equation*}
    \sss  \coloneqq  \log \log n \qquad\text{and}\qquad \MM\coloneqq \frac{\log n}{\sss}.
\end{equation*}
We emphasize that choosing $\sss=c  \log \log n$ for some  $c\ne 1$ will be suboptimal.
Here it is important to recall that all logarithms have base $2$.
For notational simplicity we will assume that $\sss,\MM\in\N$.

We now describe the basic inductive construction, which will be a recipe for constructing a skeleton of a path from the bottom-left corner to the top-right corner of a square of size $\frac{n}{2^{\ell\sss}}\times \frac{n}{2^{\ell\sss}}$, for any $\ell\in\lb 0, \MM-1\rb$.
This is analogous to the construction given in Section \ref{sec:multiscale} but with some key differences.
To avoid repetition we will mostly focus on pointing out the differences. 
The initial analysis will be in terms of the mean which will involve a counterpart to Lemma \ref{lem:cancellations} and the discussion thereafter.
Then having constructed a path whose weight sum has large mean, we will deduce the result via the concentration bound of Theorem \ref{thm:conc-brw}, which plays a similar role as Proposition \ref{prop:largefrac}.

\begin{figure}[ht]
    \centering
    \begin{subfigure}[t]{0.45\textwidth}
        \includegraphics[width=\linewidth]{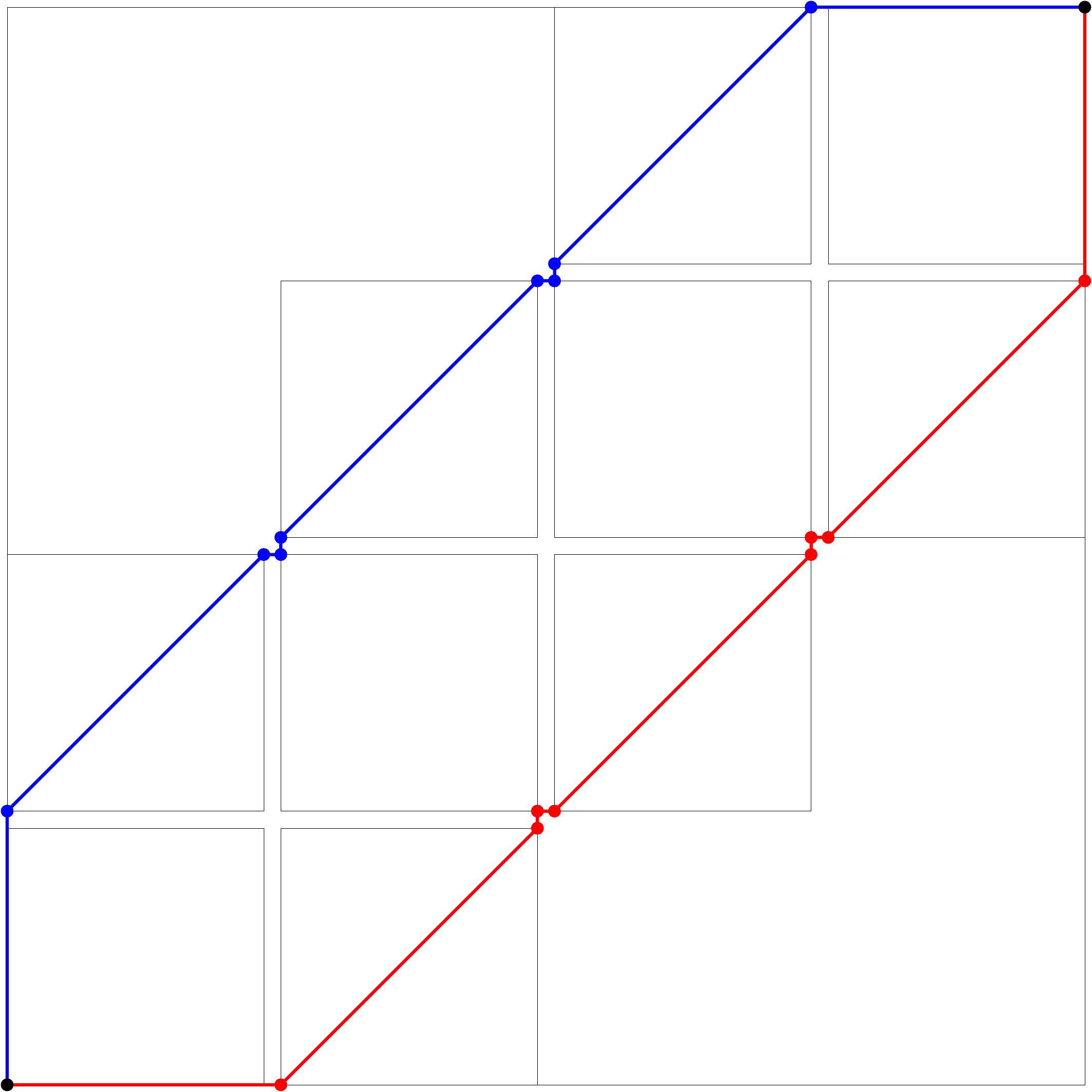}
    \end{subfigure}
    \hfill
    \begin{subfigure}[t]{0.45\textwidth}
        \includegraphics[width=\linewidth]{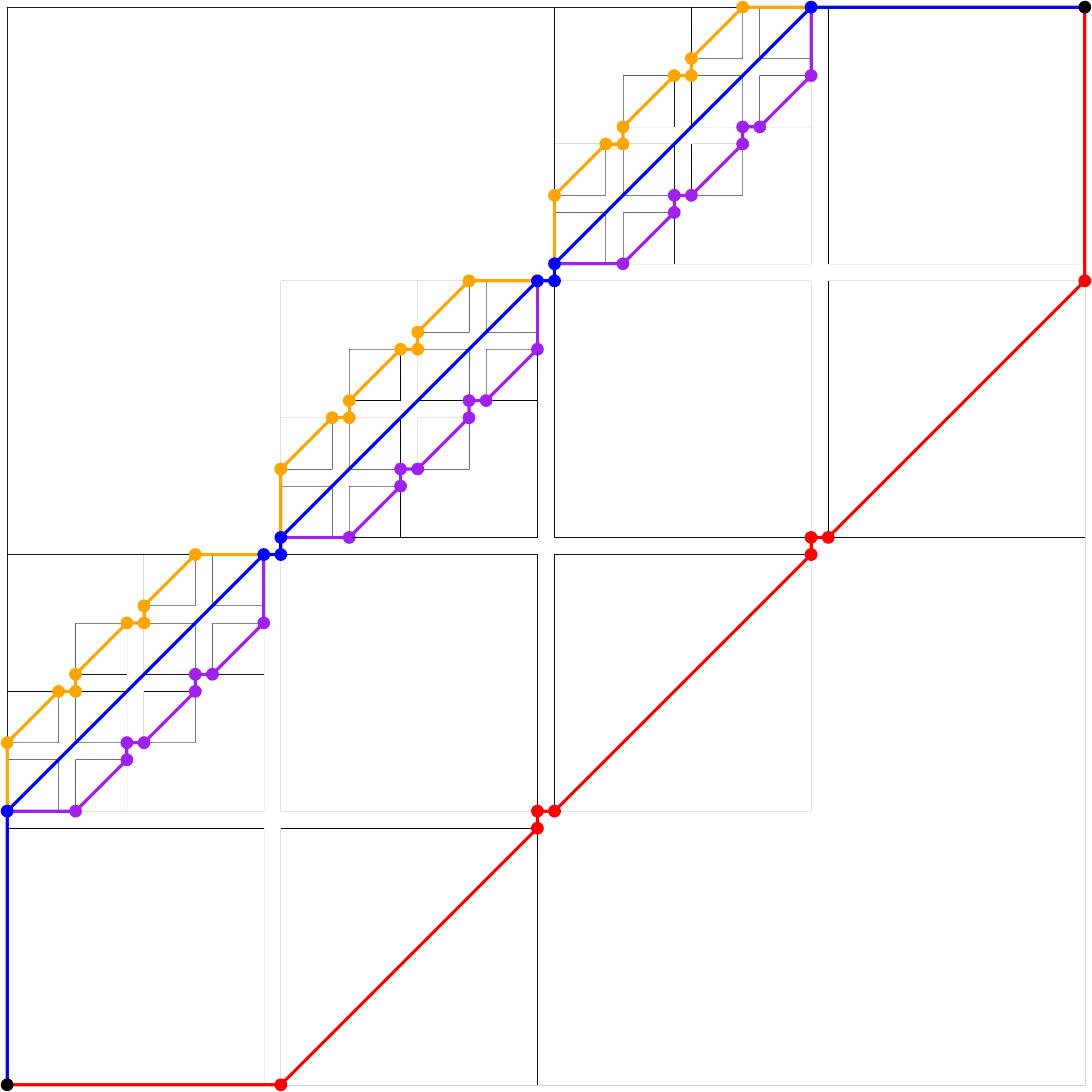}
    \end{subfigure}
    \caption{The multi-scale construction underlying Theorem \ref{thm:brw}.\\
    \textbf{Left:}
    The large black square is $\lb0,\frac{n}{2^{\ell\sss}}-1\rb^2$ which has side length $\frac{n}{2^{\ell\sss}}$,
    and the smaller squares inside have side length
    $\frac{n}{2^{(\ell+1)\sss}}$ (as drawn, $2^{\sss}=4$ and $\frac{n}{2^{\ell\sss}} = 64$).
    Here ``side length" refers to the number of lattice points on each side.
    The up-skeleton consists of the two black points and the
    eight blue points, and the down-skeleton consists of the two black
    points and the eight red points.
    The blue broken line is a schematic drawing of a path
    containing the up-skeleton, and similarly for the red broken line.
    \emph{Any} directed path containing either skeleton takes exactly
    $\frac{n}{2^{(\ell+1)\sss}}$ many steps in the bottom-left and top-right squares,
    and exactly $1$ step in the other two squares intersecting the diagonal $y=x$.
    Moreover, \emph{any} directed path containing the up-skeleton takes
    exactly $2\frac{n}{2^{(\ell+1)\sss}}-1$ many steps in each of the three
    squares above the diagonal, and similarly for the down-skeleton and
    the three squares below the diagonal.
    We select the skeleton for which the sum of BRW variables at
    scale $(\ell+1)\sss$ is larger.
    Since the two triplets of off-diagonal squares are disjoint, the corresponding BRW variables $\xi_B$ are independent of each other, and thus the expected maximum weight sum is of the same order as the fluctuations of the sum of BRW variables along each individual skeleton (see \eqref{909}).
    Note that the slope between the bottom-left and top-right corner of each off-diagonal square is $1$.
    Moreover, the slope between all other consecutive points on the skeletons is $0$ or $\infty$, which by directedness completely determines the trajectory of the path between them.
    Writing $m^{(\ell+1)}$ for the number of squares of side length
    $\frac{n}{2^{(\ell+1)\sss}}$ in which the path has slope $1$,
    we have that $m^{(\ell+1)} = (2^\sss - 1)m^{(\ell)}$.\\
    \textbf{Right:}
    Depicted here is the situation where the up-skeleton is the maximizer in the
    left figure.
    We repeat the same construction as in the left figure in each of the squares
    of side length $\frac{n}{2^{(\ell+1)\sss}}$ whose bottom-left and top-right 
    corners are given by consecutive points in the maximizing skeleton.
    That is, in each such square we consider an up-skeleton (orange) and
    a down-skeleton (purple), now at scale $(\ell+2)\sss$, and we choose
    the skeleton which maximizes the corresponding sum of BRW variables
    at scale $(\ell+2)\sss$.
    }\label{fig:brw}
\end{figure}

\subsection*{Step 1: multi-scale construction}
Fix $\ell\in \lb 0, \MM-1\rb$ and write $s\coloneqq\frac{n}{2^{(\ell+1)\sss}}$.
Consider the square $R \coloneqq \lb 0,\frac{n}{2^{\ell\sss}}-1\rb^2$
and the following sequence of ordered points in it:
\begin{align*}
    (0,0)
    \preceq 
    (0,s)
    &\prec 
    (s-1, 2s-1)
    \preceq 
    (s, 2s-1)
    \preceq
    (s,2s)\\
    &\prec 
    (2s-1, 3s-1)
    \preceq
    (2s, 3s-1)
    \preceq
    (2s,3s)\\
    &\prec\\
    &\;\;\vdots\\
    &\prec
    ((2^{\sss}-1)s-1, 2^{\sss}s-1)
    \preceq (2^{\sss}s-1, 2^{\sss}s-1)
    =\left(\frac{n}{2^{\ell\sss}}-1, \frac{n}{2^{\ell\sss}}-1\right).
\end{align*}
We call this collection of points the \emph{up-skeleton}, 
and call its reflection across the diagonal $y=x$ the \emph{down-skeleton}.
See the left side of Figure \ref{fig:brw} for an illustration of the skeletons.

Consider any up-right directed path $\pi$ from $(0,0)$ to $(\frac{n}{2^{\ell\sss}}-1,\frac{n}{2^{\ell\sss}}-1)$ containing one of the two skeletons.
We analyze the weight collected by $\pi$ from the BRW at scale $(\ell+1)\sss$ within $R$, which is given by the sum
\[
    \sum_{\substack{B\in\BB^{((\ell+1)\sss)}\\B\subset R}}\left|B\cap \pi\right|\xi_B.
\]
By the definition of the skeletons, $\pi$ passes through exactly $2^{\sss+1}-1$ many distinct squares $B\in\BB^{((\ell+1)\sss)}$ during its journey through $R$.
We call these squares \emph{diagonal} or \emph{off-diagonal} according to whether they intersect the diagonal line $y=x$.
Notice that no matter which skeleton $\pi$ contains, it passes through every diagonal square, of which there are $2^\sss$ many.
Moreover, by directedness $\pi$ takes exactly $\frac{n}{2^{(\ell+1)\sss}}$ many steps inside the bottom-left square and inside the top-right square, and exactly $1$ step inside all the other diagonal squares.
Also, the other $2^\sss-1$ many squares through which $\pi$ passes are off-diagonal, and by directedness $\pi$ takes exactly $2\frac{n}{2^{(\ell+1)\sss}}-1$ many steps in each of them.
Note that since the latter squares are off-diagonal, they are disjoint from their reflections across the line $y=x$.
We label the diagonal squares in increasing order as $B_1\prec B_2\prec\cdots\prec B_{2^{\sss}}$, so $B_1$ is the bottom-left square and $B_{2^{\sss}}$ is the top-right square.
We label the off-diagonal squares along $\pi$ as $B_1^\Box,\dots,B_{2^{\sss}-1}^\Box$, where $\Box\in\{\up,\down\}$ according to which skeleton $\pi$ contains.
From all this discussion we obtain the decomposition 
\begin{align}\label{908}
    \sum_{\substack{B\in\BB^{((\ell+1)\sss)}\\B\subset R}}\left|B\cap \pi\right|\xi_B
    &= \underbrace{\frac{n}{2^{(\ell+1)\sss}}\bigl(\xi_{B_1} + \xi_{B_{2^\sss}}\bigr) + \sum_{i=2}^{2^{\sss}-1}\xi_{B_i}}_{W}
    + \underbrace{\left(2\frac{n}{2^{(\ell+1)\sss}}-1\right)\sum_{j=1}^{2^{\sss}-1}\xi_{B_j^\Box}}_{Z^\Box}.
\end{align}
Observe that $W,Z^{\up},Z^{\down}$ are independent Gaussian random variables with mean zero.
Moreover,
\begin{align}\label{907}
    \Var(Z^\Box)
    = \left(2\frac{n}{2^{(\ell+1)\sss}}-1\right)^2 (2^{\sss}-1)
    \gs \frac{n^2}{2^{(2\ell+1)\sss}}.
\end{align}
It follows that the expected maximum weight from the scale $(\ell+1)\sss$ collected in $R$ by any path containing one of the two skeletons is given by
\begin{align}\label{909}
    \E\bigl[\max\{W+Z^{\up},\, W + Z^{\down}\}\bigr]
    = \E\bigl[\max\{Z^{\up},\, Z^{\down}\}\bigr]
    \gs 
    \frac{n}{2^{(\ell+\frac12)\sss}},
\end{align}
where we used the fact that the expected maximum of two i.i.d. $N(0,\sigma^2)$ random variables is lower bounded by $c\sigma$ for some universal constant $c>0$.

\subsection*{Step 1$'$: iteration}
We now iterate the construction just described.
Let $\cF^{(0)}$ be the trivial $\sigma$-algebra, and for $\ell\ge 1$ let $\cF^{(\ell)}$ be the $\sigma$-algebra generated by the BRW at scales $\sss,2\sss,\dots,\ell\sss$, i.e. by the random variables $\xi_B$ for $B\in\bigcup_{\ell'=1}^\ell \BB^{(\ell'\ss)}$.
For the base case $\ell=0$ we set $\VV^{(0)}\coloneqq \left\{(0,0), (n-1,n-1)\right\}$.
Suppose that for some $\ell\in\lb 0,\MM-1\rb$ 
we have specified the skeletons up to scale $\ell \sss$ 
and in particular let 
\[
    \VV^{(\ell)} = \left\{\vv^{(\ell)}_1\preceq\cdots\preceq \vv^{(\ell)}_{|\VV^{(\ell)}|}\right\}
\]
be an $\cF^{(\ell)}$-measurable set of vertices such that
for each consecutive pair $\vv^{(\ell)}_i\preceq \vv^{(\ell)}_{i+1}$, one of the following holds:
\footnote{By $\Slope(a,b)=0$ we mean that the line segment joining $a$ to $b$ is horizontal, and by $\Slope(a,b)=\infty$ we mean that the same line segment is vertical.}
\begin{itemize}
    \item $\Slope\bigl(\vv^{(\ell)}_i, \vv^{(\ell)}_{i+1}\bigr) \in \{0,\infty\}$ \quad and \quad $\norm{\vv^{(\ell)}_i - \vv^{(\ell)}_{i+1}}_1\in\left\{1, \frac{n}{2^{\ell\sss}}\right\}$,\qquad or
    \item $\Slope\bigl(\vv^{(\ell)}_i, \vv^{(\ell)}_{i+1}\bigr) = 1$ \quad and\quad $\norm{\vv^{(\ell)}_i - \vv^{(\ell)}_{i+1}}_1 = 2(\frac{n}{2^{\ell\sss}}-1)$.
\end{itemize}
In other words, each $\Rect\bigl(\vv^{(\ell)}_i, \vv^{(\ell)}_{i+1}\bigr)$ is either a degenerate rectangle (i.e. a horizontal or vertical line) of length $1$ or $\frac{n}{2^{\ell\sss}}$,
or a square 
with $\frac{n}{2^{\ell\ss}}$ many lattice points on each side.
Note that  $\VV^{(0)}\coloneqq\{(0,0), (n-1,n-1)\}$ indeed satisfies the above properties.

Conditional on $\cF^{(\ell)}$, we apply the construction from the previous subsection in each rectangle $\Rect\bigl(\vv^{(\ell)}_i, \vv^{(\ell)}_{i+1}\bigr)$ with slope $1$ 
(see the left side of Figure \ref{fig:brw}, where the black points correspond to $\vv^{(\ell)}_i$ and $\vv^{(\ell)}_{i+1}$).
For each such rectangle, we add the points in the maximizing skeleton to $\VV^{(\ell)}$, thus obtaining an $\cF^{(\ell+1)}$-measurable vertex set $\VV^{(\ell+1)}$ with the same properties as in the above bullet points with $\ell$ replaced by $\ell+1$ (see Figure \ref{fig:brw}).
This completes the induction step.

Having selected skeletons at every scale $0,\sss,\dots,\MM\sss$, 
we choose an up-right path $\pi$ containing all the skeletons in some arbitrary $\cF^{(\MM)}$-measurable way.
This measurability assumption is just for convenience: it implies that $\pi$ is independent of the BRW at all scales  
$k\not\in\{\sss,2\sss,\dots,\MM\sss\}$,
which will simplify the later analysis. For instance, we can take $\pi$ to be the leftmost path passing through all the points specified by the skeletons.

\subsection*{Step 2: counting rectangles}
Consider $\pi$ as above.
Let $m^{(\ell)}$ be the number of rectangles of slope $1$ demarcated by vertices in $\VV^{(\ell)}$.
Then \eqref{909}  shows that the expected contribution made by the BRW at scale $(\ell+1)\sss$ to the sum of weights along $\pi$ is at least
\begin{align}\label{900}
    m^{(\ell)}\cdot\frac{n}{2^{\ell\sss+\sss/2}},
\end{align}
up to a multiplicative constant.
We can evaluate $m^{(\ell)}$ using the recursion
\[
    m^{(\ell)} = \bigl(2^{\sss}-1\bigr) m^{(\ell-1)},
\]
where the factor $2^\sss-1 = (2^{\sss+1}-1)-2^\sss$ accounts for the squares of size $\frac{n}{2^{\ell\sss}}\times \frac{n}{2^{\ell\sss}}$ 
in which the path takes just one step or has slope $0$ or $\infty$
(see Figure \ref{fig:brw}).
Since $m^{(0)}=1$,
we get that
\begin{equation}\label{901}
    \begin{split}
        m^{(\ell)} = \bigl(2^{\sss}-1\bigr)^{\ell}
        &= 2^{\ell\sss}\bigl(1-2^{-\sss}\bigr)^{\ell}\\
        &\ge 2^{\ell \sss}\,2^{-\frac{10\ell}{2^{\sss}}},
    \end{split}
\end{equation}
where we used that $\sss\ge 1$ and the crude bound $(1-\frac{1}{x})^x\ge 2^{-10}$ when $x\ge 2$.
This yields the following lower bound for \eqref{900}:
\begin{align*}
    m^{(\ell)}\cdot\frac{n}{2^{\ell\sss+\sss/2}}
    &\ge \frac{n}{2^{\sss/2}}\,2^{-\frac{10\ell}{2^{\sss}}}.
\end{align*}
Summing over 
$\ell\in\lb 0, \MM-1\rb$
and using that $1-2^{-10x} \ls x$ for $x>0$
now yields
\begin{align*}
    \sum_{\ell=0}^{\MM-1} m^{(\ell)}\cdot \frac{n}{2^{\ell\sss+\sss/2}}
    \ge \sum_{\ell=0}^{\MM-1} \frac{n}{2^{\sss/2}}\,2^{-\frac{10\ell}{2^{\sss}}}
    &= \frac{n}{2^{\sss/2}} \sum_{\ell=0}^{\frac{\log n}{\sss}-1} 2^{-\frac{10\ell}{2^{\sss}}}\\
    &=\frac{n}{2^{\sss/2}}\cdot \frac{1-2^{-\frac{10}{2^{\sss}}\frac{\log n}{\sss 2^{\sss}}}}{1-2^{-\frac{10}{2^{\sss}}}}\\
    &\gs \frac{n}{2^{\sss/2}}\cdot \frac{1-2^{-\frac{10}{2^{\sss}}\frac{\log n}{\sss}}}{\frac{1}{2^{\sss}}}\\
    &\gs \frac{n}{2^{\sss/2}} \cdot \min \left\{{2^\sss}, \frac{\log n}{\sss}\right\}.
\end{align*}
This is where the need to choose $\sss$ optimally becomes apparent. A quick computation reveals that, up to lower order terms which we ignore, $\sss=\log\log n$ is indeed the optimal choice (recall that all logarithms have base $2$ by convention).
For this choice of $\sss$, the above is at least  $c\,\frac{n(\log n)^{1/2}}{\log\log n}$ for some constant $c>0$.

\subsection*{Step 3: concluding the proof}
Let $\pi$ be the path constructed above.
We decompose its weight sum $\sum_{v\in\pi}Y(v)$ as:
\[
    \sum_{v\in\pi}Y(v)
    =
    \underbrace{\sum_{\ell=1}^{\MM}\sum_{B\in\BB^{(\ell\sss)}}\left|B\cap\pi\right| \xi_B}_{L_1}
    \quad+\quad
    \underbrace{
        \sum_{k\not\in\{\sss,2\sss,\dots,\MM\sss\}}\sum_{B\in\BB^{(k)}}\left|B\cap\pi\right| \xi_B
    }_{L_2}.  
\]
We argued above that there exists $c>0$ such that for all sufficiently large $n$,
\begin{align*}
    \E[L_1] \ge c\,\frac{n(\log n)^{1/2}}{\log\log n}.
\end{align*}
Also, since $\pi$ is $\cF^{(\MM)}$-measurable, it follows that $L_2$ is a mean
zero Gaussian random variable, and in particular $\E[L_1+L_2]=\E[L_1]$.
Since $\LBRW \ge L_1+L_2$, we deduce that for all sufficiently large $n$,
\begin{align*}
    \E[\LBRW] \ge c\,\frac{n(\log n)^{1/2}}{\log\log n},
\end{align*}
which is \eqref{eq:brw-mean}.
It follows by Theorem \ref{thm:conc-brw} (viz. \eqref{eq:brw-fluctuations}) that for all sufficiently large $n$,
\begin{align*}
    \P\left(
        \LBRW \ge \frac{c}{2}\,\frac{n(\log n)^{1/2}}{\log\log n}
    \right)
    \ge 1-e^{-c'\frac{\log n}{(\log\log n)^2}},
\end{align*}
which proves \eqref{eq:brw-hp}.
This completes the proof of Theorem \ref{thm:brw}.\qed
\\

In this final section we discuss extensions of our results to LPP in hierarchical environments in higher dimensions.
\section{Higher dimensions}\label{sec:higher-dimensions}

We will consider LPP on $\Z^{d+1}$ for $d\ge 1$: for vertex weights $(X(v))_{v\in\Z^{d+1}}$, the last passage time from $(0,\dots,0)$ to $(n,\dots,n)$ is defined as
\[
    L_n  \coloneqq  \max_\pi \sum_{v\in\pi}X(v),
\]
where the maximum is over directed paths in $\Z^{d+1}$ starting at $(0,\dots,0)$ and ending at $(n,\dots,n)$.
Note that in all the upcoming results the various constants depend implicitly on $d$.

\subsection*{Critical heavy-tailed LPP}
Consider LPP on $\Z^{d+1}$ with i.i.d. weights whose common distribution is non-negative and satisfies
\[
    \P(X>t)\sim Ct^{-\al}
\]
for some $\al>0$.
As discussed in the introduction, Hambly and Martin
\cite{HMHeavyTailsLastpassage2007} analyzed this model with $d=1$ and $\al\in(0,2)$.
Their techniques and results admit natural analogues for any $d\ge1$ and $\al\in(0,d+1)$, see \cite[Section 8]{HMHeavyTailsLastpassage2007}.
The critical value $\al=d+1$ is determined by the same heuristic ``Flory argument'' as in the case $d=1$ presented in the introduction.
Briefly, it is expected that at criticality the geodesic is fully delocalized and the KPZ scaling relation holds, up to possible logarithmic corrections. 
The largest weight in $\lb 0,n\rb^{d+1}$ is of order $n^{\frac{d+1}{\al}}$ which according to the Flory heuristic dictates fluctuations, i.e. $\chi=\frac{d+1}{\al}$.
Since the KPZ scaling relation also holds it follows that $\chi=1$ and hence $\al=d+1$.

Assuming $\al=d+1$, a straightforward modification of our techniques yields the following $(d+1)$-dimensional analogue of Theorem \ref{thm:main}.
\begin{theorem}[Logarithmic correction lower bound for $\al=d+1$]\label{thm:main-d}
    There exists $c>0$ such that for all sufficiently large $n$,
    \[
        \P\left(
            L_n \ge c\,\frac{n(\log n)^{\frac{d+2}{2(d+1)}}}{\log\log n}
        \right)
        \ge 1-e^{-(\log n)^{97}}.
    \]
\end{theorem}

The exponent $\frac{d+2}{2(d+1)}$ in Theorem \ref{thm:main-d} is determined by the obvious 
$(d+1)$-dimensional analogues of Lemma \ref{lem:apriori}\ref{apriori_slope},
Proposition \ref{prop:largefrac}, and Lemma \ref{lem:cancellations}.
In particular, recalling Figure \ref{fig:iop}, the analogous cylinders have bases with side length ${(\log n)^{-\frac{1}{2(d+1)}}}$ and height $(\log n)^{\frac{d}{2(d+1)}}$.

The next result whose proof is the same as that of Theorem \ref{thm:conc} asserts that the critical last passage time is concentrated in all dimensions.
\begin{theorem}[Concentration for $\al=d+1$]\label{thm:conc-d}
    There exist $C,t_0>0$ such that for all $t\ge t_0$ and all $n\ge1$,
    \[
        \P\left(\left|\frac{L_n-\E[L_n]}{n}\right| > t\right) \le \frac{C}{t^{4/3}}.
    \]
    Note that by Theorem \ref{thm:main-d}, there exists $c>0$ such that
    \[
        \E[L_n] \ge c\,\frac{n(\log n)^{\frac{d+2}{2(d+1)}}}{\log\log n}
    \]
    for all sufficiently large $n$ and hence the fluctuations of $L_n$ around its mean are of strictly smaller order than its mean. 
\end{theorem}

\subsection*{LPP on the branching random walk}
We turn next to the $(d+1)$-dimensional analogues of our results for LPP on BRW (Theorems \ref{thm:brw} and \ref{thm:conc-brw}).
The BRW on $\Z^{d+1}$ is defined analogously to the case $d=1$: for each $k\in\lb0,\log_2n\rb$ we partition $\lb 0,n-1\rb^{d+1}$ into $2^k\times\cdots\times 2^k$ boxes and associate to each box an independent standard Gaussian random variable, and define the BRW to be the field whose value at a vertex $v$ is the sum of the Gaussian variables corresponding to the boxes containing $v$. 
We denote by $\LBRW$ the corresponding last passage time.
The proofs of Theorems \ref{thm:brw} and \ref{thm:conc-brw} (given in Section \ref{sec:brw}) go through essentially verbatim for any $d\ge1$, yielding identical conclusions:
with high probability,
\[
    \LBRW \ge c\,\frac{n(\log n)^{1/2}}{\log\log n}.
\]
Moreover, $\LBRW$ is concentrated around its mean with sub-Gaussian fluctuations of order $\Theta(n)$, and in fact $\Var(\LBRW)=\Theta(n^2)$.

\begin{remark}
    It is interesting to note that for all $d\ge1$ the exponent $\frac{d+2}{2(d+1)}$ appearing in Theorem \ref{thm:main-d} is strictly larger than its BRW counterpart $\frac12$, and moreover as $d\to\infty$ we have $\frac{d+2}{2(d+1)}\downarrow \frac12$ (cf. the discussion in Section \ref{iop} on the discrepancy between the lower bounds for $d=1$).
\end{remark}

\begin{remark}\label{rem:dgz}
    As mentioned earlier, Ding, Gwynne, and Zhuang \cite{DGZPercolationThickPoints2026} 
    recently studied percolation properties of the set of thick points of a log-correlated field in high dimensions.
    As the reader may already realize, such results are closely related to the subject of this article.
    We elaborate on this further below while referring the reader to \cite{DGZPercolationThickPoints2026} for more details and other interesting results.

    For the BRW, the results of \cite{DGZPercolationThickPoints2026} imply that for all sufficiently large $d$, as $n\to\infty$ it holds with probability $1-o(1)$ that there exists a path $P$  from $(0,\dots,0)$ to 
    $\{v\in\Z^{d+1}:\|v\|_1=n\}$
    such that all but a vanishing fraction of the BRW weights along $P$ have values of order $\log n$. Crucially, the path $P$ is not directed.
    Note that the existence of a directed path with the same properties as $P$ would imply that $\LBRW\gs n\log n$ in sufficiently high dimensions, matching the trivial upper bound $\LBRW \ls n\log n$
    (recall that the latter follows from the fact that the maximum of the BRW is $O(\log n)$ in every dimension).
    However, the set of vertices 
    with weights of order $\log n$
    is known to exhibit a fractal structure, 
    which heuristically suggests that any path contained in that set must also be fractal and hence should dramatically fail to be directed.  
    The following result to this effect was proved by Ding and Zhang \cite[Theorem 1.2]{DZLiouvilleFirstPassage2019} in the setting of the two-dimensional GFF:
    for any $\e>0$, it holds with high probability that for any path with $O(n)$ many vertices, at least a constant fraction of the weights along it are upper bounded by $\e\log n$.
    Although \cite[Theorem 1.2]{DZLiouvilleFirstPassage2019} is formulated only for the two-dimensional GFF (primarily due to the implications for two-dimensional Liouville quantum gravity), its proof does not rely on planarity and so a similar result is expected for log-correlated fields in every dimension.
    A similar statement valid in every dimension was proved by \cite{BJJ+FractalPercolationUnrectifiable2021} in the setting of fractal percolation, a model of a random fractal set that is a natural proxy for the thick point set of a log-correlated field
    (see also \cite{ChaAbsenceDirectedFractal1995,ChaLengthShortestCrossing1996,CPPNoDirectedFractal1997,OrzLowerBoundBoxcounting1998} for more qualitative results in $d=1$ relying on planarity).
    In forthcoming joint work of the first two authors with Kaihao Jing \cite{GGJ}, we obtain strong quantitative bounds to this effect in a related discrete model.

    The aforementioned path construction of \cite{DGZPercolationThickPoints2026} is in spirit quite similar to ours (Section \ref{sec:brw}), aside from the lack of directedness.
    In view of this and the above discussion, finding the true growth rate of $\LBRW$ and the dependence of the latter on the underlying dimension remains an intriguing open problem that we leave to future work.
\end{remark}

\subsection*{A distribution with finite $(d+1)\textsuperscript{th}$ moment and superlinear last passage time}

We close by discussing the $(d+1)$-dimensional analogue of
Theorem \ref{thm:secondmoment} and Proposition \ref{prop:real-secondmoment}.
For LPP on $\Z^{d+1}$ with i.i.d. non-negative weights with common distribution $X$, Martin \cite{MarLinearGrowthGreedy2002,MarLimitingShapeDirected2004} proved that the condition
\begin{align*}
    \int_0^\infty \P(X>t)^{\frac{1}{d+1}}\,\mrm{d}t < \infty
\end{align*}
implies that $g\coloneqq\limsup_{n\to\infty}\frac{L_n}{n}<\infty$.
Note that the above condition is stronger than $\E[X^{d+1}]<\infty$.
Moreover, \cite{CGGKGreedyLatticeAnimals1993} showed that $\E[X^{d+1}]<\infty$
is necessary for $g<\infty$.
The next result, like Theorem \ref{thm:secondmoment} and Proposition \ref{prop:real-secondmoment} and with the same proof, shows that $\E[X^{d+1}]<\infty$ is not sufficient for $g<\infty$.
\begin{theorem}
    Fix $\beta\in(1, \frac{d+2}{2})$ and assume that
    \[
        \P(X>t)\sim Ct^{-(d+1)}(\log t)^{-\beta}.
    \]
    Then $\E[X^{d+1}]<\infty$, and the last passage time satisfies
    \[
        \P\left(
            L_n \ge \frac{n(\log n)^{\frac{d+2}{2(d+1)} - \frac{\beta}{d+1}}}{\log\log n}
        \right)
        \ge 1-e^{-(\log n)^{97}}.
    \]
    In particular,
    \[
        \frac{L_n}{n}\to \infty \quad\text{almost surely}.
    \]
\end{theorem}

%%%%%%%%%%%%%%%%
\emergencystretch=1em
\printbibliography
%%%%%%%%%%%%%%%%

\end{document}